\documentclass[11pt,a4paper]{article}
\input{preamb.sty}

\title{A Bismut-Elworthy inequality for a Wasserstein diffusion on the circle}
\author{Victor Marx
\thanks{Université Côte d’Azur, CNRS, Laboratoire J.A. Dieudonné UMR 7351, France. \newline
Universität Leipzig, Fakultät für Mathematik und Informatik, Augustusplatz 10, 04109 Leipzig, Germany. \newline
E-mail: vm.victor.marx@gmail.com}
}
\date{July 2021}

\begin{document}

\maketitle

\begin{center}
\shadowbox{
\begin{minipage}{14cm}
\large
\textbf{Abstract.}
We introduce in this paper a strategy  to prove  gradient estimates for some infinite-dimensional diffusions on  $L_2$-Wasserstein spaces. For a specific example of a diffusion on the $L_2$-Wasserstein space of the torus, we get a Bismut-Elworthy-Li formula up to a remainder term and deduce  a gradient estimate with a rate of blow-up of order $\mathcal O(t^{-(2+\eps)})$. 
\end{minipage}
}
\end{center}



\section{Introduction}

In~\cite{vrenessesturm09}, von Renesse and Sturm introduced a diffusion process on the $L_2$-Wasserstein space $\mathcal P_2(\R)$, satisfying following properties: the large deviations in small time are given by the Wasserstein distance and  the martingale term which arises when expanding any smooth function~$\phi$ of the measure argument along the process has exactly the square norm of the Wasserstein gradient of~$\phi$ as local quadratic variation. 
There are several examples of diffusions on Wasserstein spaces, either constructed via Dirichlet forms~\cite{vrenessesturm09,andresvonrenesse10, konar_vren_coalescing_fragmentating},  as limits of particle systems~\cite{konarovskyi17system, konarovskyi17behavior} or  as a system of infinitely many particles~\cite{marx18,marx20}. 
Some diffusive properties  of those processes have already been proved, as e.g.\! a large deviation principle~\cite{konarovskyivonrenesse18} or as a restoration of uniqueness for McKean-Vlasov equations~\cite{marx20}. 

We prove in this paper another well-known diffusive property. Indeed, we control 
the gradient of the semi-group associated to a diffusion process on the $L_2$-Wasserstein space.
For finite-dimensional diffusions, this gradient estimate can be obtained from a Bismut-Elworthy-Li integration by parts formula. 
Bismut, Elworthy and Li showed that the gradient of the  semi-group $P_t \phi$ associated to the stochastic differential equation   $
\mathrm d X_t = \sigma(X_t) \mathrm d W_t + b(X_t) \mathrm dt$ on $\R^n$ can be expressed as follows
\begin{align*}
\nabla (P_t \phi)_{x_0} (v_0) = \frac{1}{t} \E{\phi(X_t) \int_0^t \langle V_s, \sigma(X_s) \mathrm dW_s \rangle},
\end{align*}
where $V_s$ is a certain stochastic process starting at $v_0$ (see~\cite{bismut81, elworthy92, elworthyli94}).
In particular, that equality shows that $\|\nabla (P_t \phi)_{x_0} (v_0)\|$ is of order $t^{-1/2}$ for small times. Important domains of application of Bismut-Elworthy-Li formulae are among others geometry~\cite{thalmaier97, thalmaierwang98, arnaudonthalmaierwang06}, non-linear PDEs~\cite{delarue03,Zhang05} or finance~\cite{fournie_et_al_99, malliavin_thalmaier06}. 
Recent interest has emerged for similar results in infinite dimension. First, Bismut-Elworthy-Li formulae were proved for Kolmogorov equations on Hilbert spaces and for reaction-diffusion systems in bounded domains of $\R^n$, see~\cite{daprato_elworthy_zabczyk, cerrai01, daprato04}. Recently, Crisan and McMurray~\cite{crisanmcmurray} and Ba\~{n}os~\cite{banos18} proved Bismut-Elworthy-Li formulae for McKean-Vlasov equations $\mathrm dX_t = b(t, X_t, \mu_t) \mathrm dt + \sigma(t, X_t, \mu_t)  \mathrm dW_t$, with $\mu_t = \mathcal L(X_t)$. 
For other recent smoothing results on McKean-Vlasov equations and mean-field games, see also~\cite{buck_li_peng_rainer, chaudru20, chaudrufrikha18, chassagneux_cr_del}.

\subsection{A gradient estimate for a Wasserstein diffusion on the torus}

In this paper, we construct a system of infinitely many particles moving on the one-dimensional torus $\tor=\mathds S^1$, identified with the interval $[0,2\pi]$. 
Considering for each time the empirical measure associated to that system, we get a diffusion process on~$\mathcal P(\tor)$, space of probability measures on~$\tor$. Then we average out that process over the realizations of an additive noise $\beta$.  This averaging increases the regularity of the process and leads to a gradient estimate of the associated semi-group. 

To state more precisely the main result of the paper, let us introduce the following equation
\begin{align}
\label{intro:sde_beta}
x_t^g(u)= g(u) + \sum_{k \in \Z} f_k \int_0^t \Re ( e^{-ik x_s^g(u)} \mathrm d W_s^k)+\beta_t, \quad t \geq 0,\ u \in [0,1].
\end{align}
Hereabove, $\beta$, $(W^{\Re,k})_{k \in \Z}$ and $(W^{\Im,k})_{k \in \Z}$ are independent standard  real-valued Brownian motions,   the notation $\Re$ denotes the real part of a complex number and $W^k_\cdot:=W^{\Re,k}_\cdot + i W^{\Im,k}_\cdot$. The sequence  $(f_k)_{k \in \Z}$ is fixed and real-valued, typically of the form $f_k= \frac{C}{(1+k^2)^{\alpha/2}}$. 
Lastly, the initial condition $g:[0,1]\to \R$ is a $\mathcal C^1$-function with positive derivative satisfying $g(1)=g(0)+2\pi$, so that $g$ is seen as the quantile function of a probability measure $\nu_0^g$  on  $\tor$. 

For each $u \in [0,1]$, $(x_t^g(u))_{t \in [0,T]}$  represents the trajectory of the stochastic particle starting at $g(u)$, which is  driven both by a common noise $W=(W^k)_k$ and by an idiosyncratic noise~$\beta$, taking over the terminology usually used for McKean-Vlasov equation. This terminology is justified since equation~\eqref{intro:sde_beta} can be seen as the counterpart on the torus $\tor$ of the equation on the real line studied in~\cite{marx20} and defined by
\begin{align*}
y_t^g (u) = g(u) +  \int_\R f(k) \int_0^t \Re( e^{-ik y_s ^g(u) }\mathrm dw(k,s))+\beta_t, \quad t \geq 0,\ u \in \R,
\end{align*}
where $(w(k,t))_{k \in \R, t\in [0,T]}$ is a complex-valued Brownian sheet on $ \R \times [0,T]$. 

Moreover, we will show that the cloud of all particles is spread over the whole torus. More precisely, for each $t \in [0,T]$, the probability measure $\nu_t^g= \operatorname{Leb}_{[0,1]} \circ (x_t^g)^{-1}$ has a density $q_t^g$ w.r.t.\!  Lebesgue measure on $\tor$ such that $q_t^g(x)>0$ for all $x \in \tor$. 
Instead of studying the process $(\nu_t^g)_{t \in [0,T]}$, we consider a more regular process defined by averaging over the realizations of $\beta$:
\begin{align}
\label{intro:def mutg}
\mu_t^g= (\Leb_{[0,1]} \otimes \mathbb P^\beta) \circ (x_t^g)^{-1},
\end{align}
\textit{i.e.\!} $\mu_t^g$ is the probability measure on $\tor$ with density $p_t^g(x):=\mathbb E^\beta \left[ q_t^g (x)\right]$, $x\in \tor$, w.r.t.\! Lebesgue measure\footnote{We assume that $\beta$ and $W$ are defined on  probability spaces $(\Omega^\beta, \mathcal G^\beta,  \mathbb P^\beta)$ and $(\Omega^W, \mathcal G^W,  \mathbb P^W)$, respectively.}.
In other words, since we assume that $\beta$ and $(W^k)_{k \in \Z}$ are independent, $\mu_t^g$ is the conditional law of $x_t^g$ given $(W^k)_{k \in \Z}$. 

The semi-group associated to $(\mu_t^g)_{t \in [0,T]}$ is defined, for any bounded and continuous function $\phi:\mathcal P_2(\tor) \to \R$, by $P_t \phi (\mu_0^g):= \mathbb E \left[ \phi(\mu_t^g) \right]$. Alternatively,  denoting the lift of $\phi$ by 
$\widehat{\phi}(X):= \phi((\Leb_{[0,1]} \otimes \mathbb P^\beta) \circ X^{-1})$ for any random variable $X \in L_2([0,1]\times \Omega^\beta)$, we can also define the semi-group by
$\widehat{P_t \phi}(g) =  \mathbb E \left[  \widehat{\phi}\left( x_t^g \right) \right]$.

The main theorem of this paper states the following upper bound for the Fréchet derivative of $g \mapsto \widehat{P_t \phi}(g)$ depending only on the $L_\infty$-norm of $\phi$.  Assume that $g \in \mathcal C^{3+\theta}$ with $\theta >0$ and let $h$ be a 1-periodic $\mathcal C^1$-function. If $\phi$ is sufficiently regular and if $f_k= \frac{C}{(1+k^2)^{\alpha/2}}$, $k \in \Z$, with $\alpha \in \left( \frac{7}{2},\frac{9}{2} \right)$, then there is $C_g$ independent of $h$ such that
\begin{align}
\label{main_result}
\left| D \widehat{P_t\phi}(g) \cdot h \right| 
=\left|\frac{\mathrm d}{\mathrm d\rho}_{\vert \rho=0} P_t \phi (\mu_0^{g+\rho  h})\right| 
\leq C_g \frac{\| \phi \|_{L_\infty}}{t^{2+\theta}}\| h\|_{\mathcal C^1}, 
\end{align}
for any $t \in (0,T]$. 
The precise assumptions over $\phi$ and the precise statement of this theorem are given below by Definition~\ref{def:phi assumptions} and Theorem~\ref{theo:Bismut pour phi assumptions}, respectively. Moreover,  $C_g$ depends polynomially on $\| g'''\|_{L_\infty}$, $\| g''\|_{L_\infty}$, $\| g'\|_{L_\infty}$ and $\| \frac{1}{g'}\|_{L_\infty}$.

\subsection{Comments on the main result}

The order $\alpha$ of the polynomial decrease of $(f_k)_{k \in \Z}$ has a key role in this paper. In equation~\eqref{intro:sde_beta}, the diffusion coefficient in front of the noise $W$ is written as a Fourier series, $(f_k)_{k \in \Z}$ being the sequence of Fourier coefficients. Therefore, it should not be surprising that the larger $\alpha$ is, the more regular the solution to~\eqref{intro:sde_beta} is\footnote{see Proposition~\ref{prop:differentiability} below.}.
Nevertheless, when we  apply Girsanov's Theorem with respect to $W$, which is part of a standard method introduced by Thalmaier and Wang~\cite{thalmaier97, thalmaierwang98}, we need $\alpha$ to be sufficiently small, in order to be able to invert the Fourier series\footnote{see Lemma~\ref{lemme:decomposition Fourier} below.}. 
So there is a balance regarding the choice of $\alpha$, which explains why we assume in our main result $\alpha$ to be bounded from above and from below. 

Moreover, the question of the order $\alpha$ is highly related to the rate $t^{-(2+\theta)}$ appearing in~\eqref{main_result}. Usually, we expect a rate of $t^{-1/2}$ for diffusions. 
As we have already mentioned, this rate follows directly from a Bismut-Elworthy-Li integration by parts formula. However, adapting the usual strategy based on Kunita's expansion as in~\cite{thalmaier97, thalmaierwang98}, we do not get an \textit{exact} integration by parts formula here. 
Indeed, the failure of the latter strategy in our case comes from the fact that it is impossible to choose $\alpha$ which is simultaneously large enough to ensure a sufficient regularity of the solution and small enough to be able to invert the Fourier series. We refer to remark~\ref{rem:balance} below for a justification of that claim. 

Therefore, the main new strategy introduced in this paper is to regularize the derivative of the solution\footnote{More exactly, we regularize the function $A^g$ defined by~\eqref{def:Agt} into a  convolution $A^{g,\eps}$ defined by~\eqref{def:Aeps}.}. By doing this, we get an \textit{approximate} integration by parts formula, in the sense that there is an additional remainder term appearing in the formula:
\begin{align*}
\frac{\mathrm d}{\mathrm d\rho}_{\vert \rho=0} P_t \phi (\mu_0^{g+\rho g'  h}) 
=\frac{1}{t} \Ewb{\phi(\mu_t^g) \sum_{k \in \Z} \int_0^t \Re( \overline{\lambda_s^{k,\eps}} \mathrm dW_s^k)}
+\mathcal O(\eps),
\end{align*}
where $\lambda^{k,\eps}_\cdot$ is a stochastic process\footnote{defined in Lemma~\ref{lemme:decomposition Fourier}.}. 
Controlling the remainder term leads us, by a final bootstrap argument, to  the desired upper bound on $| D \widehat{P_t\phi}(g) \cdot h |$, at the prize of worsening the rate of blow-up. We are not claiming that the rate of $t^{-(2+\theta)}$ is sharp, but we expect that a rate of $t^{-1/2}$ is unachievable for this process. Let us mention that the  author improves this rate of blow-up to $t^{-(1+\theta)}$, at the prize of assuming $\mathcal C^{4+\theta}$-regularity of $g$ and $h$. 
Since the proof is long and technical, it is not included in this paper but we refer to~\cite[Chapter IV]{marx_thesis} for all the details and for an application to a gradient estimate for an inhomogeneous SPDE with Hölder continuous source term.

\medskip

Furthermore, the idiosyncratic noise $\beta$ is important as well. 
Of course, the addition of~$\beta$ does not change dramatically the dynamics of the process, since it acts as a rotation on the circle of the whole system. Nevertheless, as it has already been pointed out, the diffusion process $(\mu_t^g)_{t \in [0,T]}$ is defined by an average over the realizations of~$\beta$. 
The importance of that averaging is consistent with SPDE theory. Indeed, $(\mu_t^g)_{t \in [0,T]}$  solves the following equation:
\begin{align*}
\label{intro_eq_for_mu}
\mathrm d \mu_t^g -  \frac{1+ \sum_{k \in \Z} f_k^2}{2} 
\Delta (\mu_t^g) \mathrm dt
+\partial_x \bigg( \sum_{k \in \Z} f_k \Re \big( e^{-ik \; \cdot \;} \mathrm dW_t^k   \big)\;  \mu_t^g \bigg) &=0 , 
\end{align*}
with initial condition $\mu_t^g |_{t=0} = \mu_0^g$. 
The  noise $\beta$ manifests in the additional term $\frac{1}{2}$ in front of $\Delta (\mu_t^g)$. 
On the level of the densities,  $(p_t^g)_{t\in [0,T]}$ solves the following equation
\begin{align*}
\mathrm d p_t^g (v)=
-\partial_v \bigg( p_t^g(v) \sum_{k \in \Z} f_k \Re ( e^{-ikv} \mathrm dW_t^k )  \bigg)
+\lambda   (p_t^g)''(v) \mathrm dt,
\end{align*}
with $\lambda= \frac{1+ \sum_{k \in \Z} f_k^2}{2}$. Denis and Stoica showed in~\cite{denis_stoica,denis_matoussi_stoica} that the above equation is well-posed -- and they also gave energy estimates -- if $\lambda$ is strictly larger than the critical threshold  $\lambda_{\operatorname{crit}}=\frac{\sum_{k \in \Z} f_k^2}{2}$. If we  considered equation~\eqref{intro:sde_beta} \textit{without} $\beta$, we would exactly obtain the above equation with $\lambda=\lambda_{\operatorname{crit}}$. Therefore, it seems that adding a level of randomness is crucial to get our estimate. Precisely, in our above-described strategy, the Brownian motion $\beta$ plays a key role in controlling the remainder term.

In addition, let us note  that we study a process on the one-dimensional circle $\tor$, and not on the real line as e.g.\! in~\cite{konarovskyi17system, marx20}. We made this choice rather for technical reasons, in order to deal with processes of compactly supported measures having a positive density on the whole space. The main result is  restricted to functions $g$ which have a strictly positive derivative, meaning that the associated measure has a density with respect to Lebesgue measure on the torus. The constant $C_g$ tends to infinity when $\min_{u \in [0,1]} g'(u)$ gets closer to zero. The assumptions on the regularity of $g$ and $h$ seem reasonable since our model is close to the process $(\mu^{\mathrm{MMAF}}_t)_{t}$ called \textit{modified massive Arratia flow} introduced in~\cite{konarovskyi17system}, which  has highly singular coefficients. 
Indeed, Konarovskyi and von Renesse showed in~\cite{konarovskyivonrenesse18} that $(\mu^{\mathrm{MMAF}}_t)_{t}$, which is almost surely of finite support for all $t>0$, solves the following SPDE:
\[ \mathrm d\mu^{\mathrm{MMAF}}_t = \Gamma(\mu^{\mathrm{MMAF}}_t)\mathrm dt + \mathrm{div} (\sqrt{\mu^{\mathrm{MMAF}}_t}\mathrm d W_t),\]
where $\Gamma$ is defined by $<f,\Gamma(\nu)>=\frac{1}{2} \sum_{x \in \mathrm{Supp}(\nu)} f''(x)$ for any  $f \in  \mathcal C^2_\mathrm b (\R)$.

\subsection*{Organization of the paper}

The goal of Section~\ref{sec:statement} is to define properly equation~\eqref{intro:sde_beta} and to state the main result of the paper.  The proof of the theorem is then divided into four main steps. 
 We start in the short Section~\ref{sec:preparation} 
 by splitting the gradient of the semi-group into two parts, one regularized term and one remainder term, which are separately studied in sections~\ref{sec:analysis_I1} and~\ref{sec:analysis_I2}, respectively. 
Finally in Section~\ref{sec:end_proof}, we complete the proof by a bootstrap argument.

\section{Statement of the main result}
\label{sec:statement}
 
The main result of the paper is stated in Paragraph~\ref{parag:stat_main_theo}. Before, we define precisely the diffusion on the torus (Paragraph~\ref{parag:diff_torus}), its associated semi-group (Paragraph~\ref{parag:semi_group}) and the assumptions on the test functions (Paragraph~\ref{parag:assump_test_function}).

\subsection{A diffusion on the torus}
\label{parag:diff_torus}

In this paper, we study the following  stochastic differential equation on a fixed time interval $[0,T]$
\begin{align}
\label{eq_SDE_diff_torus}
\mathrm dx_t^g(u) = \sum_{k \in \Z} f_k\; \Re \left(e^{-ik x_t^g(u)} \mathrm dW^k_t \right)+ \mathrm d\beta_t, \quad t \in [0,T], \; u \in \R,
\end{align}
with initial condition $x^g_0=g$. 
In this paragraph, we first define the assumptions made on $W^k$, $\beta$, $f_k$ and $g$, where we  emphasise the interpretation of $(x^g_t)_{t \in [0,T]}$ as a diffusion on the torus. Then we state existence, uniqueness and some important properties of solutions to equation~\eqref{eq_SDE_diff_torus}.

Let $\beta$, $(W^{\Re,k})_{k \in \Z}$, $(W^{\Im,k})_{k \in \Z}$ be a collection of independent standard  real-valued Brownian motions. Thus $W^k_\cdot=W^{\Re,k}_\cdot + i W^{\Im,k}_\cdot$ denotes a $\C$-valued Brownian motion. The notation $\Re$ denotes the real part of a complex number, so that equation~\eqref{eq_SDE_diff_torus} can alternatively be written as follows
\begin{align*}
\mathrm dx_t^g(u) 
&= \sum_{k \in \Z} f_k \cos(kx_t^g(u)) \mathrm dW_t^{\Re,k} + \sum_{k \in \Z} f_k \sin(kx_t^g(u)) \mathrm dW_t^{\Im,k}+\mathrm d\beta_t.
\end{align*}

\begin{defin}
\label{defin_f}
We say that $f:=(f_k)_{k \in \Z}$ is of order $\alpha>0$ if there are $c>0$ and $C>0$ such that $\frac{c}{\crochetk^\alpha}\leq  |f_k| \leq \frac{C}{\crochetk^\alpha}$ for every $k \in \Z$, where $\crochetk:=(1+|k|^2)^{1/2}$. 
\end{defin}

Note that if $f$ is of order $\alpha > \frac{1}{2}$, then for each $u \in \R$, the particle $(x^g_t(u))_{t\in [0,T]}$ has a finite quadratic variation equal to $\langle x^g(u),x^g(u) \rangle_t= \Big(\sum_{k\in \Z} f_k^2+1\Big) t$.

Let $\tor$ be the one-dimensional torus, that we  identify with the interval $[0,2\pi]$.  
 $\mathcal P(\tor)$ denotes the space of probability measures on the torus. We  consider the $L_2$-Wasserstein metric $W_2^\tor$ on $\mathcal P(\tor)$, defined by
 $
W_2^\tor (\mu,\nu):=\inf_{\pi \in \Pi(\mu,\nu)} 
\left( \int_{\tor^2} d^\tor (x,y)^2 \; \mathrm d\pi(x,y) \right)
^{1/2} $,
where $\Pi(\mu,\nu)$ is the set of probability measures on $\tor^2$ with first marginal $\mu$ and second marginal~$\nu$ and where $d^\tor$ is the distance on the torus defined by $d^\tor(x,y):=\inf_{k \in \Z} |x-y-2k\pi|$, where  $x,y \in \R$.

\begin{defin}
\label{defin_g}
Let  $\mathscr G^1$ be the set of $\mathcal C^1$-functions $g:\R \to \R$  such that for every $u \in \R$, $g'(u)>0$ and  $g(u+1)=g(u)+2\pi$. 
Let $\sim$ be the following equivalence relation on $\mathscr G^1$: $g_1 \sim g_2$ if and only if there exists $c\in \R$ such that $g_2(\cdot)=g_1(\cdot +c)$.  
We denote by $\mathbf G^1$ the set of equivalence classes $\mathscr G^1 / \sim$.
\end{defin}

An interpretation for Definition~\ref{defin_g} is that the initial condition $g$ is seen as the quantile function (or inverse c.d.f.\!  function) associated with the measure $\nu_0^g \in \mathcal P(\tor)$ with density $p(x)=\frac{1}{g'(g^{-1}(x))}$, $x \in [0,2\pi]$, with respect to Lebesgue measure on $\tor$. 
There is a one-to-one correspondence between $\mathbf G^1$ and the set of positive densities on the torus, see Paragraph~\ref{subsec:density and quantile} in appendix for more details. 

Existence, uniqueness, continuity and differentiability of solutions to~\eqref{eq_SDE_diff_torus} depend on the order $\alpha$ of $f$ and on the regularity of $g$, as shown by the following two propositions. The proofs, which are classical, are left to appendix.

\begin{prop}
\label{prop:exist, uniq, cont, growth}
Let $g \in \mathscr G^1$ and $f$ be of order $\alpha > \frac{3}{2}$. Then for each $u\in \R$, strong existence and pathwise uniqueness in $\mathcal C(\R \times [0,T])$ hold for equation~\eqref{eq_SDE_diff_torus}.  Moreover  almost surely, for every $u\in [0,1]$, $(x_t^g(u))_{t\in [0,T]}$ satisfies equation~\eqref{eq_SDE_diff_torus}  and  for every $t\in [0,T]$, $u \mapsto x^g_t(u)$ is strictly increasing. 
\end{prop}

\begin{proof}
See Paragraph~\ref{app_well_posed} in appendix. 
\end{proof}

For every $j\in \N$ and $ \theta \in [0,1)$, let $\mathcal C^{j+\theta}$ denote the set of $\mathcal C^j$-functions whose $j^\mathrm{th}$ derivative is $\theta$-Hölder continuous.
By extension,  $\mathscr G^{j+\theta} \subseteq \mathscr G^1$ and $\mathbf G^{j+\theta} \subseteq \mathbf G^1$,  are the subsets of $\mathscr G^1$ and of $\mathbf G^1$ consisting  of all $\mathcal C^{j+\theta}$-functions and  $\mathcal C^{j+\theta}$-equivalence classes, respectively.

\begin{prop}
\label{prop:differentiability}
Let $j \geq 1$,  $\theta \in (0,1)$,  $g \in \mathscr G^{j+\theta}$ and $f$ be of order $\alpha> j + \frac{1}{2} + \theta$. Then almost surely, for every $t\in [0,T]$, the map $u \mapsto x_t^g(u)$ is a $\mathcal C^{j+\theta'}$-function for every $\theta'< \theta$. Moreover, its first derivative satisfies  almost surely
\begin{align}
\label{exponential explicit partial xtg}
\partial_u x_t^g(u) =g'(u) \exp\bigg( \sum_{k \in \Z} f_k \int_0^t  \Re \left(-ik e^{-ik x_s^g(u)} \mathrm dW^k_s \right)-\frac{t}{2} \sum_{k \in \Z} f_k^2 k^2 \bigg), \quad u\in \R, t\in [0,T].
\end{align}
\end{prop}

\begin{proof}
See Paragraph~\ref{app_well_posed} in appendix. 
\end{proof}

The next proposition states that the flow of the SDE preserves the equivalence classes of quantile functions that we  introduced in Definition~\ref{defin_g}.

\begin{prop}
\label{prop:stabilite}
Let $\theta \in (0,1)$,  $g \in \mathscr G^{1+\theta}$ and $f$ be of order $\alpha>  \frac{3}{2}+ \theta$. 
Then almost surely, for every $t\in [0,T]$, the map $u \mapsto x_t^g(u)$ belongs to $\mathscr G^1$. 
Moreover, if $g_1 \sim g_2$, then almost surely, $x_t^{g_1} \sim x_t^{g_2}$ for every $t \in [0,T]$. 
\end{prop}

\begin{proof}
By propositions~\ref{prop:exist, uniq, cont, growth} and~\ref{prop:differentiability}, it is clear that $u \mapsto x_t^g(u)$ belongs to $\mathcal C^1$ and that $\partial_u x_t^g (u)>0$ for every $u\in \R$. 
Furthermore, let $(y_t^g)_{t\in [0,T]}$ be the process defined by $y_t^g(u):=x_t^g(u+1)-2\pi$. 
By definition of $\mathscr G^{1+\theta}$, $g(u+1)-2\pi =g(u)$, thus for every $t\in [0,T]$ and $u\in \R$, 
$y_t^g(u)= g(u) + \sum_{k \in \Z} f_k \int_0^t \Re \left(e^{-ik y_s^g(u)} \mathrm dW^k_s \right)+\beta_t$. 
Therefore $(y_t^g(u))_{t\in [0,T], u\in \R}$ and $(x_t^g(u))_{t\in [0,T], u\in \R}$ satisfy the same equation and belong to $\mathcal C(\R \times [0,T])$. By Proposition~\ref{prop:exist, uniq, cont, growth}, there is a unique solution in this space. Thus for every $t\in [0,T]$ and every $u\in \R$, $x_t^g(u+1)-2\pi =x_t^g(u)$.  We deduce that $x_t^g$ belongs to $\mathscr G^1$ for every $t\in [0,T]$. 

The proof of the second statement is similar; if there is $c\in \R$ such that $g_2(u)=g_1(u+c) $ for every $u \in \R$, then the processes $(x_t^{g_2}(u))_{t\in [0,T], u\in \R}$ and $(x_t^{g_1}(u+c))_{t\in [0,T], u\in \R}$ satisfy the same equation and  are equal. 
\end{proof}

By Proposition~\ref{prop:stabilite}, we are able to give a meaning to  equation~\eqref{eq_SDE_diff_torus} with initial value $g$ in $\mathbf G^{1+\theta}$. Indeed, for each $t\in [0,T]$, the solution $x_t^g$ will take its values in $\mathbf G^{1+\theta'}$ for every $\theta' < \theta$. 
More generally, by Proposition~\ref{prop:differentiability}, if the initial condition $g$ belongs to $\mathbf G^{j+\theta}$ for $j \geq 1$ and $\theta \in (0,1)$, then for each $t\in [0,T]$, $x_t^g$ belongs to $\mathbf G^{j+\theta'}$ for every $\theta' < \theta$.

Furthermore, the $L_p$-norms (in the space variable) of the derivatives $\partial_u^{(j)} x_t^g$, $j \geq 1$, and of  $\frac{1}{\partial_u x_t^g}$ can be easily  controlled with respect to the initial conditions. All the inequalities that will be needed later in this paper were listed and proved in Lemma~\ref{lem_control_moments} in appendix.

To conclude this paragraph, let us mention that the solution to equation~\eqref{eq_SDE_diff_torus} can be equivalently seen as a solution to the following parametric SDE
\begin{align}
\label{eq_SDE_z}
Z_t^x = x+ \sum_{k \in \Z} f_k \int_0^t \Re \left(e^{-ik Z_s^x} \mathrm dW^k_s \right)+ \beta_t, \quad x \in \R.
\end{align}
Under the same assumptions over $f$, well-posedness and regularity of solutions to equation~\eqref{eq_SDE_z} can be shown, see Proposition~\ref{prop_para} in appendix. Moreover, $Z_t$ is closely related to $x_t^g$ by the following identities.

\begin{prop}
\label{prop:kunita_identities}
Let $g\in \mathbf G^{1+\theta}$ and $f$ be of order $\alpha> \frac{3}{2}+\theta$ for some $\theta \in (0,1)$. Then 
almost surely, for every $t\in [0,T]$ and  $u\in [0,1]$,
\begin{align}
\label{egalite Zx}
Z_t^{g(u)}&=x_t^g(u), \\
\label{liens entre les dérivées}
\partial_x Z_t^{g(u)} &= \frac{\partial_u x_t^g(u)}{g'(u)}.
\end{align}
\end{prop}

\begin{proof}
See end of Paragraph~\ref{app_well_posed} in appendix. 
\end{proof}

\subsection{Semi-group averaged out by idiosyncratic noise}
\label{parag:semi_group}

According to Proposition~\ref{prop:stabilite}, $u \in [0,1] \mapsto x_t^g(u)$ is for each fixed $t$ a quantile function of the measure $\nu_t^g \in \mathcal P(\tor)$ defined by $\nu_t^g := \Leb_{[0,1]} \circ (x_t^g)^{-1}$. 
However, the stochastic process $(\nu_t^g)_{t \in [0,T]}$ is not regular enough to obtain a gradient estimate for the associated semi-group. Therefore, we average out the realization of the noise $\beta$ by defining $\mu_t^g := (\Leb_{[0,1]} \otimes \mathbb P^\beta) \circ (x_t^g)^{-1}$. 
In terms of densities, if $q_t^g$ is the density of $\nu_t^g$, then $p_t^g(x):=\mathbb E^\beta \left[ q_t^g (x)\right]$, $x\in \tor$,  is the density of $\mu_t^g$.

To be more precise, we define three sources of randomness, for the noises $W^k$, $\beta$ and the initial condition $g$, respectively. 
Let $(\Omega^W, \mathcal G^W, (\mathcal G^W_t)_{t\in [0,T]}, \mathbb P^W)$ and $(\Omega^\beta, \mathcal G^\beta, (\mathcal G^\beta_t)_{t\in [0,T]}, \mathbb P^\beta)$ be  filtered probability spaces satisfying usual conditions, on which we define a $(\mathcal G^W_t)_{t\in [0,T]}$-adapted collection $((W^{k}_t)_{t\in [0,T]})_{k \in \Z}$ of independent   $\C$-valued Brownian motions  and a $(\mathcal G^\beta_t)_{t\in [0,T]}$-adapted standard Brownian motion $(\beta_t)_{t\in [0,T]}$, respectively. 
Let $(\Omega^0, \mathcal G^0, \mathbb P^0)$ be another probability space rich enough to support $\mathbf G^1$-valued random variables with any possible distribution.
We  denote by $\mathbb E^B$, $\mathbb E^W$ and $\mathbb E^0$ the expectations  associated to $\mathbb P^\beta$, $\mathbb P^W$ and~$\mathbb P^0$, respectively.
Let $(\Omega, \mathcal G, (\mathcal G_t)_{t\in [0,T]}, \mathbb P)$ be the filtered probability space defined by $\Omega:=\Omega^W \times \Omega^\beta \times\Omega^0$, $\mathcal G:= \mathcal G^W \otimes \mathcal G^\beta \otimes \mathcal G^0$, $\mathcal G_t:= \sigma ((\mathcal G^W_s)_{s\leq t},(\mathcal G^\beta_s)_{s\leq t},\mathcal G^0)$ and $\mathbb P:= \mathbb P^W \otimes \mathbb P^\beta \otimes \mathbb P^0$. Without loss of generality, we assume the filtration $(\mathcal G_t)_{t\in [0,T]}$ to be complete and, up to adding negligible subsets to $\mathcal G^0$,  we assume that $\mathcal G_0=\mathcal G^0$.

\begin{defin}
\label{def:mesure_mu}
Fix $t\in [0,T]$ and $\omega \in\Omega^W \times \Omega^0$.
Then $x_t^g(\omega)$ is a random variable from $[0,1] \times \Omega^\beta$ to $\R$. 
We denote  by $\mu_t^g(\omega)$  its law, that is:
\begin{align*}
\mu_t^g(\omega) := \left( \Leb_{[0,1]} \otimes \mathbb P^\beta  \right) \circ \left( x_t^g(\omega) \right)^{-1}.
\end{align*} 
In particular, $\mu_t^g$ is  a random variable defined on $\Omega^W \times \Omega^0$ with values in $\mathcal P_2(\R)$.  
\end{defin}

Define now the  semi-group $(P_t)_{t \in [0,T]}$ associated to $(\mu^g_t)_{t\in [0,T]}$.
Let $\phi:\mathcal P_2(\R)\to \R$ be a bounded and continuous function. Let  $\widehat{\phi}: L_2([0,1]\times \Omega^\beta ) \to \R$ be the lifted function of $\phi$, defined by 
$\widehat{\phi}(X):= \phi((\Leb_{[0,1]} \otimes \mathbb P^\beta) \circ X^{-1})$. 
In other words, $\widehat{\phi}(X)= \phi(\mathcal L_{[0,1] \times \Omega^\beta}(X))$, where $\mathcal L_{[0,1] \times \Omega^\beta}(X)$ denotes the law of the random variable $X:[0,1] \times \Omega^\beta \to \R$. 

\begin{defin}
\label{defin:semi-group P_t}
For every $t \in \R$ and $\mu \in \mathcal P_2(\R)$, 
\begin{align*}
P_t\phi(\mu):= \mathbb E^W \left[ \widehat{\phi}(Z_t^X) \right], 
\end{align*}
where $(Z_t^x)$ is the solution to SDE~\eqref{eq_SDE_z} and $\mu=\Leb_{[0,1]} \circ X^{-1}$.  
\end{defin}

\begin{prop}
$P_t \phi$ is well-defined and for every $t\in [0,T]$ and $g \in \mathbf G^1$,
\begin{align}
\label{def:semi-group P_t bis}
P_t \phi (\mu_0^g)  = \mathbb E^W \left[  \widehat{\phi}\left( x_t^g \right) \right]= \mathbb E^W \left[ \phi (\mu_t^g)\right].
\end{align}
\end{prop}

\begin{proof}
By Proposition~\ref{prop_para}, the parametric SDE~\eqref{eq_SDE_z}  is strongly well-posed.
Thus, if $X,X' \in L_2[0,1]$ have same law, \textit{i.e.\!} $\mathcal L_{[0,1]}(X)=\mathcal L_{[0,1] }(X')$, then $\mathbb P^W$-almost surely for every $t\in [0,T]$, $\mathcal L_{[0,1] \times \Omega^\beta }(Z_t^X)=\mathcal L_{[0,1]\times \Omega^\beta}(Z_t^{X'})$.
It follows that $\mathbb E^W \left[ \widehat{\phi}(Z_t^X) \right]=\mathbb E^W \left[ \widehat{\phi}(Z_t^{X'}) \right]$ for all $t \in [0,T]$, so $\widehat{P_t \phi} (X):= \mathbb E^W \left[ \widehat{\phi}(Z_t^X) \right]$ does not depend on the representative $X$ of the law $\mu$. 

Moreover $\mathbb P^W$-almost surely , $\widehat{\phi} (x_t^g)= \phi(\mu_t^g)$. Furthermore by Proposition~\ref{prop:kunita_identities}, $\mathbb P^W \otimes \mathbb P^\beta$-almost surely, for every $t\in [0,T]$ and for every $u\in [0,1]$, $Z_t^{g(u)}=x_t^g(u)$. In particular, $\mathbb P^W$-almost surely and for every $t\in [0,T]$,  $Z_t^{g(u)}=x_t^g(u)$  holds true $\Leb_{[0,1]} \otimes  \mathbb P^\beta$-almost surely. Therefore, 
\begin{align*}
P_t \phi (\mu_0^g) = \widehat{P_t \phi}(g) = \mathbb E^W \left[  \widehat{\phi}\left( Z_t^g \right) \right] = \mathbb E^W \left[  \widehat{\phi}\left( x_t^g \right) \right]= \mathbb E^W \left[ \phi (\mu_t^g)\right],
\end{align*}
which proves~\eqref{def:semi-group P_t bis}.
\end{proof}

\subsection{Assumptions on the test functions}
\label{parag:assump_test_function}

The semi-group $(P_t)_{t \in [0,T]}$ acts on bounded and continuous functions $\phi:\mathcal P_2(\R)\to \R$. We will assume the assumptions on $\phi$ defined hereafter.

\begin{defin}
\label{defin:tor_stable}
We define an equivalence class on $\mathcal P_2(\R)$ by: $\mu \sim \nu$ if and only if $\mu(A+2\pi \Z)=\nu(A+2\pi \Z)$ for any $A \in \mathcal B[0,2\pi]$. A function $\phi: \mathcal P_2(\R) \to \R$ is said to be \emph{$\tor$-stable} if $\phi(\mu)=\phi(\nu)$ for any   $\mu \sim \nu$. 
 In particular,  $\phi$ induces a map from $\mathcal P(\tor)$ to $\R$.
 \end{defin}

In particular, for a $\tor$-stable function $\phi$ and $X \in L_2(\Omega)$, $\widehat{\phi}(X)= \widehat{\phi}(\accolade{X})$, where $\accolade{x}$ is the unique number in $[0,2\pi)$ such that $x-\accolade{x} \in 2\pi\Z$.
Let us mention two important classes of examples of $\tor$-stable functions:
\begin{itemize}
\item[-] if $h: \R \to \R$ is a $2\pi$-periodic function, the map $\phi:\mu \in \mathcal P_2(\R) \mapsto \int_\R h(x) \mathrm d\mu(x)$ is $\tor$-stable. The $2\pi$-periodicity condition ensures that $\widehat{\phi}(X)=\E{h(X)}=\E{h(\accolade{X})}=\widehat{\phi}(\accolade{X})$. 
\item[-] if $h: \R \to \R$ is a $2\pi$-periodic function, the map $\phi:\mu \in \mathcal P_2(\R) \mapsto \int_{\R^2} h(x-y) \mathrm d(\mu \otimes \mu)(x,y)$ is also $\tor$-stable. 
\end{itemize}

If the reader is not familiar with the L-derivative $\partial_\mu \phi$, we refer to Paragraph~\ref{app_diff_calculus} in appendix for a short introduction. 

\begin{defin}
\label{def:phi assumptions}
A function $\phi:\mathcal P_2(\R)\to \R$ is said to satisfy the \emph{$\phi$-assumptions} if the following three conditions hold:
\begin{itemize}
\item[$(\phi 1)$] $\phi$ is $\tor$-stable, bounded and continuous on $\mathcal P_2(\R)$. 
\item[$(\phi 2)$] $\phi$ is $L$-differentiable  and $\sup_{\mu \in \mathcal P_2(\R)} \int_\R |\partial_\mu \phi (\mu) (x) |^2 \mathrm d\mu(x) <+\infty$.
\item[$(\phi 3)$] The Fréchet derivative $D\widehat{\phi}$ is Lipschitz-continuous: there is $C>0$ such that
\begin{align*}
\E{|D \widehat{\phi}(X)-D \widehat{\phi}(Y)|^2} \leq C\E{ |X-Y|^2}, \quad X,Y \in L_2(\Omega).
\end{align*}
\end{itemize}
\end{defin}

\begin{rem}
\label{rem:phi2_implique_lipschitz}
Assumption $(\phi 2)$ implies that $\widehat{\phi}$ is  Lipschitz-continuous.
Therefore $\widehat{\phi}$ belongs to $\mathcal C_{\mathrm b}^{1,1}(L_2(\Omega))$, the space of bounded and Lipschitz continuous functions on $L_2(\Omega)$ whose Fréchet  derivative is also bounded and Lipschitz continuous on $L_2(\Omega)$. 
Let us mention that the inf-sup convolution method introduced by Lasry and Lions allows to construct, for each bounded uniformly continuous function $\varphi$ defined on $L_2(\Omega)$, a sequence $(\varphi_n)_n$ of $\mathcal C_{\mathrm b}^{1,1}(L_2(\Omega))$-functions converging uniformly to~$\varphi$ on $L_2(\Omega)$, see~\cite{lasrylions}. 
\end{rem}

The following statement shows that the class of functions satisfying the $\phi$-assumptions  is stable under the action of $(P_t)_{t \in [0,T]}$.

\begin{prop}
\label{prop:phi assumptions Pt phi}
Assume that $f$ is of order $\alpha> \frac{5}{2}$. 
Let $\phi:\mathcal P_2(\R)\to \R$ be a function satisfying the $\phi$-assumptions. Then for every $t\in [0,T]$, $P_t \phi:\mathcal P_2(\R)\to \R$ also satisfies the $\phi$-assumptions. 
Moreover, for any fixed $t\in [0,T]$, the Fréchet derivative of $g \mapsto P_t\phi (\mu_0^g)$ is given by
\begin{align}
\label{frechet_derivative}
\frac{\mathrm d}{\mathrm d\rho}_{\vert \rho=0} P_t \phi (\mu_0^{g+\rho h}) = D \widehat{P_t\phi}(g) \cdot h
&= \Ewb{\int_0^1 D\widehat{\phi}(x_t^g)_u \; \frac{\partial_u x_t^g(u)}{g'(u)}
  h(u) \mathrm du} \notag \\
&= \Ewb{ \int_0^1 \partial_\mu \phi(\mu_t^g) (x_t^g(u)) \frac{\partial_u x_t^g(u)}{g'(u)}
  h(u) \mathrm du }.
\end{align}
\end{prop}

Note that $D\widehat{\phi}(x_t^g)$ is an element of the dual of $L_2([0,1] \times \Omega^\beta)$, identified here with an element of $L_2([0,1] \times \Omega^\beta)$. The proof of Proposition~\ref{prop:phi assumptions Pt phi} is given in Paragraph~\ref{subsec:functions_proba_measures} in appendix.

\subsection{Statement of the main theorem}
\label{parag:stat_main_theo}

The main result of this paper is a gradient estimate for the semi-group $(P_t)_{t \in [0,T]}$ associated to $(\mu_t^g)_{t \in [0,T]}$, which is given at points $g \in \mathbf G^{3+\theta}$ and for directions of perturbations $h$ defined as follows.

\begin{defin}
We denote by $\Delta^1$ the set of $1$-periodic $\mathcal C^1$-functions $h:\R\to \R$.
We define the following norm on $\Delta^1$: $ \|h\|_{\mathcal C^1}:= \sup_{u \in [0,1]} |h(u)| + \sup_{u \in [0,1]} |h'(u)|$.
\end{defin}

A simple computation shows that for $|\rho|\ll 1$,  $g+\rho h$ still belongs to $\mathbf G^1$. Let us state the main theorem. 

\begin{theo}
\label{theo:Bismut pour phi assumptions}
Let $\phi:\mathcal P_2(\R)\to \R$ satisfy the $\phi$-assumptions.
Let $\theta \in (0,1)$ and $f$ be of order $\alpha=\frac{7}{2}+\theta$.
Let $g \in \mathbf G^{3+\theta}$ and $h \in \Delta^1$ be two deterministic functions.
 Then there is $C_g$ independent of $h$ such that for every $t \in (0,T]$
\begin{align}
\label{Bismut phi assumptions}
\left|\frac{\mathrm d}{\mathrm d\rho}_{\vert \rho=0} P_t \phi (\mu_0^{g+\rho  h})\right| 
\leq C_g \frac{\| \phi \|_{L_\infty}}{t^{2+\theta}}\| h\|_{\mathcal C^1},
\end{align}
where $C_g$ is bounded when $\|g'''\|_{L_\infty}+\|g''\|_{L_\infty}+\|g'\|_{L_\infty} + \| \textstyle\frac{1}{g'}\|_{L_\infty}$ is bounded. 
\end{theo}

In the following section, we split the l.h.s.\! of~\eqref{Bismut phi assumptions} into two terms, $I_1$ and $I_2$, which will be  studied separately in sections~\ref{sec:analysis_I1} and~\ref{sec:analysis_I2}, respectively. 

\section{Preparation for the proof}
\label{sec:preparation}

To start this paragraph, we rewrite the gradient of the semi-group in terms of the L-derivative of $P_t \phi$ and of the linear functional derivative of $P_t \phi$. We refer to paragraph~\ref{app_diff_calculus} in appendix for a short remainder about definitions and relationships between the different types of derivatives.

For convenience, the following lemma is written for a perturbation $g'h$ instead of $h$ (the corresponding result for $h$ can be naturally obtained by applying the following formula to $\frac{h}{g'}$ instead of $h$). For later purposes,  the lemma is stated for random functions $g$ and $h$, 
with a   $\mathcal G_0$-measurable randomness. Recall that within this framework, $g$ and $h$ are independent of $((W^k)_{k \in \Z}, \beta)$.

\begin{lemme}
\label{lemme:derivee_du_semigroupe}
Let $\phi$, $\theta$ and $f$ be as in Theorem~\ref{theo:Bismut pour phi assumptions}. 
Let $g$ and $h$ be $\mathcal G_0$-measurable random variables with values respectively in $\mathbf G^{3+\theta}$ and~$\Delta^1$. Then for every $t\in [0,T]$, 
\begin{align}
\label{formules_derivees_semigroupe}
\frac{\mathrm d}{\mathrm d\rho}_{\vert \rho=0} P_t \phi (\mu_0^{g+\rho g'  h}) 
&= \int_0^1 \partial_\mu (P_t\phi) (\mu_0^g) (g(u)) \; g'(u) h(u) \;\mathrm du  \notag\\
&=\Ewb{ \int_0^1 \partial_\mu \phi(\mu_t^g) (x_t^g(u)) \; \partial_u x_t^g(u) h(u) \mathrm du } \\
&=-\int_0^1 \frac{\delta P_t \phi}{\delta m}  (\mu_0^g)(g(u)) \; h'(u) \;\mathrm du. \notag
\end{align}
\end{lemme}

\begin{proof}
Fix $\omega^0$ in an almost-sure event of $\Omega^0$ such that $g=g(\omega^0)$ belongs to $\mathbf G^{3+\theta}$ and $h=h(\omega^0)$ belongs to $\Delta^1$. 
Since $g'$ is $1$-periodic and positive, $\rho_0:=\frac{1}{\|h'\|_{L_\infty}} \inf_{u \in \R} g'(u)$ is positive and $g+\rho h$ belongs to~$\mathbf G^1$ for every $\rho \in (-\rho_0,\rho_0)$. 
The first equality in~\eqref{formules_derivees_semigroupe} follows from the definition of the L-derivative:
\begin{multline}
\label{derivées}
\frac{\mathrm d}{\mathrm d\rho}_{\vert \rho=0} P_t \phi (\mu_0^{g+\rho g'  h})
= \frac{\mathrm d}{\mathrm d\rho}_{\vert \rho=0} \widehat{P_t \phi} (g+\rho g' h)
=  D \widehat{P_t\phi}(g) \cdot (g'h )
= \int_0^1 D \widehat{P_t\phi}(g)_u \; g'(u)h(u) \mathrm du  \\
=\int_0^1 \partial_\mu (P_t\phi) (\mu_0^g) (g(u))\; g'(u) h(u) \mathrm du. 
\end{multline}
The second equality in~\eqref{formules_derivees_semigroupe} was already stated in Proposition~\ref{prop:phi assumptions Pt phi}. 
For the third equality in~\eqref{formules_derivees_semigroupe}, we use the relationship between the L-derivative and the functional linear derivative (see Proposition~\ref{prop:link_derivatives})
\begin{multline*}
\int_0^1 \partial_\mu (P_t\phi) (\mu_0^g) (g(u))\; g'(u) h(u) \mathrm du
= \int_0^1 \partial_v \left\{ \frac{\delta P_t \phi}{\delta m}(\mu_0^g) \right\} (g(u)) g'(u)\; h(u) \mathrm du \\\begin{aligned}
&= \int_0^1 \partial_u \left\{ \frac{\delta P_t \phi}{\delta m} (\mu_0^g) (g(\cdot)) \right\}(u) \;h(u) \mathrm du \\
&=\left[ \frac{\delta P_t \phi}{\delta m} (\mu_0^g) (g(u)) \; h(u) \right]_0^1 -\int_0^1 \frac{\delta P_t \phi}{\delta m} (\mu_0^g) (g(u)) \;h'(u) \mathrm du,
\end{aligned}
\end{multline*}
by an integration by parts formula. 
Furthermore,  $v \mapsto \frac{\delta P_t \phi}{\delta m} (\mu_0^g) (v)$ is $2\pi$-periodic. This  follows from Proposition~\ref{prop:integrale nulle} in appendix, because on the one hand $P_t \phi$ satisfies the $\phi$-assumptions by Proposition~\ref{prop:phi assumptions Pt phi} and on the other hand the probability measure $\mu_0^g$ has   density $g'$ on the torus, which is strictly positive everywhere.  
 It follows that $\frac{\delta P_t \phi}{\delta m} (\mu_0^g) (g(1))= \frac{\delta P_t \phi}{\delta m} (\mu_0^g) (g(0)+2\pi)= \frac{\delta P_t \phi}{\delta m} (\mu_0^g) (g(0))$. 
 Since $h$ is $1$-periodic, we conclude that $\left[ \frac{\delta P_t \phi}{\delta m} (\mu_0^g) (g(u))  h(u) \right]_0^1=0$.
\end{proof}

By Proposition~\ref{prop:stabilite}, $x_t^g$ belongs to $\mathbf G^1$ for every $t\in[0,T]$. In particular, $u \mapsto x_t^g(u)$ is invertible and its inverse $F_t^g:=(x_t^g)^{-1}$ satisfies $F_t^g(x+2\pi)=F_t^g(x)+1$. 
We define $(A_t^g)_{t\in [0,T]}$ by
\begin{align}
\label{def:Agt}
A_t^g:= (\partial_u x_t^g)(F_t^g(\cdot)) \; h(F_t^g(\cdot)).
\end{align}
and $(A_t^{g,\eps})_{t\in [0,T]}$ by
\begin{align}
\label{def:Aeps}
A_t^{g,\eps} := A_t^g \ast \varphi_\eps = \int_\R A_t^g (\cdot-y) \varphi_\eps (y) \mathrm dy, 
\end{align}
where $\varphi(x)=\frac{1}{\sqrt{2\pi}}e^{-x^2/2}$  and  $\varphi_\eps(x)=\frac{1}{\eps}\varphi(\frac{x}{\eps})$.  

\begin{lemme}
For any $t \in [0,T]$ and $\eps >0$, 
$A_t^g$ is a $2\pi$-periodic $\mathcal C^1$-function and $A_t^{g,\eps}$ is a $2\pi$-periodic $\mathcal C^\infty$-function.
\end{lemme}

\begin{proof}
The periodicity property follows from the fact that $h$ and $\partial_u x_t^g$ are $1$-periodic and from $F_t^g(x+2\pi)=F_t^g(x)+1$. Furthermore, by Proposition~\ref{prop:differentiability}, $u \mapsto x_t^g(u)$  belongs  to $\mathcal C^{3+\theta'}$ for some $\theta' <\theta$, thus $\partial_u x_t^g \in \mathcal C^{2+\theta'}$. 
Therefore,  $A_t^g$ is a $\mathcal C^1$-function.
The properties of $A_t^{g,\eps}$  follows. 
\end{proof}

We conclude this paragraph by splitting the derivative of $P_t\phi$ into two terms. $I_1$ involves the regularized function $A_t^{g,\eps}$, whereas $I_2$ is a remainder term for which we have to show that it is small with respect to $|\eps|$. 

\begin{prop}
\label{prop_split}
Under the same assumptions as in Lemma~\ref{lemme:derivee_du_semigroupe} and for every $t\in [0,T]$ 
\begin{align*}
\frac{\mathrm d}{\mathrm d\rho}_{\vert \rho=0} P_t \phi (\mu_0^{g+\rho g'  h}) 
=I_1+I_2,
\end{align*}
where
\begin{align*}
I_1&:= \frac{1}{t} \Ewb{ \int_0^1 \!\!\int_0^t \partial_\mu \phi(\mu_t^g) (x_t^g(u)) \frac{\partial_u x_t^g(u)}{\partial_u x_s^g(u)} A_s^{g,\eps}(x_s^g(u)) \mathrm ds \mathrm du};\\
I_2&:= \frac{1}{t} \Ewb{ \int_0^1 \!\!\int_0^t \partial_\mu \phi(\mu_t^g) (x_t^g(u)) \frac{\partial_u x_t^g(u)}{\partial_u x_s^g(u)} (A_s^g-A_s^{g,\eps})(x_s^g(u)) \mathrm ds \mathrm du } .
\end{align*}
\end{prop}

\begin{proof}
By definition of $(A_t^g)_{t\in [0,T]}$, 
\begin{align*}
h(u)= \frac{1}{t} \int_0^t h(u) \mathrm ds 
&= \frac{1}{t} \int_0^t \frac{\partial_u x_s^g(F_s^g(x_s^g(u))) h(F_s^g(x_s^g(u))) }{\partial_u x_s^g(u)}  \mathrm ds  = \frac{1}{t} \int_0^t \frac{ A^g_s(x_s^g(u))}{\partial_u x_s^g(u)} \mathrm ds.
\end{align*}
Therefore, equation~\eqref{formules_derivees_semigroupe} rewrites
\begin{align*}
\frac{\mathrm d}{\mathrm d\rho}_{\vert \rho=0} P_t \phi (\mu_0^{g+\rho g'  h}) 
= \Ewb{ \int_0^1 \partial_\mu \phi(\mu_t^g) (x_t^g(u)) \partial_u x_t^g(u) \left(\frac{1}{t} \int_0^t \frac{A_s^g(x_s^g(u))}{\partial_u x_s^g(u)} \mathrm ds \right)\mathrm du }.
\end{align*}
The r.h.s.\! of the latter equality is clearly equal to $I_1+I_2$. 
\end{proof}

The proof of Theorem~\ref{theo:Bismut pour phi assumptions} is now divided into three main steps. In the following two sections, we will study separately $I_1$ and $I_2$. This will lead to the  estimate   stated in Corollary~\ref{coro:estim_I}. 
In Section~\ref{sec:end_proof}, we conclude the proof by iterating the result of Corollary~\ref{coro:estim_I} over successive time intervals.

\section{Analysis of \texorpdfstring{$I_1$}{I1}}
\label{sec:analysis_I1}

In order to control $I_1$, we adapt in this section a method of proof introduced in~\cite{thalmaier97}. To follow that strategy, we  take benefit from the fact that $A_t^{g,\eps}$, in contrast to $A_t^g$, is as  regular as needed. The drawback is that the control on $I_1$ blows up when $\eps$ goes to 0, which is the reason why the explosion rate is $t^{-2-\theta}$ in Theorem~\ref{theo:Bismut pour phi assumptions} and not $t^{-1/2}$ as in~\cite{thalmaier97}.

Let us define
\begin{align*}
a_t(u)=\int_0^t  \frac{g'(u)}{\partial_u x_s^g(u)} A_s^{g,\eps}(x_s^g(u)) \mathrm ds. 
\end{align*}
Using that notation, $I_1$ rewrites as follows
\begin{align*}
I_1=\frac{1}{t} \Ewb{ \int_0^1 \partial_\mu \phi(\mu_t^g) (x_t^g(u)) \frac{\partial_u x_t^g(u)}{g'(u)} a_t(u)    \mathrm du}. 
\end{align*}
The goal of this section is to prove the following inequality:
\begin{prop}
\label{prop:Girsanov-Fourier}
Let $\phi$, $\theta$ and $f$ be as in Theorem~\ref{theo:Bismut pour phi assumptions}. 
Let $g$ and $h$ be $\mathcal G_0$-measurable random variables with values respectively in $\mathbf G^{3+\theta}$ and~$\Delta^1$. Then there is $C>0$ independent of $g$, $h$ and $\theta$ such that for every $t\in [0,T]$, for every $\eps \in (0,1)$,
\begin{align}
\label{mainestim1}
|I_1|
&\leq  C \; \frac{\| \phi \|_{L_\infty}}{\eps^{3+2\theta}\sqrt{t}}    C_1(g) \|h\|_{\mathcal C^1}.
\end{align}
where 
$C_1(g)=1+\|g'''\|_{L_4}^2 + \|g''\|_{L_\infty}^{6} + \|g'\|_{L_\infty}^8 + \left\| \frac{1}{g'} \right\|_{L_\infty}^{8}$.
\end{prop}

The proof of the  proposition is based on writing the SDE satisfied by $(Z_t^{g(u)+\rho a_t(u)})_{t\in [0,T]}$, where $Z_t^x$ is the solution to~\eqref{eq_SDE_z}. We recall this expansion, known as Kunita's theorem, in the following lemma:

\begin{lemme}
Let $f$ be of order $\alpha> \frac{3}{2}+\theta$ for some $\theta \in (0,1)$. 
Let $(\zeta_t)_{t\in [0,T]}$ be a $(\mathcal G_t)_{t\in [0,T]}$-adapted process such that $t \mapsto \zeta_t$ is absolutely continuous, $\zeta_0=0$ almost surely and $\E{\int_0^T |\zeta_t| \mathrm dt }$ is finite.
Then almost surely, for every $x\in \R$,  $t \in [0,T]$ and  $\rho \in \R$, 
\begin{align*}&
Z_t^{x+\rho \zeta_t}
=Z_0^x+ \sum_{k \in \Z} f_k \int_0^t \Re \left( e^{-ik Z_s^{x+\rho \zeta_s}} \mathrm dW^k_s \right) + \beta_t + \rho \int_0^t \partial_x Z_s^{x+\rho \zeta_s} \;\dot{\zeta}_s \;\mathrm ds.
\end{align*}
\label{lem_Kunita}
\end{lemme}

\begin{proof}
This is an application of Theorem 3.3.1 in~\cite[Chapter III]{kunita90}. 
\end{proof}

\subsection{Fourier inversion on the torus}
\label{parag:fourier}

A key ingredient in the study of $I_1$ is the following Fourier inversion of  $A_t^{g,\eps}$, which we will later use in order to apply Girsanov's theorem.

\begin{lemme}
\label{lemme:decomposition Fourier}
Let $\theta \in (0,1)$, $g \in \mathbf G^{3+\theta}$ and $f$ be of order $\alpha=\frac{7}{2}+\theta$. Fix $\eps \in (0,1)$.  
Then there is a collection of $(\mathcal G_t)_{t\in [0,T]}$-adapted $\C$-valued processes $((\lambda_t^k)_{t\in [0,T]})_{k \in \Z}$ such that for every $t\in [0,T]$, the following equality holds 
\begin{align}
\label{eq:decomposition Fourier}
 A_t^{g,\eps} (y)  = \sum_{k \in \Z}  f_k  e^{-ik y}   \lambda_t^k,
\end{align}
and such that there is a constant $C>0$ independent of $\eps$ satisfying for each $t\in [0,T]$
\begin{align}
\label{somme lambda ks carré}
\Ewb{\sum_{k \in \Z} \int_0^t   |\lambda_s^k|^2 \mathrm ds} 
\leq C \frac{t}{\eps^{6+4\theta}} \|h\|_{\mathcal C^1}^2 C_1(g)^2,
\end{align}
where $C_1(g)=1+\|g'''\|_{L_4}^2 + \|g''\|_{L_\infty}^{6} + \|g'\|_{L_\infty}^{6} + \left\| \frac{1}{g'} \right\|_{L_\infty}^{8}$. 
\end{lemme}

\begin{proof}
Fix $t\in (0,T]$. The map $y \mapsto A_t^{g,\eps} (y)$ is a $2\pi$-periodic $\mathcal C^1$-function. Therefore, by Dirichlet's Theorem, that map is equal to the sum of its Fourier series
\begin{align*}
A_t^{g,\eps} (y)  = \sum_{k \in \Z}  c_k(A_t^{g,\eps}) e^{-ik y} ,
\end{align*}
where $c_k(A):=\frac{1}{2\pi} \int_0^{2\pi} A(y) e^{iky} \mathrm dy$ for every $2\pi$-periodic function $A$ and for every $k \in \Z$. 

Let us define $\lambda_t^k:= \frac{c_k(A_t^{g,\eps})}{f_k}$. Since $(A_t^{g,\eps})_{t\in [0,T]}$ is $(\mathcal G_t)$-adapted, it is clear that for each $k \in \Z$, $(\lambda_t^k)_{t\in [0,T]}$ is also $(\mathcal G_t)$-adapted. Equality~\eqref{eq:decomposition Fourier} clearly holds true. 
Moreover, 
\begin{align*}
\sum_{k \in \Z} \int_0^t   |\lambda_s^k|^2 \mathrm ds
= \sum_{k \in \Z} \int_0^t  \left| \frac{c_k(A_s^{g,\eps})}{f_k}\right|^2 \mathrm ds.
\end{align*}
Compute the Fourier coefficient $c_k(A_s^{g,\eps})$:
\begin{align}
\label{calcul_coeff_fourier}
c_k(A_s^{g,\eps})
= c_k(A_s^g \ast \varphi_\eps) 
&=\frac{1}{2\pi} \int_0^{2\pi} \left(\int_\R A_s^g(y-x) \varphi_\eps(x) \mathrm dx \right) e^{iky} \mathrm dy  \notag\\
&=\int_\R \varphi_\eps (x) e^{ikx} \left(\frac{1}{2\pi} \int_0^{2\pi} A_s^g(y-x)e^{ik(y-x)} \mathrm dy \right) \mathrm dx \notag\\
&= c_k(A_s^g) \int_\R \varphi (x) e^{ik \eps x}\mathrm dx.
\end{align}
Since $\int_\R \varphi (x) e^{ik \eps x}\mathrm dx=\frac{1}{\sqrt{2\pi}} \int_\R e^{-x^2/2} e^{ik \eps x}\mathrm dx=e^{-k^2 \eps^2/2}$,  there is in particular $C >0$ such that for every $k\in \Z \backslash\{0\}$ and for every $\eps>0$,  $\left|\int_\R \varphi (x) e^{ik \eps x}\mathrm dx  \right|\leq \frac{C}{|k\eps|^{3+2\theta} }$. 

Moreover, $A_s^g$ is a $\mathcal C^1$-function. 
Thus there is $C$ independent of $k$ and $s$ such that for every $k\in \Z \backslash\{0\}$, $|c_k(A_s^g)| \leq \frac{C}{|k|} \|\partial_x A_s^g \|_{L_\infty}$.
Furthermore, $|c_0(A_s^{g,\eps})|=|c_0(A_s^g)| \leq \| A_s^g \|_{L_\infty}$. 
Since $f$ is of order $\alpha=\frac{7}{2}+\theta$, there is $C$ such that for every $k\in \Z \backslash\{0\}$, $\frac{1}{|f_k|}\leq C |k|^{\frac{7}{2}+\theta}$. 
Thus we have
\begin{align*}
\sum_{k \in \Z} \int_0^t \frac{|c_k(A_s^{g,\eps})|^2}{f_k^2} \mathrm ds
&= \int_0^t \frac{|c_0(A_s^{g,\eps})|^2}{f_0^2} \mathrm ds
+ \sum_{k \neq 0} \int_0^t \frac{|c_k(A_s^{g,\eps})|^2}{f_k^2} \mathrm ds \\
& \leq C \int_0^t \| A_s^g \|^2_{L_\infty} \mathrm ds
+ C \sum_{k \neq 0} \int_0^t |k|^{7+2\theta} \frac{1}{|k\eps|^{6+4\theta}} \frac{1}{|k|^2}\|\partial_x A_s^g \|^2_{L_\infty} \mathrm ds \\
& \leq \frac{C}{\eps^{6+4\theta}} \int_0^t \| A_s^g \|^2_{\mathcal C^1} \mathrm ds,
\end{align*}
because $1 \leq \frac{1}{\eps}$ and because the sum $\sum_{k \neq 0} \frac{1}{|k|^{1+2\theta}}$ converges.
Thus there is a constant $C>0$ independent of $\eps$ satisfying the   $\mathbb P^W \otimes \mathbb P^\beta$-almost surely for every $t\in [0,T]$
\begin{align}
\label{somme lambda carres}
\sum_{k \in \Z} \int_0^t   |\lambda_s^k|^2 \mathrm ds \leq \frac{C}{\eps^{6+4\theta}} \int_0^t \| A_s^g \|^2_{\mathcal C^1} \mathrm ds. 
\end{align}
Let us compute $\Ewb{\| A_s^g \|^2_{\mathcal C^1}}$. 
Recall that $\|A_s^g\|_{L_\infty} \leq \|\partial_u x_s^g\|_{L_\infty}  \|h\|_{L_\infty}$. Thus for every  $s \in [0,T]$, 
\begin{align*}
\Ewb{\left\| A_s^g \right\|^2_{L_\infty}} \leq \|h\|_{L_\infty}^2 \Ewb{\sup_{t\leq T}\|\partial_u x_t^g\|_{L_\infty}^2} \leq C \|h\|_{L_\infty}^2 (1+\|g'' \|_{L_2}^2 + \|g'\|_{L_\infty}^4),
\end{align*}
where we used inequality~\eqref{Lp norm 7} given in appendix. 
Moreover,  the derivative of $A_s^g$ is equal to:
\begin{align*}
\partial_x A_s^g (x)
&=\frac{(h' \;\partial_u x_s^g + h \;\partial^{(2)}_u x_s^g) (F_s^g(x))}{\partial_u x_s^g(F_s^g(x))} =h'  (F_s^g(x)) + \frac{h  (F_s^g(x)) \;\partial^{(2)}_u x_s^g (F_s^g(x))}{\partial_u x_s^g(F_s^g(x))} .
\end{align*}
We deduce that
\begin{align}
\label{norme 1 Asg}
\|\partial_x A_s^g \|_{L_\infty}
\leq C \|h\|_{\mathcal C^1}\left(1+\|\partial^{(2)}_u x_s^g\|_{L_\infty}\left\|\frac{1}{\partial_u x_s^g} \right\|_{L_\infty}\right).
\end{align}
Therefore, for every $s \leq T$, 
\begin{align}
\Ewb{\left\|\partial_x A_s^g \right\|^2_{L_\infty}} 
&\leq  C \|h\|_{\mathcal C^1}^2 \left(1+ \Ewb{\sup_{t \leq T}\| \partial^{(2)}_u x_t^g\|^4_{L_\infty}} +
\Ewb{\sup_{t \leq T}\left\|\frac{1}{\partial_u x_t^g} \right\|^4_{L_\infty}} \right) \notag
\\
&\leq  C \|h\|_{\mathcal C^1}^2
\left( 
1+\|g'''\|_{L_4}^4 + \|g''\|_{L_\infty}^{12} + \|g'\|_{L_\infty}^{16} + \left\| \frac{1}{g'} \right\|_{L_\infty}^{16}
\right),
\label{norme 1 Asg_bis}
\end{align}
by~\eqref{Lp norm 7} and~\eqref{Lp norm 8} and  because  $g$ belongs to $\mathbf G^{3+\theta}$ and $\alpha=\frac{7}{2}+\theta$.
We deduce that
\begin{align*}
\Ewb{\sum_{k \in \Z} \int_0^t   |\lambda_s^k|^2 \mathrm ds} 
\leq \frac{Ct}{\eps^{6+4\theta}} \|h\|_{\mathcal C^1}^2 \left( 
1+\|g'''\|_{L_4}^4 + \|g''\|_{L_\infty}^{12} + \|g'\|_{L_\infty}^{16} + \left\| \frac{1}{g'} \right\|_{L_\infty}^{16}
\right),
\end{align*}
which is inequality~\eqref{somme lambda ks carré}. 
This completes the proof of the lemma.  
\end{proof}

\begin{rem}
\label{rem:balance}
After the proof of Lemma~\ref{lemme:decomposition Fourier}, we can  now explain precisely why a regularization of $A^g$ was needed. Imagine for a while that instead of looking for the Fourier inverse of $A^{g,\eps}$ in~\eqref{eq:decomposition Fourier}, we were looking for the Fourier inverse of $A^g$. In order to prove an inequality like~\eqref{somme lambda ks carré}, we would have to show that 
\begin{align*}
\sum_{k \in \Z} (1+k^2)^\alpha \int_0^t  \left| c_k(A_s^{g})\right|^2 \mathrm ds < \infty.
\end{align*}
The latter sum converges if $\E{\int_0^t \| A_s^g \|^2_{\mathcal C^p} \mathrm ds}$ is bounded for a certain $p > 1+2\alpha$. 
In turn, if the latter expectation is bounded then for almost every $s$, $y \mapsto A_s^g(y)$ is of class $\mathcal C^p$. But we know, by definition~\eqref{def:Agt} of $A_s^g$ and by Proposition~\ref{prop:differentiability}, that $A_s^g \in \mathcal C^p$ if $h \in \mathcal C^p$, $g \in \mathcal C^{p+\theta}$ and $\alpha >p+\frac{1}{2}$. The regularity of $h$ is not a big problem, since we could simply assume higher regularity in the assumptions of Theorem~\ref{theo:Bismut pour phi assumptions}. However, it is impossible to choose $\alpha$ so that both inequalities $p > 1+2\alpha$ and $\alpha >p+\frac{1}{2}$ hold simultaneously. Regularizing $A^g$ allows to work with $A_s^{g,\eps} \in \mathcal C^p$ without having to assume that $\alpha > p +\frac{1}{2}$. 
\end{rem}

\subsection{A Bismut-Elworthy-like formula}

Let us state and prove an integration by parts formula,  close to Bismut-Elworthy formula.

\begin{prop}
\label{prop:bismut}
Let $\theta \in (0,1)$, $g \in \mathbf G^{3+\theta}$ and $f$ be of order $\alpha=\frac{7}{2}+\theta$.
For every $t \in (0,T]$, 
\begin{align}
\label{bismutI1}
I_1
=\frac{1}{t} \Ewb{\phi(\mu_t^g) \sum_{k \in \Z} \int_0^t \Re( \overline{\lambda_s^k} \mathrm dW_s^k)},
\end{align}
where $\overline{\lambda_s^{k}}$ denotes the complex conjugate of $\lambda_s^{k}$.
\end{prop}

In view of proving Proposition~\ref{prop:bismut}, let us introduce the following stopping times. 
Let $M_0$ be an integer large enough so that for every $u\in \R$, $\frac{1}{M_0}< g'(u) < M_0$. For every $M \geq M_0$, define:
\begin{align}
\tau_M^1&:= \inf\{ t\geq 0: \left\| \partial_u x_t^g  \right\|_{L_\infty} \geq M \}\wedge T; \notag \\
\tau_M^2&:= \inf\{ t\geq 0: \left\| \textstyle \frac{1}{\partial_u x_t^g(\cdot)} \right\|_{L_\infty} \geq M \}\wedge T; 
\label{def:tau2M}\\
\tau_M&:= \tau_M^1 \wedge \tau_M^2.\notag
\end{align}
Since $g \in \mathbf G^{3+\theta}$ and $f$ is of order $\alpha > \frac{7}{2}$, inequalities~\eqref{Lp norm 7} and~\eqref{Lp norm 8} imply that 
\begin{align}
\label{convergence temps d'arret}
\mathbb P^W \otimes \mathbb P^\beta \left[ \tau_M < T \right] \underset{M \to +\infty}{\longrightarrow} 0.
\end{align}

\begin{lemme}
\label{lemme:decomposition Fourier M}
Let $M \geq M_0$, $\theta \in (0,1)$, $g \in \mathbf G^{3+\theta}$ and $f$ be of order $\alpha=\frac{7}{2}+\theta$. Fix $\eps \in (0,1)$.  
 Then there is a collection of $(\mathcal G_t)_{t\in [0,T]}$-adapted $\C$-valued processes $((\lambda_t^{k,M})_{t\in [0,T]})_{k \in \Z}$ such that for every $t\in [0,T]$, the following equality holds 
\begin{align}
\label{eq:decomposition Fourier M}
\mathds 1_{\{ t \leq \tau_M \}} A_t^{g,\eps} (y)  = \sum_{k \in \Z}  f_k  e^{-ik y}   \lambda_t^{k,M},
\end{align}
and  there is a constant $C_{M,\eps}>0$ such that $\mathbb P^W \otimes \mathbb P^\beta$-almost surely,  $\sum_{k \in \Z} \int_0^T   |\lambda_t^{k,M}|^2 \mathrm dt \leq C_{M,\eps}$. 
\end{lemme}

\begin{proof}
Define for every $t\in [0,T]$, $\lambda_t^{k,M}:= \mathds 1_{\{ t \leq \tau_M \}} \frac{c_k(A_t^{g,\eps})}{f_k}= \mathds 1_{\{ t \leq \tau_M \}} \lambda_t^k$. 
Similarly as in  the proof of Lemma~\ref{lemme:decomposition Fourier}, there is a constant $C>0$ such that for every $k\in \Z \backslash\{0\}$ and for every $\eps>0$,  $\left|\int_\R \varphi (x) e^{ik \eps x}\mathrm dx  \right|\leq \frac{C}{|k\eps|^{4+2\theta} }$. Furthermore, for every $k \in \Z$, $|c_k(A_s^g)|\leq \| A_s^g \|_{L_2(\tor)}$. 
Thus we have
\begin{align*}
\sum_{k \in \Z} \int_0^T   |\lambda_t^{k,M}|^2 \mathrm dt 
&\leq C \int_0^T \mathds 1_{\{ t \leq \tau_M \}}|c_0(A_t^g)|^2 \mathrm dt
+ \sum_{k \neq 0} \int_0^T \mathds 1_{\{ t \leq \tau_M \}} |k|^{7+2\theta}\frac{1}{|k\eps|^{8+4\theta}} |c_k(A_t^g)|^2 \mathrm dt \\
& \leq \frac{C}{\eps^{8+4\theta}} \int_0^T \mathds 1_{\{ t \leq \tau_M \}} \| A_t^g \|^2_{L_2(\tor)} \mathrm dt.
\end{align*}
By Definition~\eqref{def:Agt}, for every $s\in [0,T]$, 
\begin{align*}
\mathds 1_{\{ t \leq \tau_M \}} \left\| A_t^g \right\|_{L_2(\tor)}
\leq C \mathds 1_{\{ t \leq \tau_M \}} \left\| A_t^g \right\|_{L_\infty(\tor)}
\leq C\mathds 1_{\{ t \leq \tau_M^1 \}}  \left\| h\right\|_{L_\infty} \left\| \partial_u x_t^g\right\|_{L_\infty}
\leq C M  \left\| h\right\|_{L_\infty}.
\end{align*}
Since the constant does not depend on $t$, we deduce the statement of the lemma. 
\end{proof}

Define the $(\mathcal G_t)$-adapted process $(a^M_t)_{t\in [0,T]}$ by $a^M_t= a_{t \wedge \tau_M}$, in other words: 
\begin{align}
\label{def:alpha_t}
a^M_t(u):=\int_0^t \mathds 1_{\{ s \leq \tau_M \}} \frac{g'(u)}{\partial_u x_s^g(u)} A_s^{g,\eps} (x_s^g(u)) \mathrm ds.
\end{align} 
We easily check that for every $u\in \R$, $a^M_0(u)=0$ and that $\dot{a}^M_t(u)=\frac{g'(u)}{\partial_u x_t^g(u)}  \mathds 1_{\{ t \leq \tau_M \}} A_t^{g,\eps} (x_t^g(u))$ is a $1$-periodic and  continuous function of $u\in \R$.

\begin{lemme}
\label{lemme:deux inegalites}
Let $\eps \in (0,1)$. 
For every $M \geq M_0$, there are two constants $C^a_M$ (depending on $T$, $M$, $g'$ and $h$)  and $C_{M,\eps}$ (depending on $T$, $M$, $\eps$ and $h$)  such that for every $t\in [0,T]$ 
\begin{align}
\label{atM norme infini}
 \|a_t^M\|_{L_\infty} &\leq C_M^a ; \\
\int_0^{t \wedge \tau_M}\|A_s^{g,\eps} \|^2_{\mathcal C^1} \mathrm ds &\leq C_{M,\eps}. 
\label{Atgeps norme C1}
\end{align}
\end{lemme}

\begin{proof}
By definition of $\tau_M$, $|a^M_t(u)| \leq T \|g'\|_{L_\infty} M \sup_{s \leq \tau_M} \|A_s^{g,\eps}\|_{L_\infty}$ for every $t\in [0,T]$ and  $u\in \R$. 
Since $A_s^{g,\eps}$ is $2\pi$-periodic and $A_s^{g,\eps}=A_s^g \ast \varphi_\eps$, with $\left\|   \varphi_\eps\right\|_{L_1(\R)}=1$, we have  $\left\| A_s^{g,\eps}\right\|_{L_\infty}\leq \left\| A_s^g\right\|_{L_\infty(\tor)}$. Recall that by definition~\eqref{def:Agt}, $\left\| A_s^g\right\|_{L_\infty(\tor)} \leq \left\| h\right\|_{L_\infty} \left\| \partial_u x_s^g\right\|_{L_\infty}$. We deduce that 
\begin{align}
\label{la}
\mathds 1_{\{s \leq \tau_M \}}  \left\|  A_s^{g,\eps}\right\|_{L_\infty(\tor)}
 \leq\left\| h\right\|_{L_\infty} \mathds 1_{\{s \leq \tau_M^1 \}}\left\| \partial_u x_s^g\right\|_{L_\infty}
\leq M\left\| h\right\|_{L_\infty}. 
\end{align}
Therefore, inequality~\eqref{atM norme infini} holds with $C^a_M:= T \|g'\|_{L_\infty} M^2 \|h\|_{L_\infty}$.

For every $t\in [0,T]$,  $ \partial_x A_t^{g,\eps}= A_t^g \ast \partial_x  \varphi_\eps $. Since $\left\| \partial_x  \varphi_\eps\right\|_{L_1(\R)}\leq \frac{C}{\eps}$, we obtain
\begin{align}
\label{lb}
\mathds 1_{\{t \leq \tau_M \}}  \left\| \partial_x A_t^{g,\eps}\right\|_{L_\infty(\tor)} \leq \frac{C}{\eps}\left\| h\right\|_{L_\infty} \mathds 1_{\{t \leq \tau_M^1 \}}\left\| \partial_u x_t^g\right\|_{L_\infty}
\leq \frac{CM}{\eps }\left\| h\right\|_{L_\infty}. 
\end{align}
It follows from~\eqref{la} and~\eqref{lb} that   $\int_0^{t \wedge \tau_M}\|A_s^{g,\eps} \|^2_{\mathcal C^1} \mathrm ds \leq  T \frac{C}{\eps^2}  M^2 \|h\|_{L_\infty}^2$ for every $t\in [0,T]$, whence we obtain~\eqref{Atgeps norme C1}. 
\end{proof}

Using the constant $C_M^a$ appearing in~\eqref{atM norme infini}, we define $\rho_0:=\rho_0(M)=\frac{1}{2 C^a_M}$. 
The following lemma makes use of Kunita's expansion.

\begin{lemme}
\label{lemme:comparaison Z Y}
Let $f$ be of order $\alpha=\frac{7}{2}+\theta$. 
Define the auxiliary process $(Y_t^{\rho,M})_{t\in [0,T]}$ as the solution to:
\begin{align}
\label{comparaison Y}
Y_t^{\rho,M} (u) =  g(u) +\sum_{k \in \Z} f_k \int_0^t \Re \left(e^{-ik Y_s^{\rho,M}(u)} \mathrm dW^k_s  \right)+ \beta_t  + \rho \int_0^t \mathds 1_{\{ s \leq \tau_M \}} A_s^{g,\eps} (Y_s^{\rho,M}(u)) \mathrm ds. 
\end{align}
Then there exists  $C$ depending on $M$, $f$, $g'$, $h$, $T$  and $\eps$ such that for every $\rho \in (-\rho_0,\rho_0)$ and for every $t\in [0,T]$,  
\begin{align}
\Ewb{\int_0^1  |Z_t^{g(u)+\rho a_t^M(u)} - Y_t^{\rho,M} (u)|^2 \mathrm du}^{1/2}\leq C |\rho|^{5/4}.
\end{align}
\end{lemme}

\begin{proof}
Fix $u\in \R$ and write the equation satisfied by $(Z_t^{g(u)+\rho a_t^M(u)})_{t\in [0,T]}$. We apply Kunita's expansion (Lemma~\ref{lem_Kunita}) with $(\zeta_t)_{t\in [0,T]}:=(a_t^M(u))_{t\in [0,T]}$, $x=g(u)$ and $\zeta_t=a_t^M(u)$.  
By inequality~\eqref{atM norme infini}, we have $\E{\int_0^T |\zeta_t|\mathrm dt} \leq C_M$.   Thus   $\mathbb P^W \otimes \mathbb P^\beta$-almost surely, 
\begin{align}
Z_t^{g(u)+\rho a^M_t(u)}
&=Z_0^{g(u)}+ \sum_{k \in \Z} f_k \int_0^t \Re \left( e^{-ik Z_s^{g(u)+\rho a^M_s(u)}} \mathrm dW^k_s \right) + \beta_t + \rho \int_0^t \partial_x Z_s^{g(u)+\rho a_s^M(u)} \dot{a}_s(u) \mathrm ds \notag\\
&=g(u)+ \sum_{k \in \Z} f_k \int_0^t \Re \left( e^{-ik Z_s^{g(u)+\rho a_s^M(u)}} \mathrm dW^k_s \right) + \beta_t +\rho \int_0^t \mathds 1_{\{ s \leq \tau_M \}} A_s^{g,\eps} (Z_s^{g(u)}) \mathrm ds \notag \\
&\quad + \rho \int_0^t (\partial_x Z_s^{g(u)+\rho a_s^M(u)}-\partial_x Z_s^{g(u)}  )\frac{g'(u)}{\partial_u x_s^g(u)} \mathds 1_{\{ s \leq \tau_M \}}A_s^{g,\eps} (x_s^g(u)) \mathrm ds,
\label{comparaison Z}
\end{align}
where we used the identities~\eqref{egalite Zx} and~\eqref{liens entre les dérivées}.

Comparing equation~\eqref{comparaison Z} with equation~\eqref{comparaison Y} satisfied by $(Y_t^{\rho,M} (u))_{t\in [0,T]}$, we have  for every $t\in [0,T]$, 
\begin{align}
\label{eq:I1I2I3}
\Ewb{\int_0^1  |Z_t^{g(u)+\rho a_t^M(u)} - Y_t^{\rho,M} (u)|^2 \mathrm du} \leq 3(E_1+ \rho^2 E_2+\rho^2 E_3),
\end{align}
where
\begin{align*}
E_1&:=  \Ewb{\int_0^1  \Bigg|\sum_{k \in \Z}\int_0^t  f_k \Re \left((e^{-ik Z_s^{g(u)+\rho a_s^M(u)}}-e^{-ik Y_s^{\rho,M}(u)}) \mathrm dW^k_s  \right)\Bigg|^2 \mathrm du} ;   \\
E_2&:=  \Ewb{\int_0^1 \left|  \int_0^t \mathds 1_{\{ s \leq \tau_M \}}  (A_s^{g,\eps} (Z_s^{g(u)}) -  A_s^{g,\eps} (Y_s^{\rho,M}(u)) )  \mathrm ds  \right|^2 \mathrm du}  ;  \\
E_3&:=  \Ewb{\int_0^1 \left|  \int_0^t   (\partial_x Z_s^{g(u)+\rho a_s^M(u)}-\partial_x Z_s^{g(u)}  )\frac{g'(u)}{\partial_u x_s^g(u)} \mathds 1_{\{ s \leq \tau_M \}} A_s^{g,\eps} (x_s^g(u))    \mathrm ds \right|^2 \mathrm du}    . 
\end{align*}

\textbf{Control on $E_1$.} By Itô's isometry and since $y \mapsto e^{-iky}$ is $k$-Lipschitz, 
\begin{align}
\label{eq:I1}
E_1 
&\leq  \Ewb{\int_0^1 \!\!\int_0^t  \sum_{k \in \Z} f_k^2 k^2  \left|Z_s^{g(u)+\rho a_s^M(u)} -  Y_s^{\rho,M}(u)\right|^2 \mathrm ds  \mathrm du} \notag\\
&\leq C   \int_0^t \Ewb{\int_0^1   \left|Z_s^{g(u)+\rho a_s^M(u)} -  Y_s^{\rho,M}(u)\right|^2 \mathrm du} \mathrm ds ,
\end{align}
because $\sum_{k \in \Z}  f_k^2 k^2 <+\infty$, since $\alpha> \frac{3}{2}$. 

\textbf{Control on $E_2$.} By~\eqref{Atgeps norme C1},  there is $C>0$ such that $\int_0^{t \wedge \tau_M}  \left\| \partial_x A_s^{g,\eps}\right\|_{L_\infty(\tor)}^2 \mathrm ds \leq C$. It follows from Cauchy-Schwarz inequality that 
\begin{align*}
E_2 
& \leq  \Ewb{\int_0^1   \int_0^{t \wedge \tau_M}  \left\| \partial_x A_s^{g,\eps}\right\|_{L_\infty(\tor)}^2 \mathrm ds  \; \int_0^t |Z_s^{g(u)} -  Y_s^{\rho,M}(u)|^2  \mathrm ds   \mathrm du} \\
&\leq C \int_0^t \Ewb{\int_0^1    \left|Z_s^{g(u)+\rho a_s^M(u)} -  Y_s^{\rho,M}(u)\right|^2 \mathrm du} \mathrm ds \\
&\quad  +  C
\int_0^1 \!\! \int_0^t \Ewb{\left| Z_s^{g(u)} - Z_s^{g(u)+\rho a_s^M(u)} \right|^2  }\mathrm ds \mathrm du.
\end{align*}
Moreover, by inequality~\eqref{atM norme infini}, $|\rho a_t^M(u)| \leq  \rho_0 C^a_M =\frac{1}{2}$  for every $t\in [0,T]$, $u\in \R$ and $\rho \in (-\rho_0,\rho_0)$. 
Fix  $u\in [0,1]$. Let $J_u$ be the interval $[g(u)-\frac{1}{2}, g(u)+\frac{1}{2}]$. 
By inequality~\eqref{différence Ztx Zty} and by Kolmogorov's Lemma (see~\cite[p.26, Thm I.2.1]{revuzyor13}), it follows that, up to considering a modification of the process $(Z_t^x)_{x\in J_u}$, there is a constant $C_{\operatorname{Kol}}$ independent of $u$ such that 
\begin{align*}
\Ewb{\sup_{x,y \in J_u, x\neq y}  \frac{\sup_{t\leq T} |Z_t^x - Z_t^y |^2}{|x-y|^{1/2}}  } \leq C_{\operatorname{Kol}}.
\end{align*}
We deduce that for every $\rho \in (-\rho_0,\rho_0)$, 
\begin{multline*}
\Ewb{\left| Z_s^{g(u)} - Z_s^{g(u)+\rho a_s^M(u)} \right|^2  } \\
\begin{aligned}
&\leq\mathbb E^W \mathbb E^\beta \bigg[
\mathds 1_{\{\rho  a_t^M(u)\neq 0  \}}  \frac{\sup_{t \leq T} |  Z_t^{g(u)} - Z_t^{g(u)+\rho a_t^M(u)}  |^2 }{|\rho a_t^M(u)|^{1/2}}  |\rho  a_t^M(u)|^{1/2}\bigg]
\leq C  C_{\operatorname{Kol}} |\rho|^{1/2},
\end{aligned}
\end{multline*}
where the constants are independent of $s$ and $u$. 
We conclude that for every $\rho \in (-\rho_0,\rho_0)$
\begin{align}
\label{eq:I2}
E_2 \leq C\int_0^t \Ewb{\int_0^1    \left|Z_s^{g(u)+\rho a_s^M(u)} -  Y_s^{\rho,M}(u)\right|^2 \mathrm du} \mathrm ds + C |\rho|^{1/2}.
\end{align}

\textbf{Control on $E_3$.} By definition~\eqref{def:tau2M} of $\tau^2_M$, for every $s \leq \tau_M$, $\left\| \textstyle \frac{1}{\partial_u x_s^g(\cdot)} \right\|_{L_\infty} \leq M$. Thus
\begin{align*}
E_3 
&\leq \|g'\|_{L_\infty} M \Ewb{\int_0^1 \!\!  \int_0^t   \left|\partial_x Z_s^{g(u)+\rho a_s^M(u)}-\partial_x Z_s^{g(u)}  \right|^2 \mathrm ds \; \int_0^{t\wedge \tau_M} \|A_s^{g,\eps}\|_{L_\infty} ^2    \mathrm ds  \mathrm du} \\
&\leq C \Ewb{\int_0^1 \!\!  \int_0^t   \left|\partial_x Z_s^{g(u)+\rho a_s^M(u)}-\partial_x Z_s^{g(u)}  \right|^2 \mathrm ds  \mathrm du},
\end{align*}
where the last inequality follows from~\eqref{Atgeps norme C1}.
By inequality~\eqref{différence pZtx pZty} and the fact that $f$ is of order $\alpha> \frac{5}{2}$, we can apply as before Kolmogorov's Lemma on $\partial_x Z$ instead of $Z$. We get for every $\rho \in (-\rho_0,\rho_0)$, 
\begin{align*}
\Ewb{   \left|  \partial_x Z_s^{g(u)+\rho a_s^M(u)}-\partial_x Z_s^{g(u)}   \right|^2 } \leq  C C_{\operatorname{Kol}} |\rho|^{1/2}.
\end{align*}
Therefore, for every $\rho \in (-\rho_0,\rho_0)$, $E_3 \leq C |\rho|^{1/2}$. 

\textbf{Conclusion.} Putting together the last inequality with~\eqref{eq:I1I2I3}, \eqref{eq:I1} and~\eqref{eq:I2}, we obtain for every $t\in [0,T]$ 
\begin{multline*}
\Ewb{\int_0^1 |Z_t^{g(u)+\rho a_t^M(u)} - Y_t^{\rho,M} (u)|^2 \mathrm du}\\
\leq  C \int_0^t \Ewb{\int_0^1    \left|Z_s^{g(u)+\rho a_s^M(u)} -  Y_s^{\rho,M}(u)\right|^2 \mathrm du} \mathrm ds +C |\rho|^{5/2}. 
\end{multline*}
By Gronwall's inequality, the proof of Lemma~\ref{lemme:comparaison Z Y} is complete. 
\end{proof}

Finally, the following lemma states a Bismut-Elworthy formula. Remark that the only difference with the  formula of  Proposition~\ref{prop:bismut} is the localization by~$\tau_M$. 
\begin{lemme}
\label{lemme 31}
Let $\theta \in (0,1)$, $g \in \mathbf G^{3+\theta}$ and $f$ be of order $\alpha=\frac{7}{2}+\theta$.
For every $M \geq M_0$ and for every $t\in[0,T]$, 
\begin{align}
\label{egalite localisee}
 \Ewb{\int_0^1 \partial_\mu \phi(\mu_t^g) (x_t^g(u)) \frac{\partial_u x_t^g(u)}{g'(u)} a_t^M(u) \mathrm du}
=\Ewb{\phi(\mu_t^g)  \sum_{k \in \Z} \int_0^t \Re (\overline{\lambda_s^{k,M}} \mathrm dW_s^k)  }.
\end{align} 
\end{lemme}

\begin{proof}
Take the real part of equality~\eqref{eq:decomposition Fourier M}  with $y=Y_s^{\rho,M}(u)$. Recall that $A_s^{g,\eps}$ and $f_k$ are real-valued.  We obtain  for every $M \geq M_0$, for every $u\in \R$, for every $\rho \in (-\rho_0(M),\rho_0(M))$ and for every $s \in [0,T]$, 
\begin{align*}
\mathds 1_{\{ s \leq \tau_M \}} A_s^{g,\eps} (Y_s^{\rho,M}(u))  = \sum_{k \in \Z}  f_k \Re \left(e^{-ik Y_s^{\rho,M}(u)}   \lambda_s^{k,M} \right).
\end{align*}
Thus, we rewrite equality~\eqref{comparaison Y} in the following way:  for every $t\in [0,T]$ 
\begin{align*}
Y_t^{\rho,M} (u) =  g(u) +\sum_{k \in \Z}\int_0^t  f_k \Re \left(e^{-ik Y_s^{\rho,M}(u)} (\mathrm dW^k_s  +\rho \lambda_s^{k,M} \mathrm ds)  \right)+ \beta_t.
\end{align*}
Recall that  $\lambda_s^{k,M}$ is complex-valued. Define for every $t\in [0,T]$ 
\begin{align*}
\mathcal E^\rho_t :=\exp \Bigg(-\rho \sum_{k \in \Z} \int_0^t \Re (\overline{\lambda_s^{k,M}} \mathrm dW_s^k)  -\frac{\rho^2}{2}  \sum_{k \in \Z} \int_0^t |\lambda_s^{k,M}|^2 \mathrm ds   \Bigg).
\end{align*}
Recall that by Lemma~\ref{lemme:decomposition Fourier M},  there is a constant $C_{M,\eps}>0$ such that $\mathbb P^W \otimes \mathbb P^\beta$-almost surely,  $\sum_{k \in \Z} \int_0^T   |\lambda_s^{k,M}|^2 \mathrm ds \leq C_{M,\eps}$. It follows from Novikov's condition that the process $(\mathcal E_t^\rho)_{t\in [0,T]}$ is a $\mathbb P^W \otimes \mathbb P^\beta$-martingale.
Let $\mathbb P^\rho$ be the probability measure on $\Omega^W \times \Omega^\beta$ such that $\mathbb P^\rho$ is absolutely continuous with respect to $\mathbb P^W \otimes\mathbb P^\beta$ with  density $\frac{\mathrm d \mathbb P^\rho}{\mathrm d (\mathbb P^W \otimes\mathbb P^\beta)}=\mathcal E^\rho_T$.
By Girsanov's Theorem, $((W_t^k +\rho \lambda_t^{k,M})_{t\in [0,T]})_{k \in \Z}$ is a collection of independent $\mathbb P^\rho$-Brownian motions, independent of $(\beta,\mathcal G_0)$. By uniqueness in law of equation~\eqref{eq_SDE_diff_torus}, the law of $(Y_t^{\rho,M})_{t\in [0,T]}$ under $\mathbb P^\rho$ is equal to the law of $(x_t^g)_{t\in [0,T]}$ under $\mathbb P^W \otimes \mathbb P^\beta$.

Fix $t\in [0,T]$. 
Recall that $\widehat{\phi}(Y):= \phi(\mathcal L_{[0,1] \times \Omega^\beta}(Y))$ for  every $Y \in L_2([0,1]\times \Omega^\beta )$. 
Then
\begin{align*}
\Ewb{\widehat{\phi}(Y_t^{\rho,M} ) \mathcal E_t^\rho }= \Ewb{\widehat{\phi}(Y_t^{\rho,M} ) \frac{\mathrm d \mathbb P^\rho}{\mathrm d(\mathbb P^W \otimes\mathbb P^\beta)} } = \Ewb{\widehat{\phi}(x_t^g)}.
\end{align*}
The r.h.s.\! does not depend on $\rho$, so we have
\begin{align}
\label{derivee_nulle}
\frac{\mathrm d}{\mathrm d\rho}_{\vert \rho=0} \Ewb{\widehat{\phi}(Y_t^{\rho,M} ) \mathcal E_t^\rho } =0.
\end{align}
Let us now prove that $\frac{\mathrm d}{\mathrm d\rho}_{\vert \rho=0} \Ewb{\widehat{\phi}(Z_t^{g+\rho a_t^M}) \mathcal E_t^\rho } =0$. 
By assumption $(\phi 2)$, $\widehat{\phi}$ is a Lipschitz-continuous function. 
By Lemma~\ref{lemme:comparaison Z Y}, we have for every $\rho \in (-\rho_0, \rho_0)$
\begin{multline*}
\left|\Ewb{\widehat{\phi}(Z_t^{g+\rho a_t^M}) \mathcal E_t^\rho } -\Ewb{\widehat{\phi}(Y_t^{\rho,M} ) \mathcal E_t^\rho }  \right| \\
\begin{aligned}
&\leq \mathbb E^W \left[|\widehat{\phi}(Z_t^{g+\rho a_t^M})  - \widehat{\phi}(Y_t^{\rho,M} )  |^2\right]^{1/2} \Ewb{\left| \mathcal E_t^\rho\right|^2}^{1/2} \\
&\leq  \|\widehat{\phi}\|_{\operatorname{Lip}} \mathbb E^W \left[\|Z_t^{g+\rho a_t^M}  - Y_t^{\rho,M}\|_{{L_2([0,1] \times \Omega^\beta)}}^2\right]^{1/2} \Ewb{\left| \mathcal E_t^\rho \right|^2}^{1/2} \\
&\leq C_{M,\eps} |\rho|^{5/4}\Ewb{\left| \mathcal E_t^\rho \right|^2}^{1/2}.
\end{aligned}
\end{multline*}
Moreover, recalling that $\sum_{k \in \Z} \int_0^T   |\lambda_s^{k,M}|^2 \mathrm ds \leq C_{M,\eps}$ (see Lemma~\ref{lemme:decomposition Fourier M})
\begin{align*}
\Ewb{\left| \mathcal E_t^\rho \right|^2}
&\leq e^{\rho_0^2 C_{M,\eps}} \; \Ewb{ \textstyle \exp \Big(\! -2\rho \sum_{k \in \Z} \int_0^t \Re (\overline{\lambda_s^{k,M}} \mathrm dW_s^k)  -\frac{(2\rho)^2}{2}  \sum_{k \in \Z} \int_0^t |\lambda_s^{k,M}|^2 \mathrm ds   \Big)} \\
&=e^{\rho_0^2 C_{M,\eps}},
\end{align*}
since the exponential term is a $\mathbb P^W \otimes \mathbb P^\beta$-martingale. Therefore, 
\begin{align}
\label{derivee_nulle_bis}
\left|\Ewb{\widehat{\phi}(Z_t^{g+\rho a_t^M}) \mathcal E_t^\rho } -\Ewb{\widehat{\phi}(Y_t^{\rho,M} ) \mathcal E_t^\rho }  \right|
\leq C_{M,\eps}|\rho|^{5/4}. 
\end{align}
It follows from~\eqref{derivee_nulle} and~\eqref{derivee_nulle_bis} that $\frac{\mathrm d}{\mathrm d\rho}_{\vert \rho=0} \Ewb{\widehat{\phi}(Z_t^{g+\rho a_t^M}) \mathcal E_t^\rho } =0$. 
By~\eqref{frechet_derivative}, we compute
\begin{align*}
0&= \frac{\mathrm d}{\mathrm d\rho}_{\vert \rho=0} \Ewb{\widehat{\phi}(Z_t^{g+\rho a_t^M}) \mathcal E_t^\rho }\\
&= \Ewb{\int_0^1 D\widehat{\phi}(Z_t^g)_u \; \partial_x Z_t^{g(u)} a_t^M(u) \mathrm du}
- \Ewb{ \widehat{\phi}(Z_t^g)  \sum_{k \in \Z} \int_0^t \Re (\overline{\lambda_s^{k,M}} \mathrm dW_s^k)  } \\
&= \Ewb{\int_0^1 \partial_\mu \phi(\mu_t^g) (x_t^g(u)) \frac{\partial_u x_t^g(u)}{g'(u)} a_t^M(u) \mathrm du}
- \Ewb{ \phi(\mu_t^g)  \sum_{k \in \Z} \int_0^t \Re (\overline{\lambda_s^{k,M}} \mathrm dW_s^k)  }.
\end{align*}
We used Proposition~\ref{prop:kunita_identities} for the last equality. Therefore, equality~\eqref{egalite localisee} holds true. 
\end{proof}

Finally, we prove Proposition~\ref{prop:bismut}. 

\begin{proof}[Proof (Proposition~\ref{prop:bismut})]
We want to prove
\begin{align*}
  \Ewb{ \int_0^1 \partial_\mu \phi(\mu_t^g) (x_t^g(u)) \frac{\partial_u x_t^g(u)}{g'(u)} a_t(u)    \mathrm du}
=\Ewb{ \phi(\mu_t^g)  \sum_{k \in \Z} \int_0^t \Re (\overline{\lambda_s^{k}} \mathrm dW_s^k)  }.
\end{align*}
which is equivalent to~\eqref{bismutI1}.
In order to obtain that equality, it is sufficient to pass to the limit when $M \to +\infty$ in~\eqref{egalite localisee}.
Recall that by~\eqref{convergence temps d'arret}, 
$\mathbb P^W \otimes \mathbb P^\beta \left[ \tau_M < T \right] \to_{M \to +\infty} 0$. Since $\{ \tau_M < T\}_{M \geq M_0}$ is a non-increasing sequence of events, it follows that $\mathbb P^W \otimes \mathbb P^\beta$-almost surely, $\mathds 1_{\{ t \leq \tau_M \}} \to \mathds 1_{\{ t \leq T \}}$.
Thus, it only remains to prove uniform integrability of both members of equality~\eqref{egalite localisee}. Precisely, we want to prove:
\begin{align}
\sup_{M \geq M_0}  \Ewb{\int_0^1 \bigg( \partial_\mu \phi(\mu_t^g) (x_t^g(u)) \frac{\partial_u x_t^g(u)}{g'(u)} a_t^M(u) \bigg)^{3/2} \mathrm du} &<+\infty ;
\label{unifinteg 1} \\
\sup_{M \geq M_0} \Ewb{\bigg(\phi(\mu_t^g)  \sum_{k \in \Z} \int_0^t \Re (\overline{\lambda_s^{k,M}} \mathrm dW_s^k)\bigg)^{2} } &<+\infty.
\label{unifinteg 2}
\end{align}

\textbf{Proof of~\eqref{unifinteg 1}. }
For every $M \geq M_0$, by Hölder's inequality
\begin{multline*}
\Ewb{\int_0^1 \bigg( \partial_\mu \phi(\mu_t^g) (x_t^g(u)) \frac{\partial_u x_t^g(u)}{g'(u)} a_t^M(u) \bigg)^{3/2} \mathrm du} \\
\leq \Ewb{\int_0^1 \bigg( \partial_\mu \phi(\mu_t^g) (x_t^g(u))  \bigg)^2 \mathrm du}^{3/4}\Ewb{\int_0^1 \bigg|  \frac{\partial_u x_t^g(u)}{g'(u)} a_t^M(u) \bigg|^6 \mathrm du}^{1/4}.
\end{multline*}
By assumption $(\phi 2)$, $\Ewb{\int_0^1 \big( \partial_\mu \phi(\mu_t^g) (x_t^g(u))  \big)^2 \mathrm du}$ is  bounded. 
Moreover, by inequality~\eqref{Lp norm 1}, 
\begin{align*}
\Ewb{\int_0^1 \bigg|  \frac{\partial_u x_t^g(u)}{g'(u)} a_t^M(u) \bigg|^6 \mathrm du}  
\leq C \left\| \frac{1}{g'} \right\|^6_{L_\infty}  \|g'\|^6_{L_{12}} \; \Ewb{\int_0^1 \big|   a_t^M(u) \big|^{12} \mathrm du}^{1/2}. 
\end{align*}
By definition~\eqref{def:alpha_t} of $a_t^M$, we have
\begin{align*}
\Ewb{\int_0^1   |a_t^M(u) |^{12} \mathrm du}
&\leq \Ewb{\int_0^1 T^{11} \int_0^t  \left| \frac{g'(u)}{\partial_u x_s^g(u)} A_s^{g,\eps} (x_s^g(u))  \right|^{12} \mathrm ds \mathrm du} .
\end{align*}
Remark that for every $s\in [0,T]$, $\|A_s^{g,\eps}\|_{L_\infty}\leq \|A_s^g\|_{L_\infty} \leq \|\partial_u x_s^g\|_{L_\infty}  \|h\|_{L_\infty}$. 
Thus
\begin{multline*}
\Ewb{\int_0^1   |a_t^M(u) |^{12} \mathrm du}
\leq C \|g'\|^{12}_{L_\infty}  \|h\|^{12}_{L_\infty}  \Ewb{\sup_{t\leq T}\|\partial_u x_t^g\|^{12}_{L_\infty} \sup_{t\leq T}\left\| \textstyle \frac{1}{\partial_u x_t^g(\cdot)} \right\|^{12}_{L_\infty}} \\
\leq C \|g'\|^{12}_{L_\infty}  \|h\|^{12}_{L_\infty}  \Ewb{\sup_{t\leq T}\|\partial_u x_t^g\|^{24}_{L_\infty}}^{1/2} \Ewb{ \sup_{t\leq T}\left\| \textstyle \frac{1}{\partial_u x_t^g(\cdot)} \right\|^{24}_{L_\infty}}^{1/2} \leq C,
\end{multline*}
where the constant $C$ does not depend on $M$.  The last inequality is obtained by inequalities~\eqref{Lp norm 7} and~\eqref{Lp norm 8}, because $g \in \mathbf G^{2+\theta}$ and $\alpha>\frac{5}{2}+\theta$. 
We deduce~\eqref{unifinteg 1}.

\textbf{Proof of~\eqref{unifinteg 2}. }
For every $M \geq M_0$ and $t\in [0,T]$
\begin{align*}
\Ewb{\bigg(\phi(\mu_t^g)  \sum_{k \in \Z} \int_0^t \Re (\overline{\lambda_s^{k,M}} \mathrm dW_s^k)\bigg)^{2} }
&\leq \|\phi\|^2_{L_\infty} \Ewb{\sum_{k \in \Z} \int_0^t |\lambda_s^{k,M}|^2  \mathrm ds  }\notag\\
&\leq \|\phi\|^2_{L_\infty} \Ewb{\sum_{k \in \Z} \int_0^t |\lambda_s^k|^2  \mathrm ds  },
\end{align*}
since $\lambda_s^{k,M}= \mathds 1_{\{ s \leq \tau_M \}} \lambda_s^k$. 
By inequality~\eqref{somme lambda carres}, $\Ewb{\sum_{k \in \Z} \int_0^t |\lambda_s^k|^2  \mathrm ds  }$ is bounded, so we deduce~\eqref{unifinteg 2}. 
It completes the proof of Proposition~\ref{prop:bismut}. 
\end{proof}

\subsection{Conclusion of the analysis}

Putting together Lemma~\ref{lemme:decomposition Fourier} and Proposition~\ref{prop:bismut}, we conclude the  proof of Proposition~\ref{prop:Girsanov-Fourier}. 
\begin{proof}[Proof (Proposition~\ref{prop:Girsanov-Fourier})]
By Cauchy-Schwarz inequality applied to~\eqref{bismutI1}, 
\begin{align*}
|I_1| \leq 
 \frac{1}{t}\| \phi \|_{L_\infty} \Ewb{\sum_{k \in \Z} \int_0^t |\lambda_s^k |^2 \mathrm ds}^{1/2} 
 \leq  \frac{C}{t}\| \phi \|_{L_\infty} \frac{\sqrt{t}}{\eps^{3+2\theta}} C_1(g) \|h\|_{\mathcal C^1},
\end{align*}
where we applied inequality~\eqref{somme lambda ks carré}. 
\end{proof}

\begin{rem}
\label{rem_approx_bismut}
Note that we have proved a Bismut-Elworthy-Li integration by parts formula up to a remainder term. Indeed, by propositions~\ref{prop_split} and~\ref{prop:bismut}, we proved that
\begin{align*}
\frac{\mathrm d}{\mathrm d\rho}_{\vert \rho=0} P_t \phi (\mu_0^{g+\rho g'  h}) 
=\frac{1}{t} \Ewb{\phi(\mu_t^g) \sum_{k \in \Z} \int_0^t \Re( \overline{\lambda_s^k} \mathrm dW_s^k)}
+I_2,
\end{align*}
where it should be recalled that $((\lambda_t^k)_{t\in [0,T]})_{k \in \Z}$, defined by~\eqref{eq:decomposition Fourier}, depends on $\eps$. In the next section, we prove that $I_2$ is of order $\mathcal O(\eps)$. 
\end{rem}

\section{Analysis of \texorpdfstring{$I_2$}{I2}}
\label{sec:analysis_I2}

In this section, we look for an upper bound of $|I_2|$. Define $H_s^{g,\eps}(u):= \frac{(A_s^g-A_s^{g,\eps})(x_s^g(u))}{\partial_u x_s^g(u)}$. Then $I_2$ rewrites as follows

\begin{align}
\label{I_2_avec_H}
I_2= \frac{1}{t} \Ewb{ \int_0^1 \!\!\int_0^t \partial_\mu \phi(\mu_t^g) (x_t^g(u)) \;\partial_u x_t^g(u) \; H_s^{g,\eps}(u) \;\mathrm ds \mathrm du }.
\end{align}
Moreover, let $(K_t^{g,\eps})_{t\in [0,T]}$ be the process defined by:
\begin{align}
\label{defin:Ktgeps}
K_t^{g,\eps}(u):=
\int_0^t H_s^{g,\eps}(u) \frac{1}{t-s}\int_s^t \partial_u x_r^g(u) \mathrm d\beta_r \; \mathrm ds.
\end{align}
We also introduce the  notation $\left[ \frac{\delta \psi}{\delta m} \right]$, denoting the zero-average linear functional derivative.  For every $\mu \in \mathcal P_2(\R)$ and $v \in \R$, 
\begin{align}
\label{notation moyenne}
\left[ \frac{\delta \psi}{\delta m} \right] (\mu) (v):= \frac{\delta \psi}{\delta m} (\mu) (v) - \int_\R \frac{\delta \psi}{\delta m} (\mu) (v') \mathrm d\mu (v').
\end{align}

The main result of this section is the following proposition. 
\begin{prop}
\label{prop:estim 2}
Under the same assumptions as Proposition~\ref{prop:Girsanov-Fourier}, there is $C>0$ independent of $g$, $h$ and $\theta$ such that for every $t\in [0,T]$ and $\eps \in (0,1)$,
\begin{align}
\label{mainestim2}
|I_2|
&\leq  \frac{C}{\sqrt{t}}\;\eps \|h\|_{\mathcal C^1}   C_2(g) \;
\Ewb{\left| \int_0^1 \left[ \frac{\delta \phi}{\delta m} \right] (\mu_t^g) (x_t^g(u)) \frac{K_t^{g,\eps} (u)}{\left\| K_t^{g,\eps} \right\|_{L_\infty}}\mathrm du\right|^2   }^{1/2}.
\end{align}
where $C_2(g) =   1+ \|g'''\|_{L_8}^3 + \|g''\|_{L_\infty}^{12}+\|g'\|_{L_\infty}^{12}+ \left\| \frac{1}{g'} \right\|^{24}_{L_\infty}$. 
\end{prop}

\subsection{Progressive measurability}
\label{parag:prog_measurability}

We start by showing  that Yamada-Watanabe Theorem applies here, \textit{i.e.\!}  that the processes $(x_t^g)_{t\in [0,T]}$, $(\mu_t^g)_{t\in [0,T]}$ and $(H_t^{g,\eps})_{t\in [0,T]}$ can all be written as progressively measurable functions of $u$ and the noises $(W^k)_{k \in \Z}$ and $\beta$. 

For that purpose, let $(\Theta, \mathcal B(\Theta))$ be the canonical space defined by $\Theta= \mathcal C([0,T], \C)^\Z \times \mathcal C([0,T], \R)$ and $\mathcal B(\Theta)= \mathcal B(\mathcal C([0,T], \C)^\Z) \otimes \mathcal B(\mathcal C([0,T], \R))$. 
Let $\mathbf P$ be the probability measure on $(\Theta, \mathcal B(\Theta))$ defined as  the distribution of $((W^k)_{k \in \Z}, \beta)$ on $\Omega^W \times \Omega^\beta$. 
Let $\mathcal B_t(\mathcal C([0,T], \R)):=\sigma(\mathbf x(s);0\leq s \leq t)$; in other words the process $(\mathcal B_t(\mathcal C([0,T], \R)))_{t\in [0,T]}$ is the canonical filtration on $(\mathcal C([0,T], \R), \mathcal B(\mathcal C([0,T], \R) ))$. Similarly, let $(\mathcal B_t(\mathcal C([0,T], \C)^\Z))_{t\in [0,T]}$ be the canonical filtration on $(\mathcal C([0,T], \C)^\Z, \mathcal B(\mathcal C([0,T], \C)^\Z ))$. 
Let $(\widehat{\mathcal B}_t(\Theta))_{t\in [0,T]}$ be the augmentation of the filtration $(\mathcal B_t(\mathcal C([0,T], \C)^\Z)\otimes \mathcal B_t(\mathcal C([0,T], \R) ))_{t\in [0,T]}$ by the null sets of $\mathbf P$. 
Those notations are inspired by the textbook~\cite[pp.308-311]{karatzasshreve91}.  
We denote  elements of $\Theta$ in bold, e.g. $((\mathbf {w}^k)_{k\in \Z}, \mathbf b) \in \Theta$.

\begin{lemme}
\label{lemma:prog_measurability}
Let $\theta$, $\eps$, $g$, $f$ and $h$ be as in Proposition~\ref{prop:estim 2}. 
Then 
\begin{itemize}
\item[$(a)$] there is a $\mathcal B(\R) \otimes \mathcal B(\Theta)/\mathcal B(\mathcal C([0,T], \R))$-measurable function
\begin{align*}
\mathcal X: \R \times \Theta &\to \mathcal C([0,T], \R) \\
(u,(\mathbf {w}^k)_{k\in \Z}, \mathbf b) & \mapsto \mathcal X(u,(\mathbf {w}^k)_{k\in \Z}, \mathbf b)
\end{align*}
which is, for every fixed $t\in [0,T]$, $\mathcal B(\R) \otimes \widehat{\mathcal B}_t(\Theta) /\mathcal B_t(\mathcal C([0,T], \R))$-measurable,
such that $\mathcal X$ is continuous in $u$ for $\mathbf P$-almost every fixed $((\mathbf {w}^k)_{k\in \Z}, \mathbf b) \in \Theta$ and  such that $\mathbb P^W \otimes \mathbb P^\beta$-almost surely, for every $u\in \R$ and for every $t\in [0,T]$, 
\begin{align}
\label{egalite_(a)}
x_t^g(u)= \mathcal X_t (u,(W^k)_{k \in \Z}, \beta);
\end{align}
\item[$(b)$] there is a $\mathcal B(\mathcal C([0,T], \C)^\Z) /\mathcal B(\mathcal C([0,T], \mathcal P_2(\R)))$-measurable function
\begin{align*}
\mathcal P: \mathcal C([0,T], \C)^\Z &\to \mathcal C([0,T], \mathcal P_2(\R)) \\
(\mathbf {w}^k)_{k\in \Z} & \mapsto \mathcal P((\mathbf {w}^k)_{k\in \Z})
\end{align*}
which is, for every fixed $t\in [0,T]$, $\mathcal B_t(\mathcal C([0,T], \C)^\Z) /\mathcal B_t(\mathcal C([0,T], \mathcal P_2(\R)))$-measurable,
such that $\mathbb P^W$-almost surely,  for every $t\in [0,T]$, 
\begin{align}
\label{egalite_(b)}
\mu_t^g= \mathcal P_t ((W^k)_{k \in \Z});
\end{align}
\item[$(c)$] there is a progressively-measurable function  $\mathcal H:[0,T] \times \R \times \Theta \to \R$, \textit{i.e.\!} for every $t\in [0,T]$, 
\begin{align*}
[0,t] \times \R \times \Theta &\to \R \\
(s,u,(\mathbf {w}^k)_{k\in \Z}, \mathbf b) & \mapsto \mathcal H_s(u,(\mathbf {w}^k)_{k\in \Z}, \mathbf b)
\end{align*}
is $\mathcal B[0,t] \otimes \mathcal B(\R) \otimes \widehat{\mathcal B}_t(\Theta) /\mathcal B(\R)$-measurable, 
such that $\mathbb P^W \otimes \mathbb P^\beta$-almost surely, for every $u\in \R$ and for every $t\in [0,T]$, 
\begin{align}
\label{egalite_(c)}
H_t^{g,\eps}(u)= \mathcal H_t (u,(W^k)_{k \in \Z}, \beta).
\end{align}
\end{itemize}
\end{lemme}

\begin{proof}
Consider the canonical space $(\Theta, \mathcal B(\Theta), (\widehat{\mathcal B}_t(\Theta))_{t\in [0,T]}, \mathbf P)$. By Proposition~\ref{prop:exist, uniq, cont, growth},  there is a strong and pathwise unique solution to~\eqref{eq_SDE_diff_torus} with initial condition $x_0^g=g$. Therefore, for every fixed $u \in \R$, there is a unique solution $(\mathbf{x}^g_t(u))_{t\in [0,T]}$ to
\begin{align*}
\mathbf{x}^g_t(u)=g(u)+\sum_{k \in \Z} \int_0^t f_k \Re \left(e^{-ik \mathbf{x}^g_s(u)} \mathrm d\mathbf {w}^k_s \right)+ \mathbf{b}_t.
\end{align*}

\textbf{Proof of $(a)$.}
By Yamada-Watanabe Theorem,  the law of $(x^g, (W^k)_{k \in \Z}, \beta)$ under $\mathbb P^W \otimes \mathbb P^\beta$ is equal to the law of $(\mathbf x^g, (\mathbf w^k)_{k \in \Z}, \mathbf b)$ under $\mathbf P$. This result is proved in~\cite[Prop.\! 5.3.20]{karatzasshreve91} for a finite-dimensional noise, but the proof is the same for the infinite-dimensional noise $((\mathbf {w}^k)_{k\in \Z}, \mathbf b) \in \Theta$. 
By a corollary to this theorem (see~\cite[Coro.\! 5.3.23]{karatzasshreve91}), it follows that for every $u\in \Q$, there is a $\mathcal B(\Theta)/\mathcal B(\mathcal C([0,T], \R))$-measurable function
\begin{align*}
\mathcal X^u:  \Theta &\to \mathcal C([0,T], \R) \\
((\mathbf {w}^k)_{k\in \Z}, \mathbf b) & \mapsto \mathcal X^u((\mathbf {w}^k)_{k\in \Z}, \mathbf b)
\end{align*}
which is, for every fixed $t\in [0,T]$, $\widehat{\mathcal B}_t(\Theta) /\mathcal B_t(\mathcal C([0,T], \R))$-measurable,
 such that $\mathbf P$-almost surely, for every $t\in [0,T]$, 
\begin{align}
\label{eq:karatzas_shreve}
\mathbf{x}^g_t(u)= \mathcal X^u_t ((\mathbf {w}^k)_{k\in \Z}, \mathbf b).
\end{align}

Moreover, again by Proposition~\ref{prop:exist, uniq, cont, growth}, there is a $\mathbf P$-almost sure event $A \in \mathcal B(\Theta)$  such that for every $((\mathbf {w}^k)_{k\in \Z}, \mathbf b) \in A$, the function $(t,u) \mapsto \mathbf{x}^g_t(u)$ is continuous on $[0,T] \times \R$. Up to modifying the almost-sure event $A$, we may assume that for every $((\mathbf {w}^k)_{k\in \Z}, \mathbf b) \in A$ and for every $u \in \Q$, equality~\eqref{eq:karatzas_shreve} holds. Therefore, we can define a continuous function in the variable $u \in \R$ by extending $u \in \Q \mapsto \mathcal X^u$. More precisely, define for every $u \in \R$, $((\mathbf {w}^k)_{k\in \Z}, \mathbf b) \in \Theta$, 
\begin{align*}
\mathcal X (u, (\mathbf {w}^k)_{k\in \Z}, \mathbf b)=
\left\{
\begin{aligned}
&\lim_{\substack{ u_n \to u \\ (u_n)_n \in \Q^\N}}  \mathcal X^{u_n} ( (\mathbf {w}^k)_{k\in \Z}, \mathbf b) \quad  & & \text{ if } ((\mathbf {w}^k)_{k\in \Z}, \mathbf b) \in A, \\
&0  \quad & &\text{ otherwise.}
\end{aligned}
\right.
\end{align*}
In the latter definition, the limit exists and for every $((\mathbf {w}^k)_{k\in \Z}, \mathbf b) \in A$,   $\mathcal X_t (u, (\mathbf {w}^k)_{k\in \Z}, \mathbf b)= \mathbf x_t^g(u)$ holds for any $t\in [0,T]$ and  $u \in \R$. 
By construction, for every $((\mathbf {w}^k)_{k\in \Z}, \mathbf b) \in \Theta$, $u \in \R \mapsto \mathcal X (u, (\mathbf {w}^k)_{k\in \Z}, \mathbf b) \in \mathcal C([0,T], \R)$ is continuous. 
It remains to show that $\mathcal X$ is progressively-measurable. 
Fix $t\in [0,T]$. 
By construction of $\mathcal X^u$, we know that for every $u\in \Q$, 
\begin{align*}
[0,t] \times \Theta &\to \R \\
(s,(\mathbf {w}^k)_{k\in \Z}, \mathbf b ) & \mapsto  \mathcal X_s^u((\mathbf {w}^k)_{k\in \Z}, \mathbf b)
\end{align*}
is $\mathcal B[0,t] \otimes \widehat{\mathcal B}_t(\Theta)/ \mathcal B(\R)$-measurable. 
Since $\mathcal X$ is the limit of  $\mathcal X_n:= \sum_{k \in \Z} \mathcal X^{k/n} \mathds 1_{\{ u \in [\frac{k}{n}, \frac{k+1}{n})\}}$, we deduce that for every $t\in [0,T]$, 
\begin{align*}
[0,t] \times \R \times \Theta &\to \R \\
(s,u,(\mathbf {w}^k)_{k\in \Z}, \mathbf b ) & \mapsto  \mathcal X_s(u,(\mathbf {w}^k)_{k\in \Z}, \mathbf b)
\end{align*}
is $\mathcal B[0,t] \otimes \mathcal B(\R) \otimes \widehat{\mathcal B}_t(\Theta)/ \mathcal B(\R)$-measurable. 

Recall that $\mathcal L^{\mathbb P^W \otimes \mathbb P^\beta} (x^g, (W^k)_{k \in \Z}, \beta)= \mathcal L^{\mathbf P} (\mathbf x^g, (\mathbf w^k)_{k \in \Z}, \mathbf b)$. 
Since $\mathbb P$-almost surely, for every $u \in \R$ and for every $t\in [0,T]$, 
$\mathbf x_t^g(u)= \mathcal X_t (u, (\mathbf {w}^k)_{k\in \Z}, \mathbf b)$, we deduce that $\mathbb P^W \otimes \mathbb P^\beta$-almost surely,   for every $u \in \R$ and for every $t\in [0,T]$, equality~\eqref{egalite_(a)} holds. It completes the proof of~$(a)$.

\textbf{Proof of $(b)$.}
This step is equivalent to find $\mathcal P:\mathcal C([0,T], \C)^\Z \to \mathcal C([0,T], \mathcal P_2(\R))$ such that for every bounded measurable function $\Upsilon:\R \to \R$, the function 
\begin{align*}
\langle \Upsilon,\mathcal P \rangle:  \mathcal C([0,T], \C)^\Z &\to \mathcal C([0,T],\R) \\
(\mathbf {w}^k)_{k\in \Z} & \mapsto \langle \Upsilon, \mathcal P((\mathbf {w}^k)_{k\in \Z}) \rangle = \int_\R \Upsilon (x) \mathrm d \mathcal P((\mathbf {w}^k)_{k\in \Z}) (x)
\end{align*}
is $\mathcal B_t(\mathcal C([0,T], \C)^\Z) /\mathcal B_t(\mathcal C([0,T], \R))$-measurable for every fixed $t\in [0,T]$.
Define $\mathcal P$ by duality: for every $\Upsilon:\R \to \R$ bounded and measurable, 
\begin{align*}
 \langle \Upsilon,  \mathcal P((\mathbf {w}^k)_{k\in \Z}) \rangle := \int_{\mathcal C([0,T],\R)} \int_0^1 \Upsilon( \mathcal X(v,(\mathbf {w}^k)_{k\in \Z}, \mathbf b ))  \; \mathrm dv \;  \mathrm d \mu_{\mathrm{Wiener}} (\mathbf b),
\end{align*}
where $\mu_{\mathrm{Wiener}}$ denotes the Wiener measure on $\mathcal C([0,T],\R)$. 
Thus  $\mathbb P^W$-almost surely, for every $t \in [0,T]$, for every $\Upsilon:\R \to \R$ bounded and measurable, 
\begin{align*}
\langle \Upsilon,  \mathcal P_t((W^k)_{k\in \Z}) \rangle
&=\mathbb E^\beta \left[ \int_0^1 \Upsilon( \mathcal X_t(v,(W^k)_{k\in \Z}, \beta ))  \mathrm dv  \right] = \langle \Upsilon, \mu_t^g \rangle,
\end{align*}
where the last equality follows from Definition~\ref{def:mesure_mu}.
Thus we proved equality~\eqref{egalite_(b)}. 

Moreover, for every $t\in [0,T]$,  by composition of two measurable functions, 
\begin{align*}
[0,t] \times \R \times \Theta &\to \R  \\
(s,u,(\mathbf {w}^k)_{k\in \Z}, \mathbf b ) & \mapsto \Upsilon( \mathcal X_s(u,(\mathbf {w}^k)_{k\in \Z}, \mathbf b ))
\end{align*}
is $\mathcal B[0,t] \otimes \mathcal B(\R) \otimes \widehat{\mathcal B}_t(\Theta)/ \mathcal B(\R)$-measurable. 
By Fubini's Theorem, it follows that for every $t\in [0,T]$, 
\begin{align*}
[0,t] \times \mathcal C([0,T], \C)^\Z  &\to \R \\
(s,(\mathbf {w}^k)_{k\in \Z} ) & \mapsto  \int_{\mathcal C([0,T],\R)}  \int_0^1 \Upsilon( \mathcal X_s(v,(\mathbf {w}^k)_{k\in \Z}, \mathbf b ))  \;\mathrm dv \; \mathrm d \mu_{\mathrm{Wiener}} (\mathbf b),
\end{align*}
is $\mathcal B[0,t] \otimes \mathcal B_t( \mathcal C([0,T], \C)^\Z)/\mathcal B(\R)$-measurable. 
This completes the proof of $(b)$.

\textbf{Proof of $(c)$.}
Define, on the canonical space $(\Theta, \mathcal B(\Theta))$, $\mathbf F_t^g= (\mathbf x_t^g)^{-1}$ and
\begin{align*}
\mathbf A_t^g&:= \partial_u \mathbf x_t^g(\mathbf F_t^g(\cdot)) h(\mathbf F_t^g(\cdot)) ; \\
\mathbf A_t^{g,\eps} &:= \int_\R \mathbf A_t^g(\cdot-y)   \varphi_\eps (y) \mathrm dy ; \\
\mathbf H_t^{g,\eps}(u)&:= \frac{1}{\partial_u \mathbf x_t^g(u)} (\mathbf A_t^g-\mathbf A_t^{g,\eps})(\mathbf x_t^g(u)).
\end{align*}
In order to prove that $\mathbf H^{g, \eps}$ can be written as a progressively measurable function of $u$ and $((\mathbf w^k)_{k \in \Z}, \mathbf b)$, we will prove successively that this property  holds for $\partial_u \mathbf x^g$, $\mathbf F^g$, $\mathbf A^g$ and $\mathbf A^{g,\eps}$ and we will deduce the result for $\mathbf H^{g, \eps}$ by composition of progressively measurable functions. 

Let us start with $\partial_u \mathbf x^g$. 
By Proposition~\ref{prop:differentiability}, since $g\in \mathbf G^{1+\theta}$ and $\alpha > \frac{3}{2}+\theta$, $\mathbf P$-almost surely, for every $t\in [0,T]$, the map $u\mapsto \mathbf x_t^g(u)$ is of class $\mathcal C^1$. 
Thus there exists  a $\mathbf P$-almost-sure event $A \in \mathcal B(\theta)$ such that for every $((\mathbf w^k)_{k \in \Z}, \mathbf b) \in A$, $ \mathbf x_t^g(u)=\mathcal X_t (u, (\mathbf {w}^k)_{k\in \Z}, \mathbf b)$ holds for every $(t,u) \in [0,T] \times \R$ and such that $u \mapsto \mathbf x_t^g(u)$ belongs to $\mathcal C^1$. 
Define for every $((\mathbf w^k)_{k \in \Z}, \mathbf b) \in A$, for every $(t,u) \in [0,T] \times \R$, 
\begin{align*}
\partial_u\mathcal X_t(u,(\mathbf w^k)_{k \in \Z}, \mathbf b):= \limsup_{\eta \searrow 0} \frac{\mathcal X_t(u+\eta,(\mathbf w^k)_{k \in \Z}, \mathbf b)-\mathcal X_t(u,(\mathbf w^k)_{k \in \Z}, \mathbf b)}{\eta}.
\end{align*}
Thus for every $((\mathbf w^k)_{k \in \Z}, \mathbf b) \in A$ and for every $(t,u) \in [0,T] \times \R$, $\partial_u \mathbf x_t^g(u)=\partial_u\mathcal X_t(u,(\mathbf w^k)_{k \in \Z}, \mathbf b)$.
Moreover, by progressive-measurability of $\mathcal X$, it follows from the definition of $\partial_u \mathcal X$ is also progressively measurable; more precisely, for every $t\in [0,T]$, 
\begin{align*}
[0,t] \times \R \times \Theta &\to \R \\
(s,u,(\mathbf {w}^k)_{k\in \Z}, \mathbf b ) & \mapsto \partial_u \mathcal X_s(u,(\mathbf {w}^k)_{k\in \Z}, \mathbf b)
\end{align*}
is $\mathcal B[0,t] \otimes \mathcal B(\R) \otimes \widehat{\mathcal B}_t(\Theta)/ \mathcal B(\R)$-measurable. 

Now, consider $\mathbf F^g$. 
Define for every $x\in [0,2\pi]$
\begin{align}
\label{defin:widetilde Ftg}
\widetilde{\mathbf F}_t^g(x):=\int_0^1 \mathds 1_{\{\mathbf x_t^g(v)-\mathbf x_t^g(0) \leq x\}} \mathrm dv.
\end{align}
Thus we have for every $x\in [\mathbf x_t(0),\mathbf x_t(0)+2\pi]$
\begin{align*}
\widetilde{\mathbf F}_t^g(x-\mathbf x_t^g(0))=\int_0^1 \mathds 1_{\{\mathbf x_t^g(v) \leq x\}} \mathrm dv
=\int_0^1 \mathds 1_{\{v\leq \mathbf F_t^g(x)\}} \mathrm dv=\mathbf F_t^g(x). 
\end{align*}
Therefore, since for every $x\in \R$, $\mathbf F_t^g(x+2\pi)=\mathbf F_t^g(x)+1$, we have
\begin{align*}
\mathbf F_t^g(x)=\sum_{k \in \Z} \mathds 1_{\{ x-2\pi k \in [\mathbf x_t(0),\mathbf x_t(0)+2\pi) \}} \left( \widetilde{\mathbf F}_t^g(x-2\pi k -\mathbf x_t^g(0))+k \right).
\end{align*}
Hence it is sufficient to prove that we can write $\widetilde{\mathbf F}_t^g$ as a progressively measurable function of $x$ and $((\mathbf w^k)_{k \in \Z}, \mathbf b)$. 
Recall that $\mathbf P$-almost surely, $u \mapsto \mathbf x^g (u) = \mathcal X(u,(\mathbf w^k)_{k \in \Z}, \mathbf b)$ is continuous. Thus there is $\mathcal I$ such that $\mathbf P$-almost surely, for every $v \in [0,1]$, for every $x \in [0,2\pi]$, 
$\mathds 1_{\{\mathbf x_\cdot^g(v)-\mathbf x_\cdot^g(0) \leq x\}} =\mathcal I_\cdot (v,x,(\mathbf w^k)_{k \in \Z}, \mathbf b)$
and such that for every $t\in [0,T]$, 
\begin{align*}
[0,t] \times [0,1] \times [0,2\pi] \times \Theta &\to \R \\
(s,v,x,(\mathbf {w}^k)_{k\in \Z}, \mathbf b ) & \mapsto  \mathcal I_s(v,x,(\mathbf {w}^k)_{k\in \Z}, \mathbf b)
\end{align*}
is $\mathcal B[0,t] \otimes \mathcal B([0,1] \times [0,2\pi]) \otimes \widehat{\mathcal B}_t(\Theta)/ \mathcal B(\R)$-measurable. 
It follows from Fubini's Theorem and from~\eqref{defin:widetilde Ftg} that 
for every $t\in [0,T]$, 
\begin{align*}
[0,t]  \times [0,2\pi] \times \Theta &\to \R \\
(s,x,(\mathbf {w}^k)_{k\in \Z}, \mathbf b ) & \mapsto  \int_0^1 \mathcal I_s(v,x,(\mathbf {w}^k)_{k\in \Z}, \mathbf b) \mathrm dv = \widetilde{\mathbf F}_s^g(x)
\end{align*}
is $\mathcal B[0,t] \otimes \mathcal B( [0,2\pi]) \otimes \widehat{\mathcal B}_t(\Theta)/ \mathcal B(\R)$-measurable.

Let us conclude with  $\mathbf A^g$, $\mathbf A^{g,\eps}$ and  $\mathbf H^{g, \eps}$. 
First, remark that $\mathbf A^g$ is obtained by products and compositions of $\partial_u \mathbf x^g$, $\mathbf F^g$ and $h$, where $h$ is a $\mathcal C^1$-function. Thus $x \mapsto \mathbf A^g(x)$ is a progressively measurable function of $x$ and $((\mathbf {w}^k)_{k\in \Z}, \mathbf b)$. 
It follows also that $(x,y) \mapsto \mathbf A^g (x-y) \varphi_\eps (y)$ is a progressively measurable function of $x$, $y$ and $(\mathbf {w}^k)_{k\in \Z}, \mathbf b$. By Fubini's Theorem, we deduce that $x \mapsto \mathbf A^{g, \eps}(x)$ is a progressively measurable function of $x$ and $((\mathbf {w}^k)_{k\in \Z}, \mathbf b)$. 
Again by products and compositions, it follows that there is a progressively measurable function $\mathcal H$  such that $\mathbf P$-almost surely, for every $u\in \R$ and for every $t\in [0,T]$, 
\begin{align*}
\mathbf H_t^{g,\eps}(u) = \mathcal H_t (u, (\mathbf {w}^k)_{k\in \Z}, \mathbf b). 
\end{align*}
It follows that $\mathbf P^W \otimes \mathbf P^\beta$-almost surely, equality~\eqref{egalite_(c)} holds. It completes the proof of $(c)$.
\end{proof}

\subsection{Idiosyncratic noise}
\label{parag:idiosyncratic noise}

Coming back to equality~\eqref{I_2_avec_H}, and applying the  relation $\partial_\mu \phi(\mu_t^g) = \partial_v \left\{\frac{\delta \phi}{\delta m} (\mu_t^g)\right\}$ (see Proposition~\ref{prop:link_derivatives} in appendix), we have
\begin{align*}
I_2&= \frac{1}{t} \Ewb{ \int_0^1 \!\!\int_0^t \partial_v \left\{\frac{\delta \phi}{\delta m} (\mu_t^g)\right\} (x_t^g(u)) \; \partial_u x_t^g(u)  \;  H^{g,\eps}_s(u) \;\mathrm ds \mathrm du } \\
&=\frac{1}{t}\int_0^1 \!\!\int_0^t \Ewb{  \partial_u \left\{\frac{\delta \phi}{\delta m} (\mu_t^g) (x_t^g(\cdot)) \right\} (u)\; H^{g,\eps}_s(u)  }\mathrm ds \mathrm du.
\end{align*}
By definition~\eqref{notation moyenne}, $\left[\frac{\delta \phi}{\delta m}\right]$ is equal to $\frac{\delta \phi}{\delta m}$ up to a constant, so their derivatives are equal and
\begin{align}
\label{reprendre_egalite}
I_2=\frac{1}{t}\int_0^1 \!\!\int_0^t \Ewb{  \partial_u \left\{\left[\frac{\delta \phi}{\delta m}\right] (\mu_t^g) (x_t^g(\cdot)) \right\} (u) \; H^{g,\eps}_s(u)  }\mathrm ds \mathrm du.
\end{align}

In the following lemma, we prove that $I_2$ can be expressed in terms of $\frac{\delta \phi}{\delta m}$ instead of its derivative. This key step is, as shown below, a consequence of Girsanov's Theorem applied with respect to the idiosyncratic noise $\beta$.

\begin{lemme}
\label{lemme:traitement de I2}
Let $\theta \in (0,1)$, $g \in \mathbf G^{1+\theta}$ and $f$ be of order $\alpha>\frac{3}{2}+\theta$.
Let $h \in \Delta^1$ and $\eps >0$. 
Fix $u\in [0,1]$ and $s<t \in [0,T]$.  
Thus the  following equality holds true
\begin{multline}
\label{34}
\Ewb{\partial_u \left\{\left[ \frac{\delta \phi}{\delta m} \right] (\mu_t^g) (x_t^g(\cdot)) \right\}(u) \; H^{g,\eps}_s(u)} \\
=\Ewb{\left[ \frac{\delta \phi}{\delta m} \right](\mu_t^g)(x_t^g(u)) \;H^{g,\eps}_s(u) \; \frac{1}{t-s}\int_s^t \partial_u x_r^g(u) \mathrm d\beta_r}.
\end{multline}
\end{lemme}

\begin{proof}
Fix $u\in [0,1]$ and $s<t \in [0,T]$.
Define, for every $r\in [0,T]$, $\xi_r:=
\frac{1}{t-s}\int_0^r \mathds 1_{\{z \in [s,t] \}} \mathrm dz$.

For every $\nu \in [-1,1]$, denote by $(x_r^\nu)_{r\in [0,T]}$ the process $(x_r^{g+\nu \xi_r})_{r\in [0,T]}$. 
By Proposition~\ref{prop:kunita_identities}, $\mathbb P^W \otimes \mathbb P^\beta$-almost surely, $x_r^\nu(u)=x_r^{g+\nu \xi_r}(u)=Z_r^{g(u)+\nu \xi_r}$.
Apply Kunita's expansion (Lemma~\ref{lem_Kunita}) to $x=g(u)$ and $\zeta_t=\xi_t$. We obtain for every $u\in \R$, $\mathbb P^W \otimes \mathbb P^\beta$-almost surely for every $r\in [0,T]$ and every $\nu \in [-1,1]$:
\begin{align*}
x_r^\nu (u) 
=  g(u)+\sum_{k\in \Z} f_k \int_0^r \Re \left( e^{-ik x_z^\nu(u)} \mathrm dW_z^k \right) + \beta_r + \nu \int_0^r \partial_x Z_z^{g(u)+\nu \xi_z} \dot{\xi}_z \mathrm dz.
\end{align*}
Since both terms of the last equality are almost surely continuous with respect to $u\in \R$, that equality holds almost surely for every $u \in \R$. 

For every $\nu \in [-1,1]$,  define the following stopping time
\begin{align*}
\sigma^\nu:=\inf \left\{ r\geq 0: \left|\nu \int_0^r \partial_x Z_z^{g(u)+\nu \xi_z} \dot{\xi}_z \mathrm d\beta_z  \right| \geq 1 \right\} \wedge T. 
\end{align*}
Define the process $(y_r^\nu)_{r\in [0,T]}$ as the solution to 
\begin{align*}
\mathrm dy_r^\nu (u) 
&= \sum_{k\in \Z} f_k \Re \left( e^{-ik y_r^\nu(u)} \mathrm dW_r^k \right) + \mathrm d\beta_r + \nu \mathds 1_{\{ r \leq \sigma^\nu \}} \partial_x Z_r^{g(u)+\nu \xi_r} \dot{\xi}_r \mathrm dr,
\end{align*}
and $\beta^\nu_r:= \beta_r + \nu \int_0^r \mathds 1_{\{ z \leq \sigma^\nu \}} \partial_x Z_z^{g(u)+\nu \xi_z} \dot{\xi}_z \mathrm dz$. 
Let us define for every $r\in [0,T]$
\begin{align*}
\mathcal E^\nu_r=\exp \left(-\nu \int_0^{r\wedge \sigma^\nu} \partial_x Z_z^{g(u)+\nu \xi_z} \dot{\xi}_z \mathrm d\beta_z -\frac{\nu^2}{2} \int_0^{r\wedge \sigma^\nu} \left| \partial_x Z_z^{g(u)+\nu \xi_z} \dot{\xi}_z\right|^2 \mathrm d z\right).
\end{align*}
By definition of $\sigma^\nu$, we have $\mathcal E^\nu_r \leq \exp \left(-\nu \int_0^{r\wedge \sigma^\nu} \partial_x Z_z^{g(u)+\nu \xi_z} \dot{\xi}_z \mathrm d\beta_z \right)\leq \exp \left( 1 \right)$. In particular, 
$(\mathcal E^\nu_r)_{r\in [0,T]}$ is a $\mathbb P^W \otimes \mathbb P^\beta$-martingale. Define $\mathbb P^\nu$ as the absolutely continuous probability measure with respect to $\mathbb P^W \otimes \mathbb P^\beta$ with density $\frac{\mathrm d \mathbb P^\nu}{\mathrm d(\mathbb P^W \otimes \mathbb P^\beta)}= \mathcal E^\nu_T$.
Thus by Girsanov's Theorem, the law under $\mathbb P^\nu$ of $((W^k)_{k\in \Z},\beta^\nu)$ is equal to the law under $\mathbb P^W \otimes \mathbb P^\beta$ of $((W^k)_{k\in \Z},\beta)$. 
It follows that $(\Omega, \mathcal G, (\mathcal G_t)_{t\in [0,T]}, \mathbb P^\nu, y^\nu, (W^k)_{k \in \Z}, \beta^\nu)$ is a weak solution to equation~\eqref{eq_SDE_diff_torus}.

By Lemma~\ref{lemma:prog_measurability} and  Yamada-Watanabe Theorem, 
\begin{align*}
\mathcal L^{\mathbb P^\nu} (y^\nu, (W^k)_{k \in \Z}, \beta^\nu)=
\mathcal L^{\mathbb P^W \otimes \mathbb P^\beta} (x^g, (W^k)_{k \in \Z}, \beta)= \mathcal L^{\mathbf P} (\mathbf x^g, (\mathbf w^k)_{k \in \Z}, \mathbf b),
\end{align*}
and $\mathbb P^\nu$-almost surely, for every $u\in \R$ and  $t\in [0,T]$, 
\begin{align}
\label{condition_d}
y^\nu_t(u)= \mathcal X_t (u,(W^k)_{k \in \Z}, \beta^\nu).
\end{align}

Moreover, we claim that $\mathbb E^\nu \left[\left[ \frac{\delta \phi}{\delta m} \right](\mu_t^g) (y_t^\nu(u)) H^{g,\eps}_s(u) \right]$ does not depend on $\nu$. Indeed, by~\eqref{condition_d}
\begin{multline*}
F(\nu):=\mathbb E^\nu \left[\left[ \frac{\delta \phi}{\delta m} \right](\mu_t^g) (y_t^\nu(u)) H^{g,\eps}_s(u) \right]
=\mathbb E^\nu \left[\left[ \frac{\delta \phi}{\delta m} \right](\mu_t^g) \big(\mathcal X_t (u,(W^k)_{k \in \Z}, \beta^\nu)\big) H^{g,\eps}_s(u) \right] \\
=\Ewb{\left[ \frac{\delta \phi}{\delta m} \right](\mu_t^g) \big(\mathcal X_t (u,(W^k)_{k \in \Z}, \beta^\nu)\big) H^{g,\eps}_s(u) \frac{\mathrm d \mathbb P^\nu}{\mathrm d(\mathbb P^W \otimes \mathbb P^\beta)}}.
\end{multline*}
Furthermore, by Lemma~\ref{lemma:prog_measurability}
\begin{align*}
F(\nu)&=\Ewb{\left[ \frac{\delta \phi}{\delta m} \right]\big(\mathcal P_t ((W^k)_{k \in \Z})\big ) \; \big(\mathcal X_t (u,(W^k)_{k \in \Z}, \beta^\nu)\big) \; \mathcal H_s (u,(W^k)_{k \in \Z}, \beta) \; \frac{\mathrm d \mathbb P^\nu}{\mathrm d(\mathbb P^W \otimes \mathbb P^\beta)}}\\
&=\mathbb E^\nu \left[\left[ \frac{\delta \phi}{\delta m} \right]\big(\mathcal P_t ((W^k)_{k \in \Z})\big ) \; \big(\mathcal X_t (u,(W^k)_{k \in \Z}, \beta^\nu)\big) \; \mathcal H_s (u,(W^k)_{k \in \Z}, \beta) \right].
\end{align*}
Moreover,  the processes $(\beta_r)$ and $(\beta^\nu_r)$ are equal on the interval $[0,s]$, because $\xi_r\equiv 0$ on $[0,s]$. 
Since $(\mathcal H_r)_{r\in [0,T]}$ is progressively measurable, then $\mathbb P^\nu$-almost surely, $\mathcal H_s (u,(W^k)_{k \in \Z}, \beta)= \mathcal H_s (u,(W^k)_{k \in \Z}, \beta^\nu)$. 
Therefore, 
\begin{align*}
F(\nu)&=\mathbb E^\nu \left[\left[ \frac{\delta \phi}{\delta m} \right]\big(\mathcal P_t ((W^k)_{k \in \Z})\big ) \; \big(\mathcal X_t (u,(W^k)_{k \in \Z}, \beta^\nu)\big) \; \mathcal H_s (u,(W^k)_{k \in \Z}, \beta^\nu) \right] \\
&=  \Ewb{\left[ \frac{\delta \phi}{\delta m} \right]\big(\mathcal P_t ((W^k)_{k \in \Z})\big ) \; \big(\mathcal X_t (u,(W^k)_{k \in \Z}, \beta)\big) \; \mathcal H_s (u,(W^k)_{k \in \Z}, \beta)},
\end{align*}
since  the law of $((W^k)_{k\in \Z},\beta^\nu)$ under $\mathbb P^\nu$ is equal to the law of $((W^k)_{k\in \Z},\beta)$ under $\mathbb P^W \otimes \mathbb P^\beta$.
The last term of that equality does not depend on $\nu$ anymore, so  we get finally
\begin{align}
\label{derivee_nulle_ter}
\frac{\mathrm d}{\mathrm d\nu}_{\vert \nu=0} F(\nu) =
\frac{\mathrm d}{\mathrm d\nu}_{\vert \nu=0} \mathbb E^\nu \left[\left[ \frac{\delta \phi}{\delta m} \right](\mu_t^g) (y_t^\nu(u)) H^{g,\eps}_s(u) \right]=0.
\end{align}
Furthermore, 
\begin{align}
\label{décomposition avec reste}
\mathbb E^\nu \Big[\Big[ \frac{\delta \phi}{\delta m} \Big](\mu_t^g) (y_t^\nu(u)) H^{g,\eps}_s(u) \Big]
&=\Ewb{\Big[ \frac{\delta \phi}{\delta m} \Big](\mu_t^g) (y_t^\nu(u)) H^{g,\eps}_s(u) \mathcal E^\nu_t } \notag\\
&=\Ewb{\Big[ \frac{\delta \phi}{\delta m} \Big](\mu_t^g) (x_t^\nu(u)) H^{g,\eps}_s(u)\mathcal E^\nu_t } 
+R(\nu), 
\end{align}
where $R(\nu)= \Ewb{ \mathds 1_{\{ \sigma^\nu <T  \}} \Big( \Big[ \frac{\delta \phi}{\delta m} \Big](\mu_t^g) (y_t^\nu(u))  - \Big[ \frac{\delta \phi}{\delta m} \Big](\mu_t^g) (x_t^\nu(u))   \Big) H^{g,\eps}_s(u)\mathcal E^\nu_t } $; we used here the fact that   $\mathds 1_{\{ \sigma^\nu =T \}} (x_t^\nu(u)-y_t^\nu(u))=0$. 
Let us show that $R(\nu)= \mathcal O(|\nu|^2)$.
By Hölder's inequality and by the fact that $\mathcal E_t^\nu \leq \exp(1)=e$, we have
\begin{align}
\label{gestion du petit o de nu}
|R(\nu)|
&\leq 2 e
   (\mathbb P^W \otimes \mathbb P^\beta) \Big[\sigma^\nu <T\Big]^{1/4} \Ewb{H^{g,\eps}_s(u)^2}^{1/2} 
  \Ewb{ \sup_{v\in \R} \Big|\Big[ \frac{\delta \phi}{\delta m} \Big](\mu_t^g) (v)\Big|^4}^{\frac{1}{4}}.
\end{align}
We control the different terms appearing on the r.h.s.\! of~\eqref{gestion du petit o de nu}.
By Markov's inequality, by Burkholder-Davis-Gundy inequality  and by  inequality~\eqref{borne sur pZtx}, for every $\nu \in [-1,1]$
\begin{align*}
\mathbb P^W \otimes \mathbb P^\beta \Big[ \sigma^\nu < T \Big]
&\leq \Ewb{\sup_{r \leq T}  \Big|\nu \int_0^r \partial_x Z_z^{g(u)+\nu \xi_z} \frac{\mathds 1_{\{z \in [s,t] \}}}{t-s} \mathrm d\beta_z  \Big|^8}\\
&\leq C |\nu |^8 \; \Ewb{  \Big|\int_s^t \Big|\partial_x Z_r^{g(u)+\nu \xi_r}  \Big|^2  \frac{1}{(t-s)^2}\mathrm dr  \Big|^4} \\
&\leq \frac{C}{(t-s)^5} |\nu |^8 \; \Ewb{ \int_s^t \Big|\partial_x Z_r^{g(u)+\nu \xi_r}  \Big|^8  \mathrm dr  }\leq C |\nu |^8,
\end{align*}
where $C$ is a constant depending on $s$ and $t$ changing from line to line.

Let us show that $\Ewb{H^{g,\eps}_s(u)^2}<+\infty$. Recall that $H_s^{g,\eps}(u):= \frac{(A_s^g-A_s^{g,\eps})(x_s^g(u))}{\partial_u x_s^g(u)}$. Thus 
\begin{align*}
\Ewb{H^{g,\eps}_s(u)^2}^{1/2}
&\leq \Ewb{\left\| \frac{1}{\partial_u x_s^g} \right\|^4_{L_\infty}}^{1/4}
\Ewb{\|A_s^g - A_s^{g,\eps}\|^4_{L_\infty}}^{1/4}.
\end{align*}
By~\eqref{Lp norm 8}, $\Ewb{\left\| \frac{1}{\partial_u x_s^g} \right\|^4_{L_\infty}}$ is finite. 
Moreover, since $A_s^{g,\eps} = A_s^g \ast \varphi_\eps$, we have  
$\|A_s^g - A_s^{g,\eps}\|_{L_\infty} \leq  C \eps \|\partial_x A_s^g \|_{L_\infty}$, where 
$C= \int_\R |y|\varphi(y)\mathrm dy$. 
Using inequality~\eqref{norme 1 Asg} and an analogue to~\eqref{norme 1 Asg_bis} with exponent $4$ instead of $2$, we check that $\Ewb{\|\partial_x A_s^g \|_{L_\infty}}$ is finite. Thus there is $C$ such that $\Ewb{H^{g,\eps}_s(u)^2}\leq C$.

Then show that $\Ewb{ \sup_{v\in \R} \Big|\Big[ \frac{\delta \phi}{\delta m} \Big](\mu_t^g) (v)\Big|^4}$ is finite. 
By definition~\eqref{notation moyenne}, for every $v \in \R$, 
\begin{align}
\label{ineq 1}
\Big|\Big[ \frac{\delta \phi}{\delta m} \Big](\mu_t^g) (v)\Big|
\leq \mathbb E^\beta \left[ \int_0^1 \Big| \frac{\delta \phi}{\delta m}(\mu_t^g) (v) - \frac{\delta \phi}{\delta m} (\mu_t^g) (x_t^g(u))  \Big|  \mathrm du  \right].
\end{align}
By inequality~\eqref{Lipschitz derivee de Lions 2}, there is a $C>0$ such that $\mathbb P^W $-almost surely for every $x\in [0,2\pi]$, 
\begin{align}
\label{eq:35ab}
\left|  \partial_v \left\{ \frac{\delta \phi}{\delta m}(\mu_t^g)\right\}(x)\right| 
=\left|  \partial_\mu \phi (\mu_t^g)(x)\right| 
\leq C(1+2\pi) + C \mathbb E^\beta \left[\int_0^1 |x_t^g (u')| \mathrm du'\right] .
\end{align}
Thus there is $C>0$ such that $\mathbb P^W $-almost surely, for every $v, v' \in [0,2\pi]$, 
\begin{align*}
\left|   \frac{\delta \phi}{\delta m}(\mu_t^g)(v) -  \frac{\delta \phi}{\delta m}(\mu_t^g)(v')\right|
\leq  C   + C \mathbb E^\beta \left[\int_0^1 |x_t^g (u')| \mathrm du'\right].
\end{align*}
By Proposition~\ref{prop:integrale nulle}, $v\mapsto \frac{\delta \phi}{\delta m}(\mu_t^g)(v) $ is $2\pi$-periodic, thus the latter inequality holds for every $v, v' \in \R$. Combining that inequality with~\eqref{ineq 1}, we get for every $v \in \R$, 
\begin{align*}
\Big|\Big[ \frac{\delta \phi}{\delta m} \Big](\mu_t^g) (v)\Big|
\leq  C   + C \mathbb E^\beta \left[\int_0^1 |x_t^g (u')| \mathrm du'\right].
\end{align*}
This leads to
\begin{align}
\label{ineq 2}
\Ewb{ \sup_{v\in \R} \Big|\Big[ \frac{\delta \phi}{\delta m} \Big](\mu_t^g) (v)\Big|^4}
\leq C + C\Ewb{ \int_0^1 |x_t^g (u')| \mathrm du'  } ,
\end{align}
which is finite. Thus $\Ewb{ \sup_{v\in \R} \Big|\Big[ \frac{\delta \phi}{\delta m} \Big](\mu_t^g) (v)\Big|^4} \leq C$, whence we finally deduce, in view of inequality~\eqref{gestion du petit o de nu}, that $|R(\nu)| \leq C |\nu|^2$.

Thus $R(\nu)= \mathcal O(|\nu|^2)$. It follows from~\eqref{derivee_nulle_ter} and~\eqref{décomposition avec reste}  that
\begin{align*}
\frac{\mathrm d}{\mathrm d\nu}_{\vert \nu=0}\Ewb{\left[ \frac{\delta \phi}{\delta m} \right](\mu_t^g) (x_t^\nu(u)) \;H^{g,\eps}_s(u)\;\mathcal E^\nu_t }=0
\end{align*}
By~\eqref{ineq 2},  $\Big(\left[ \frac{\delta \phi}{\delta m} \right](\mu_t^g) (x_t^\nu(u))\Big)_{\nu\in [-1,1]}$ is uniformly integrable. 
Using  inequality~\eqref{eq:35ab}, we prove in the same way that $\Big( \partial_v \Big\{\left[ \frac{\delta \phi}{\delta m} \right](\mu_t^g) \Big\}(x_t^\nu(u))\Big)_{\nu\in [-1,1]}$ is also uniformly integrable. Recall that $x_t^\nu(u)=Z_t^{g(u)+\nu \xi_t}$ and that, by inequality~\eqref{borne sur pZtx}, $(\partial_x Z_t^{g(u)+\nu \xi_t})_{\nu\in [-1,1]}$ is uniformly integrable. Thus we get by differentiation:
\begin{align*}
0&= \frac{\mathrm d}{\mathrm d\nu}_{\vert \nu=0} \Ewb{\left[ \frac{\delta \phi}{\delta m} \right](\mu_t^g) (Z_t^{g(u)+\nu \xi_t}) \;  H^{g,\eps}_s(u) \; \mathcal E^\nu_t} \notag  \\
&= \Ewb{\partial_v \Big\{\left[ \frac{\delta \phi}{\delta m} \right](\mu_t^g) \Big\} (Z_t^{g(u)}) \; \partial_x Z_t^{g(u)} \; \xi_t \; H^{g,\eps}_s(u)}  \\
&\quad -  \Ewb{\left[ \frac{\delta \phi}{\delta m} \right](\mu_t^g) (Z_t^{g(u)}) \; H^{g,\eps}_s(u)\int_0^t \partial_x Z_r^{g(u)} \dot{\xi}_r \mathrm d\beta_r} \notag
\end{align*}

Using  $Z_t^{g(u)}=x_t^g(u)$ and $\partial_x Z_t^{g(u)} = \frac{\partial_u x_t^g(u)}{g'(u)}$ and recalling that $\xi_r:=\frac{1}{t-s}\int_0^r \mathds 1_{\{z \in [s,t] \}} \mathrm dz$, we have proved that
\begin{multline*}
\Ewb{\partial_v \Big\{\left[ \frac{\delta \phi}{\delta m} \right](\mu_t^g) \Big\} (x_t^g(u)) \frac{\partial_u x_t^g(u)}{g'(u)}   \; H^{g,\eps}_s(u)}
\\= \Ewb{\left[ \frac{\delta \phi}{\delta m} \right](\mu_t^g) (x_t^g(u)) H^{g,\eps}_s(u)\frac{1}{t-s}\int_s^t \frac{\partial_u x_r^g(u)}{g'(u)}   \mathrm d \beta_r}.
\end{multline*}
We multiply both sides by $g'(u)$ and we obtain  equality~\eqref{34}, since $\partial_u \left\{\left[ \frac{\delta \phi}{\delta m} \right] (\mu_t^g) (x_t^g(\cdot)) \right\}(u) =\partial_v \Big\{\left[ \frac{\delta \phi}{\delta m} \right](\mu_t^g) \Big\} (x_t^g(u)) \; \partial_u x_t^g(u)$.
\end{proof}

\subsection{Conclusion of the analysis}
\label{parag:conclusion_I2}

We conclude the proof of Proposition~\ref{prop:estim 2}.
\begin{proof}[Proof (Proposition~\ref{prop:estim 2})]
Putting together equalities~\eqref{reprendre_egalite} and~\eqref{34} and definition~\eqref{defin:Ktgeps} of $K_t^{g,\eps}$, 
\begin{align*}
I_2
= \frac{1}{t} \Ewb{\int_0^1 \left[\frac{\delta \phi}{\delta m}\right](\mu_t^g)(x_t^g(u)) K_t^{g,\eps}(u) \mathrm du}.
\end{align*}
By Cauchy-Schwarz inequality, 
\begin{align*}
|I_2| \leq \frac{1}{t}   \Ewb{\left\| K_t^{g,\eps} \right\|_{L_\infty}^2   }^{1/2} 
\Ewb{\left| \int_0^1 \left[ \frac{\delta \phi}{\delta m} \right] (\mu_t^g) (x_t^g(u)) \frac{K_t^{g,\eps} (u)}{\left\| K_t^{g,\eps} \right\|_{L_\infty}}\mathrm du\right|^2   }^{1/2}.
\end{align*}
It remains to estimate $ \Ewb{\left\| K_t^{g,\eps} \right\|_{L_\infty}^2   }$. For every $u\in [0,1]$, 
\begin{align*}
\left|K_t^{g,\eps}(u)\right|
&\leq  \int_0^t \left\| A_s^g-A_s^{g,\eps} \right\|_{L_\infty} \frac{1}{|\partial_u x_s^g(u)|} \frac{1}{t-s}\left|\int_s^t \partial_u x_r^g(u) \mathrm d\beta_r\right|\mathrm ds
\end{align*}
By inequality~\eqref{norme 1 Asg},
\begin{align*}
\left\| A_s^g-A_s^{g,\eps} \right\|_{L_\infty} \leq C \eps \left\| \partial_x A_s^g \right\|_{L_\infty}
\leq C \eps \|h\|_{\mathcal C^1}\left(1+\|\partial^{(2)}_u x_s^g\|_{L_\infty}\left\|\frac{1}{\partial_u x_s^g} \right\|_{L_\infty}\right).
\end{align*} 
Thus we obtain
\begin{align*}
\left\| K_t^{g,\eps} \right\|_{L_\infty}
&\leq C\eps  \|h\|_{\mathcal C^1} 
\left\{ 1+ \textstyle\sup_{r\leq T}\left\|\partial^{(2)}_u x_r^g \right\|_{L_\infty} \right\}
\left\{\textstyle\sup_{r\leq T} \left\| \frac{1}{\partial_u x_r^g}  \right\|_{L_\infty}+\textstyle\sup_{r\leq T}\left\| \frac{1}{\partial_u x_r^g}  \right\|_{L_\infty}^2\right\}
\\
&\quad \quad\cdot \int_0^t \frac{1}{t-s}\left\| \int_s^t \partial_u x_r^g(\cdot) \mathrm d\beta_r  \right\|_{L_\infty} 
\mathrm ds.
\end{align*}
By Hölder's equality, we obtain
\begin{align}
\label{eqprop:8.2}
\Ewb{\left\| K_t^{g,\eps} \right\|_{L_\infty}^2 }^{1/2} 
&\leq C\eps \|h\|_{\mathcal C^1} E_1E_2E_3,
\end{align}
where
\begin{align*}
E_1&:= 1+ \Ewb{\textstyle \sup_{r\leq T} \|\partial^{(2)}_u x_r^g  \|_{L_\infty}^8 }^{1/8};\\
E_2&:= \Ewb{\textstyle\sup_{r\leq T} \|\frac{1}{\partial_u x_r^g} \|_{L_\infty}^8 }^{1/8}+\Ewb{\textstyle\sup_{r\leq T} \|\frac{1}{\partial_u x_r^g} \|_{L_\infty}^{16} }^{1/8}; \\
E_3&:= \Ewb{\left( \int_0^t \frac{1}{t-s}\left\| \int_s^t \partial_u x_r^g(\cdot) \mathrm d\beta_r  \right\|_{L_\infty} 
\mathrm ds \right)^4}^{1/4}.
\end{align*}

Recall that  $g$ belongs to $\mathbf G^{3+\theta}$ and $f$ is of order $\alpha>\frac{7}{2}+\theta$.  Thus by~\eqref{Lp norm 7} and by~\eqref{Lp norm 8}
\begin{align*}
E_1 & \leq C(1+ \|g''' \|_{L_8} + \|g''\|_{L_\infty}^3+\|g'\|_{L_\infty}^3) ; \\
E_2&\leq C (1+\|g''\|_{L_\infty}^4+\|g'\|_{L_\infty}^4 + \| \textstyle\frac{1}{g'}\|_{L_\infty}^8) .
\end{align*}
Furthermore, $E_3 \leq E_{3,1}+E_{3,2}$, where 
\begin{align*}
E_{3,1}&:=  \Ewb{\Big( \int_0^t \frac{1}{t-s}\Big| \int_s^t \partial_u x_r^g(0) \mathrm d\beta_r  \Big|\mathrm ds \Big)^4}^{1/4}; \\
E_{3,2}&:= \Ewb{\Big( \int_0^t \frac{1}{t-s} \int_0^1 \Big| \int_s^t \partial^{(2)}_u x_r^g(v) \mathrm d\beta_r  \Big|\mathrm ds \mathrm dv\Big)^4}^{1/4}.
\end{align*}
By Hölder's inequality, we have
\begin{align*}
E_{3,1}
&\leq \Ewb{\Big( \int_0^t \frac{1}{|t-s|^{1/2}}\mathrm ds \Big)^3
\int_0^t \frac{1}{|t-s|^{5/2}}\Big| \int_s^t \partial_u x_r^g(0) \mathrm d\beta_r  \Big|^4 \mathrm ds }^{1/4} \\
&\leq C \; t^{3/8}  \Big(\int_0^t \frac{1}{|t-s|^{5/2}} \Ewb{\Big| \int_s^t \partial_u x_r^g(0) \mathrm d\beta_r  \Big|^4}\mathrm ds\Big)^{1/4}.
\end{align*}
By Burkholder-Davis-Gundy inequality, it follows that
\begin{align*}
E_{3,1}
&\leq C \; t^{3/8} \Big(   \int_0^t \frac{|t-s|^2}{|t-s|^{5/2}}\mathrm ds \Big)^{1/4} \Ewb{\sup_{r\leq T} |\partial_u x_r^g(0)|^4 }^{1/4}\leq C\sqrt{t} \|g'\|_{L_\infty},
\end{align*}
where the last inequality holds by~\eqref{Lp norm 2}. 
By the same computation, 
\begin{align*}
E_{3,2}
&\leq C \sqrt{t}\Ewb{\sup_{r\leq T} \int_0^1 |\partial_u^{(2)} x_r^g(v)|^4 \mathrm dv }^{1/4} \leq C \sqrt{t}\;(1+\|g''\|_{L_\infty}+\|g'\|^2_{L_\infty}),
\end{align*}
where the last inequality holds by~\eqref{Lp norm 5}. 
We deduce that $E_3 \leq C \sqrt{t}\;(1+\|g''\|_{L_\infty}+\|g'\|^2_{L_\infty})$.
By inequality~\eqref{eqprop:8.2} and the  estimates on $E_i$, for $i=1,2,3$, we finally get:
\begin{align*}
\Ewb{\left\| K_t^{g,\eps} \right\|_{L_\infty}^2 }^{1/2} 
&\leq C \sqrt{t}\eps \|h\|_{\mathcal C^1} C_2(g),
\end{align*}
where  $C_2(g) =  1+ \|g'''\|_{L_8}^3 + \|g''\|_{L_\infty}^{12}+\|g'\|_{L_\infty}^{12}+ \| \textstyle\frac{1}{g'}\|^{24}_{L_\infty} $. 
\end{proof}

As a conclusion of sections~\ref{sec:analysis_I1} and~\ref{sec:analysis_I2}, we have proved the following inequality.
\begin{coro}
\label{coro:estim_I}
Let $\phi$, $\theta$ and $f$ be as in Theorem~\ref{theo:Bismut pour phi assumptions}. 
Let $g$ and $h$ be $\mathcal G_0$-measurable random variables with values respectively in $\mathbf G^{3+\theta}$ and~$\Delta^1$. 
Let $(K_t^{g,\eps})_{t\in [0,T]}$ be defined by~\eqref{defin:Ktgeps}.
Then there is $C>0$ independent of  $g$, $h$ and $\theta$ such that $\mathbb P^0$-almost surely, for every $t\in (0,T]$ and $\eps \in (0,1)$, 
\begin{align}
\label{inegalite_I}
\left|\frac{\mathrm d}{\mathrm d\rho}_{\vert \rho=0} P_t \phi (\mu_0^{g+\rho g'  h}) \right|
&\leq  C \; \frac{\| \phi \|_{L_\infty}}{\eps^{3+2\theta}\sqrt{t}}    C_1(g) \|h\|_{\mathcal C^1}  \notag \\
&\quad + \frac{C}{\sqrt{t}}\;\eps \|h\|_{\mathcal C^1}   C_2(g) \;
\Ewb{\left| \int_0^1 \left[ \frac{\delta \phi}{\delta m} \right] (\mu_t^g) (x_t^g(u)) \frac{K_t^{g,\eps} (u)}{\left\| K_t^{g,\eps} \right\|_{L_\infty}}\mathrm du\right|^2   }^{1/2}, 
\end{align}
where $C_1(g)=1+\|g'''\|_{L_4}^2 + \|g''\|_{L_\infty}^{6} + \|g'\|_{L_\infty}^8 + \left\| \frac{1}{g'} \right\|_{L_\infty}^{8}$ and $C_2(g) =   1+ \|g'''\|_{L_8}^3 + \|g''\|_{L_\infty}^{12}+\|g'\|_{L_\infty}^{12}+ \left\| \frac{1}{g'} \right\|^{24}_{L_\infty}$. 
\end{coro}

\section{Proof of the main theorem}
\label{sec:end_proof}

Essentially, Corollary~\ref{coro:estim_I} states that we can control the gradient of $P_t \phi$ by the gradient of $\phi$. 
By iterating the inequality over successive time steps, we will conclude  the proof of Theorem~\ref{theo:Bismut pour phi assumptions}.

\begin{defin}
\label{def:K_t}
Let $\mathcal K_t$ be the set of $\mathcal G_t$-measurable random variables taking  their values $\mathbb P$-almost surely  in the set of continuous $1$-periodic functions $k:\R \to \R$ satisfying $\|k\|_{L_\infty}=1$. 
\end{defin}

\begin{prop}
\label{prop:decalage de s}
Let $\phi$, $\theta$,  $f$ and $g$ be as in Theorem~\ref{theo:Bismut pour phi assumptions}. Let $t,s \in [0,T]$ such that $t+s \leq T$. For  each $\mathcal G_s$-measurable function $h$ with values in $\Delta^1$ satisfying $\mathbb P$-almost surely $\|h\|_{\mathcal C^1} \leq 4$, there exists $C_g>0$  independent of $s$, $t$ and $h$ such that 
\begin{multline}
\label{passage s a t}
\E{\left|\int_0^1 \left[\frac{\delta P_{t}\phi}{\delta m}\right] (\mu_{s}^g) (x_{s}^g(u)) \; h' (u) \mathrm du \right| ^2 }^{\frac{1}{2}}
\leq C_g\frac{ \| \phi \|_{L_\infty}}{t^{2+\theta}} \\
 + \frac{1}{2^{3+\theta}} \sup_{k \in \mathcal K_{t+s}}\E{\left| \int_0^1 \left[ \frac{\delta  \phi}{\delta m} \right] (\mu_{t+s}^g) (x_{t+s}^g(u)) \; k(u)  \mathrm du\right|^2  }^{\frac{1}{2}}.
\end{multline}
\end{prop}

\begin{proof}
By equality~\eqref{formules_derivees_semigroupe}, 
\begin{align}
\frac{\mathrm d}{\mathrm d\rho}_{\vert \rho=0} P_t \phi (\mu_0^{g+\rho g'  h}) 
&=-\int_0^1 \frac{\delta P_t \phi}{\delta m}  (\mu_0^g)(g(u)) \; h'(u) \mathrm du \notag\\
&= -\int_0^1 \left[ \frac{\delta P_t \phi}{\delta m} (\mu_0^g) (g(u)) - \int_0^1\frac{\delta P_t \phi}{\delta m} (\mu_0^g) (g(u')) \mathrm du' \right] h'(u) \mathrm du \notag
\\
&=-\int_0^1 \left[\frac{\delta P_t \phi}{\delta m}\right]  (\mu_0^g)(g(u)) \; h'(u) \mathrm du,
\label{deriv_h}
\end{align}
where the second equality follows from the fact that $h$ is $1$-periodic and the last equality follows from~\eqref{notation moyenne}. 
Apply now inequality~\eqref{inegalite_I} with $\eps_0=\frac{1}{2^{\frac{7}{2}+\theta}}  \frac{\sqrt{t}}{C \|h\|_{\mathcal C^1} C_2(g)} $.
For every $\mathcal G_0$-measurable $g$ and $h$, 
\begin{multline*}
\left|\int_0^1 \left[\frac{\delta P_t\phi}{\delta m}\right] (\mu_0^g) (g(u)) \;h' (u) \mathrm du \right| 
\leq C\frac{  \| \phi \|_{L_\infty}}{t^{2+\theta}}C_3(g) \|h\|_{\mathcal C^1}^{4+2\theta}\\
 + \frac{1}{2^{\frac{7}{2}+\theta}} \;
\Ewb{\left| \int_0^1 \left[ \frac{\delta \phi}{\delta m} \right] (\mu_t^g) (x_t^g(u)) \frac{K_t^{g,\eps_0} (u)}{\left\| K_t^{g,\eps_0} \right\|_{L_\infty}}\mathrm du\right|^2   }^{\frac{1}{2}},
\end{multline*}
where $C_3(g)=C_1(g) C_2(g)^{3+2\theta}$. Moreover, $\Ewb{\cdot}=\E{\cdot \vert \mathcal G_0}$, since for any random variable $X$ on $\Omega$ and any $\mathcal G_0$-measurable  $Y$, $\E{XY}= \mathbb E^0 \Ewb{XY}=\mathbb E^0\left[ \Ewb{X} Y \right]$. 
Thus it follows from the latter inequality that:
\begin{multline*}
\E{\left|\int_0^1 \left[\frac{\delta P_t\phi}{\delta m}\right] (\mu_0^g) (g(u)) \;h' (u) \mathrm du \right|^2  \Big \vert \mathcal G_0 }
\leq C \frac{\| \phi \|_{L_\infty}^2}{t^{4+2\theta}}C_3(g)^2 \|h\|_{\mathcal C^1}^{8+4\theta}\\
+ \frac{2}{2^{7+2\theta}} \;
\E{\left| \int_0^1 \left[ \frac{\delta \phi}{\delta m} \right] (\mu_t^g) (x_t^g(u)) \frac{K_t^{g,\eps_0} (u)}{\left\| K_t^{g,\eps_0} \right\|_{L_\infty}}\mathrm du\right|^2 \Big \vert \mathcal G_0  }.
\end{multline*}
Now, consider a deterministic function $g$ and a $\mathcal G_s$-measurable $h$, where $s  \leq T -t$. 
Then, repeating the whole argument with  the $\mathcal G_s$-measurable variables $x_s^g$ and  $h$ instead of  $g$ and a $\mathcal G_0$-measurable $h$, respectively, we get:
\begin{multline*}
\E{\left|\int_0^1 \left[\frac{\delta P_t\phi}{\delta m}\right] (\mu_s^g) (x_s^g(u)) \;h' (u) \mathrm du \right|^2  \Big \vert \mathcal G_s }
\leq C\frac{ \| \phi \|_{L_\infty}^2}{t^{4+2\theta}}C_3(x_s^g)^2 \|h\|_{\mathcal C^1}^{8+4\theta}\\
+ \frac{1}{2^{6+2\theta}} \;
\E{\Bigg| \int_0^1 \left[ \frac{\delta \phi}{\delta m} \right] (\mu_{t+s}^{s,x_s^g}) (x_{t+s}^{s,x_s^g}(u)) \frac{K_{t+s}^{s,x_s^g,\eps_s} (u)}{\| K_{t+s}^{s,x_s^g,\eps_s} \|_{L_\infty}}\mathrm du\Bigg|^2 \Big \vert \mathcal G_s  },
\end{multline*}
where 
$x_{t+s}^{s,x_s^g}(u)$ denotes the value at time $t+s$ and at point $u$ of the unique solution to~\eqref{eq_SDE_diff_torus} which is equal to $x_s^g$ at time $s$ and where $\eps_s$ is $\mathcal G_s$-measurable. 
By strong uniqueness of~\eqref{eq_SDE_diff_torus}, we have the following flow property: $x_{t+s}^{s,x_s^g}=x_{t+s}^g$ and  $\mu_{t+s}^{s,x_s^g}=\mu_{t+s}^g$. Therefore, 
\begin{multline*}
\E{\left|\int_0^1 \left[\frac{\delta P_t\phi}{\delta m}\right] (\mu_s^g) (x_s^g(u)) \;h' (u) \mathrm du \right|^2  \Big \vert \mathcal G_s }
\leq C\frac{\| \phi \|_{L_\infty}^2}{t^{4+2\theta}}C_3(x_s^g)^2 \|h\|_{\mathcal C^1}^{8+4\theta}\\
+ \frac{1}{2^{6+2\theta}} \;
\E{\Bigg| \int_0^1 \left[ \frac{\delta \phi}{\delta m} \right] (\mu_{t+s}^g) (x_{t+s}^g(u)) \frac{K_{t+s}^{s,x_s^g,\eps_s} (u)}{\| K_{t+s}^{s,x_s^g,\eps_s} \|_{L_\infty}}\mathrm du\Bigg|^2 \Big \vert \mathcal G_s  }.
\end{multline*}
Remark that $u \mapsto K_{t+s}^{s,x_s^g,\eps_s} (u) / \| K_{t+s}^{s,x_s^g,\eps_s}\|_{L_\infty}$ belongs to $\mathcal K_{t+s}$. Thus, taking the expectation of the latter inequality, there is $C>0$ so that for every $\mathcal G_s$-measurable function $h$ satisfying $\|h\|_{\mathcal C^1} \leq 4$
\begin{multline*}
\E{\left|\int_0^1 \left[\frac{\delta P_t\phi}{\delta m}\right] (\mu_s^g) (x_s^g(u)) \;h' (u) \mathrm du \right|^2  }^{\frac{1}{2}}
\leq C\frac{\| \phi \|_{L_\infty}}{t^{2+\theta}} \E{ C_3(x_s^g)^2}^{\frac{1}{2}} \\
+ \frac{1}{2^{3+\theta}} \; \sup_{k \in \mathcal K_{t+s}}
\E{\left| \int_0^1 \left[ \frac{\delta \phi}{\delta m} \right] (\mu_{t+s}^g) (x_{t+s}^g(u)) \; k(u) \;\mathrm du\right|^2  }^{\frac{1}{2}}.
\end{multline*}
In order to prove inequality~\eqref{passage s a t}, it remains to show that there is $C_g$ such that $\E{C_3(x_s^g)^2}\leq C_g$. Since  $C_3(g)=C_1(g) C_2(g)^{3+2\theta}$, we have:
\begin{align*}
\E{C_3(x_s^g)^2}
&=\mathbb E \Big[ \Big(1+\|\partial_u^{(3)} x_s^g\|_{L_4}^2 + \|\partial_u^{(2)} x_s^g\|_{L_\infty}^{6} + \|\partial_u x_s^g\|_{L_\infty}^8 + \| \textstyle\frac{1}{\partial_u x_s^g} \|_{L_\infty}^{8} \Big) ^2  \notag \\
& \quad  \quad \quad\cdot  \left(  1+ \|\partial_u^{(3)} x_s^g\|_{L_8}^3 + \|\partial_u^{(2)} x_s^g\|_{L_\infty}^{12}+\|\partial_u x_s^g\|_{L_\infty}^{12}+ \| \textstyle\frac{1}{\partial_u x_s^g}\|^{24}_{L_\infty} \right) ^{6+4\theta} \Big].
\end{align*}
We refer to~\eqref{Lp norm 5}, \eqref{Lp norm 7} and~\eqref{Lp norm 8} to argue that  the r.h.s.\! is bounded by a  uniform constant in $s \in [0,T]$  and depending polynomially on $\|g'''\|_{L_\infty}$, $\|g''\|_{L_\infty}$, $\|g'\|_{L_\infty}$ and $\| \textstyle\frac{1}{g'}\|_{L_\infty}$. The constant is finite since $g$ belongs to $\mathbf G^{3+\theta}$.  
\end{proof}

\begin{coro}
\label{coro:2t_et_t}
Let $\phi$, $\theta$, $f$ and $g$ satisfy the same assumption as in Proposition~\ref{prop:decalage de s}. 
Let $t,s \in [0,T]$ such that $2t+s \leq T$. 
For any $h:\R \to \R$ be a $\mathcal G_s$-measurable random variable with values in $\Delta^1$ satisfying $\mathbb P$-almost surely $\|h\|_{\mathcal C^1} \leq 4$, there exists  $C_g>0$ independent of $s$, $t$ and $h$ such that 
\begin{multline}
\label{ineg_coro}
\E{\left|\int_0^1 \left[\frac{\delta P_{2t}\phi}{\delta m}\right] (\mu_{s}^g) (x_{s}^g(u)) \; h' (u) \mathrm du \right| ^2 }^{\frac{1}{2}}
\leq C_g\frac{ \| \phi \|_{L_\infty}}{t^{2+\theta}} \\
 + \frac{1}{2^{3+\theta}} \sup_{k \in \mathcal K_{t+s}}\E{\left| \int_0^1 \left[ \frac{\delta  P_t \phi}{\delta m} \right] (\mu_{t+s}^g) (x_{t+s}^g(u)) \; k(u)  \mathrm du\right|^2  }^{\frac{1}{2}}.
\end{multline}
\end{coro}

\begin{proof}
We get the above inequality by applying~\eqref{passage s a t} to $P_t \phi$ instead of $\phi$. We note that $P_t (P_t \phi)=P_{2t} \phi$ and that $ \| P_t \phi \|_{L_\infty} \leq  \| \phi \|_{L_\infty}$. 
\end{proof}

Fix $t_0 \in (0,T]$. 
For every $t\in (0,t_0]$,  define 
\begin{align*}
S_t:=\sup_{k \in \mathcal K_{t_0-t}}\E{\left| \int_0^1 \left[ \frac{\delta P_t \phi}{\delta m} \right] (\mu_{t_0-t}^g) (x_{t_0-t}^g(u)) \; k(u) \;  \mathrm du\right|^2  }^{\frac{1}{2}},
\end{align*}
where $ \mathcal K_{t_0-t}$ is defined by Definition~\ref{def:K_t}. 

\begin{prop}
\label{prop:passage_sup 2t a t}
Let $\phi$, $\theta$, $f$ and $g$ be as in Theorem~\ref{theo:Bismut pour phi assumptions}. 
For every $t\in (0,\frac{t_0}{2}]$, we have:
\begin{align}
\label{ineq:S2t}
S_{2t} \leq  C_g\frac{\| \phi \|_{L_\infty}}{t^{2+\theta}} + \frac{1}{2^{3+\theta}}S_t.
\end{align}
\end{prop}

\begin{proof}
Fix $t\in (0,\frac{t_0}{2}]$ and $k \in \mathcal K_{t_0-2t}$. Hence $k$ is a continuous $1$-periodic function and a $\mathcal G_{t_0-2t}$-measurable random variable so that $\mathbb P$-almost surely, $\|k \|_{L_\infty}=\sup_{u \in [0,1]}|k(u)|=1$. 

Let us denote by $h$ the map defined for every $u \in \R$ by $h(u):= \int_0^u (k(v)-\overline{k}) \mathrm dv$, where $\overline{k}=\int_0^1 k(v) \mathrm dv$. 
We check that $h$ is a $\mathcal G_{t_0-2t}$-measurable  $1$-periodic $\mathcal C^1$-function.
 Moreover, $\|h\|_{L_\infty}\leq 2$ and $\|\partial_u h\|_{L_\infty} \leq 2$; thus $\|h\|_{\mathcal C^1}\leq 4$. Therefore, the assumptions of Corollary~\ref{coro:2t_et_t} are satisfied, with $s=t_0-2t$. We apply~\eqref{ineg_coro}  with $s=t_0-2t$:
\begin{align*}
\E{\left|\int_0^1 \left[\frac{\delta P_{2t}\phi}{\delta m}\right] (\mu_{t_0-2t}^g) (x_{t_0-2t}^g(u)) \; h' (u) \;\mathrm du \right| ^2 }^{\frac{1}{2}}
\leq C_g\frac{ \| \phi \|_{L_\infty}}{t^{2+\theta}} + \frac{1}{2^{3+\theta}}S_t. 
\end{align*}
Moreover, $h'(u)=k(u)-\overline{k}$ and 
by definition~\eqref{notation moyenne}, $\int_0^1 \left[\frac{\delta P_{2t}\phi}{\delta m}\right] (\mu_{t_0-2t}^g) (x_{t_0-2t}^g(u)) \cdot \overline{k} \;\mathrm du =0$. Thus
\begin{align*}
\E{\left|\int_0^1 \left[\frac{\delta P_{2t}\phi}{\delta m}\right] (\mu_{t_0-2t}^g) (x_{t_0-2t}^g(u)) \; k(u) \; \mathrm du \right| ^2 }^{\frac{1}{2}}
\leq C_g\frac{ \| \phi \|_{L_\infty}}{t^{2+\theta}} + \frac{1}{2^{3+\theta}}S_t, 
\end{align*}
and by taking the supremum over all $k$ in $\mathcal K_{t_0-2t}$, we get $S_{2t} \leq  C_g\frac{ \| \phi \|_{L_\infty}}{t^{2+\theta}} + \frac{1}{2^{3+\theta}}S_t$. 
\end{proof}

We complete the proof of Theorem~\ref{theo:Bismut pour phi assumptions}. 

\begin{proof}[Proof (Theorem~\ref{theo:Bismut pour phi assumptions})]
It follows from Proposition~\ref{prop:passage_sup 2t a t} that for every $t\in (0,\frac{t_0}{2}]$, 
\begin{align*}
(2t)^{2+\theta} S_{2t} \leq 2^{2+\theta} C_g \| \phi \|_{L_\infty} + \frac{1}{2} t^{2+\theta} S_t. 
\end{align*}
Therefore, denoting by $\mathbf S:=\sup_{t\in (0,t_0]} t^{2+\theta} S_t$, we have $\mathbf S \leq 2^{2+\theta} C_g \| \phi \|_{L_\infty} + \frac{1}{2} \mathbf S$. Since $\mathbf S$ is finite, we obtain $\mathbf S \leq 2^{3+\theta} C_g \| \phi \|_{L_\infty}$.  Thus for every $t_0 \in (0,T]$, $t_0^{2+\theta} S_{t_0} \leq 2^{3+\theta} C_g \| \phi \|_{L_\infty}$. 
Therefore, for any deterministic $1$-periodic function $k:\R\to \R$ and for every $t\in (0,T]$, we have
\begin{align*}
\left| \int_0^1 \left[ \frac{\delta P_t \phi}{\delta m} \right] (\mu_0^g) (g(u)) \; k(u)\;  \mathrm du\right|\leq  C_g \frac{\| \phi \|_{L_\infty}}{t^{2+\theta}}\|k\|_{L_\infty}.
\end{align*}
Let $h \in \Delta^1$. Thus $k=\partial_u \left( \frac{h}{g'} \right)$ is a $1$-periodic function and we deduce that 
\begin{align*}
\left| \int_0^1 \left[ \frac{\delta P_t \phi}{\delta m} \right] (\mu_0^g) (g(u))\; \partial_u \left( \frac{h}{g'} \right) (u) \; \mathrm du\right|
\leq  C_g \frac{\| \phi \|_{L_\infty}}{t^{2+\theta}}\left\|\partial_u \left( \frac{h}{g'} \right)\right\|_{L_\infty} 
&\leq C_g \frac{\| \phi \|_{L_\infty}}{t^{2+\theta}}\left\| h\right\|_{\mathcal C^1},
\end{align*}
for a new constant $C_g$. 
Applying equality~\eqref{deriv_h} with $\frac{h}{g'}$ instead of $g'$, we obtain
\begin{align*}
\left|\frac{\mathrm d}{\mathrm d\rho}_{\vert \rho=0} P_t \phi (\mu_0^{g+\rho  h})\right| 
= \left| \int_0^1 \left[ \frac{\delta P_t \phi}{\delta m} \right] (\mu_0^g) (g(u))\; \partial_u \left( \frac{h}{g'} \right) (u) \; \mathrm du\right|
\leq C_g \frac{\| \phi \|_{L_\infty}}{t^{2+\theta}}\left\| h\right\|_{\mathcal C^1},
\end{align*}
 which concludes the proof of the theorem. 
\end{proof}

\appendix

\section{Appendix}

\subsection{Density functions and quantile functions on the torus}
\label{subsec:density and quantile}

We define the set of positive densities on the torus.

\begin{defin}
\label{defin_density}
Let $\mathcal P^+$ be the set of continuous functions $p:\tor \to \R$ such that for every $x \in \tor$, $p(x)>0$ and $\int_\tor p =1$.  $\mathcal P^+$ can also be seen as the set of $2\pi$-periodic and continuous functions $p:\R \to (0,+\infty)$  such that $\int_0^{2\pi} p(x) \mathrm dx=1$. 
\end{defin}

Let $p \in \mathcal P^+$ and $x_0 \in \tor$ be an arbitrary point on the torus. Define a cumulative distribution function (c.d.f.) $F_0:\R \to \R$ by $F_0(x)=\int_{x_0}^x p(y) \mathrm dy$, 
 for each $x\in \R$. 
Since $p$ is $2\pi$-periodic and $\int_0^{2\pi} p = 1$, $F_0$ satisfies $F_0(x+2\pi)=F_0(x)+1$ for each $x\in \R$.
It follows from the continuity and  from the positivity of $p$ that  $F_0$ is a $\mathcal C^1$-function and for every $x\in \R$, $F_0'(x)=p(x)>0$, so that $F_0$ is strictly increasing. Therefore, the inverse function $g_0:=F_0^{-1}:\R \to \R$ is well defined. The following properties of $g_0$ are straightforward:
\begin{itemize}
\item[-] for every $x\in \R$, $g_0 \circ F_0(x)=x$ and for every $u\in \R$, $F_0\circ g_0 (u)=u$;
\item[-] $g_0$ is a strictly increasing $\mathcal C^1$-function and for each $u\in \R$, $g_0'(u)=\frac{1}{p(g_0(u))}$;
\item[-] $g_0(0)=x_0$ and for every $u \in \R$, $g_0(u+1)=g_0(u)+2\pi$ (we  say that $g_0$ is pseudo-periodic);
\item[-] $g_0':\R \to \R$ is positive everywhere and is a $1$-periodic function. 
\end{itemize}

\begin{prop}
\label{prop:bijection densite quantile}
There is a one-to-one correspondence between the set $\mathcal P^+$ and the set $\mathbf G^1$ of Definition~\ref{defin_g}.  
\end{prop}

\begin{proof}
Let $\iota:\mathcal P^+ \to \mathbf G^1$ be the map such that for every $p\in \mathcal P^+$, $\overline{g}=\iota(p)$ is the equivalence class
given by the above construction. 
We show that $\iota$ is one-to-one. 

First, $\iota$ is injective. Indeed, let $p_1, p_2 \in \mathcal P^+$ such that $\iota(p_1)=\iota(p_2)$. Let $x_0 \in \tor$ and define, for $i=1,2$, $F_i(x) =\int_{x_0}^x p_i(y) \mathrm dy$ and $g_i=F_i^{-1}$. 
Then by construction $\overline{g_1}=\iota(p_1)=\iota(p_2)=\overline{g_2}$. Therefore, there is $c\in \R$ such that $g_2(\cdot)=g_1(\cdot+c)$. Thus for every $x\in \R$, 
\begin{align*}
F_1(x)=F_1(g_2\circ F_2(x))=F_1(g_1(F_2(x)+c))=F_2(x)+c. 
\end{align*}
Thus $F_1$ and $F_2$ share the same derivative: $p_1=p_2$. 

Second, $\iota$ is surjective. Let $\overline{g}\in \mathbf G^1$ and $g$ be a representative of the class $\overline{g}$. It is a $\mathcal C^1$-function such that $g'(u)>0$ for every $u\in \R$ and, since $g(u+1)=g(u)+2\pi$ for every $u\in \R$, $g'$ is $1$-periodic. 
Define $F:=g^{-1}:\R \to \R$. In particular, $F$ is a $\mathcal C^1$-function such that $F'>0$ and for every $x\in \R$, $F(x+2\pi)=F(x)+1$. Thus $p:= F'$ is a continuous function with  values in $(0,+\infty)$ and for every $x\in \R$, $p(x)=\frac{1}{g'(F(x))}$. Thus for every $x\in \R$, $p(x+2\pi)=p(x)$ and $\int_0^{2\pi} p=1$. 
Therefore $p$ belongs to $\mathcal P^+$. 
We check that $\overline{g}=\iota(p)$. Let $x_0$ be an arbitrary point in~$\tor$, $F_0$ be defined by $F_0(x)=\int_{x_0}^x p(y) \mathrm dy$ and $g_0:=F_0^{-1}$. Since $F_0'=p=F'$, there is $c\in \R$ such that $F_0(\cdot)=F(\cdot)+c$. Therefore, $g_0(\cdot)=g(\cdot+c)$, whence $g_0 \sim g$. This completes the proof. 
\end{proof}

\subsection{Properties of the diffusion on the torus}
\label{app_well_posed}

We prove in this paragraph  several properties of the diffusion constructed in Paragraph~\ref{parag:diff_torus} of the main text. First, let us show propositions~\ref{prop:exist, uniq, cont, growth} and~\ref{prop:differentiability}. 

\begin{proof}[Proof (Proposition~\ref{prop:exist, uniq, cont, growth})]
Strong existence and uniqueness hold by a fixed-point argument: we refer to the proof of~\cite[Prop.\! 3]{marx20}. The additional Brownian motion $(\beta_t)_{t\in [0,T]}$ does not add any difficulty to that proof, since it does not depend  on the initial condition $g$ nor on the variable~$u$. 
By a standard application of Kolmorogov's Lemma, we also obtain the existence of a version in $\mathcal C(\R \times [0,T])$, see the proof of~\cite[Prop.\! 5]{marx20}. Moreover, the fact that the map $u \mapsto x^g_t(u)$ is strictly increasing is obtained by the study of the process $(x^g_t(u_2)-x^g_t(u_1))_{t \in [0,T]}$ for every $u_1<u_2$ as in~\cite[Prop.\! 6]{marx20}. The fact that it holds $\mathbb P^W \otimes \mathbb P^\beta$-almost surely, for every $0\leq u_1<u_2 \leq 1$ follows from the continuity of  $x^g$, see~\cite[Cor.\! 7]{marx20}.
\end{proof}
 
\begin{proof}[Proof (Proposition~\ref{prop:differentiability})]
Assume that $g$ belongs to $\mathscr G^{1+\theta}$ and that $f$ is of order $\alpha>\frac{3}{2}+\theta$. This second assumption ensures that $\sum_{k\in \Z} \crochetk^{2+2\theta} |f_k|^2$ converges. By differentiating formally (w.r.t.\! variable  $u$) equation~\eqref{eq_SDE_diff_torus}, consider a solution $(z_t(u))_{t \in [0,T], u\in \R}$ to:
\begin{align}
\label{eq_derivative}
z_t(u)=g'(u) + \sum_{k \in \Z} f_k \int_0^t z_s(u) \Re \left(-ik e^{-ik x_s^g(u)} \mathrm dW^k_s \right).
\end{align}
Using the fact that $g'$ is $\theta$-Hölder continuous and that  $\sum_{k\in \Z} \crochetk^{2+2\theta} |f_k|^2 <+\infty$, we prove by standard arguments that for each $u \in \R$, the solution $(z_t(u))_{t \in [0,T]}$ exists, is unique and that the map $u \mapsto z_\cdot(u) \in L_2(\Omega, \mathcal C[0,T])$ is $\theta'$-Hölder continuous for each $\theta'<\theta$. 

Furthermore, for each $u \in \R$,  $\frac{x^g_\cdot(u+\eps)-x^g_\cdot(u)}{\eps} \rightarrow_{\eps \to 0} z_\cdot (u) $ in $L_2(\Omega, \mathcal C[0,T])$. Indeed, define for each $t \in [0,T]$ and $\eps \neq 0$, $E_\eps(t) := \Ewb{\sup_{s\leq t} | \frac{x^g_s(u+\eps)-x^g_s(u)}{\eps} -z_s(u) |^2}$. 
We easily get a constant $C$ depending on $\|g\|_{\mathcal C^{1+\theta}}$ and on $\sum_{k\in \Z} \crochetk^{2+2\theta} |f_k|^2$ such that for each $t \in [0,T]$, $E_\eps (t) \leq C |\eps|^{2\theta} + C\int_0^t E_\eps (s) \mathrm ds$. By Gronwall's Lemma, it follows that $E_\eps(T) \leq C |\eps|^{2\theta}$, thus $E_\eps(T) \to 0$. Therefore, using the continuity of $z$, we get almost surely for every $u \in \R$, for every $\eps \neq 0$ and for every $t \in [0,T]$:
\begin{align*}
\frac{x^g_t(u+\eps)-x^g_t(u)}{\eps} = \int_0^1 z_t(u+\lambda \eps) \mathrm d\lambda.
\end{align*}
Thus almost surely, $\partial_u x^g_t(u)=z_t(u)$ for every $u \in \R$ and $t \in [0,T]$ and furthermore it is given by the exponential form~\eqref{exponential explicit partial xtg}. The statements for higher derivatives are obtained similarly. For a detailed version of this proof with every computation of the inequalities mentioned above, see~\cite[Lemmas II.12 and II.13]{marx_thesis}. 
\end{proof}

By the previous proof, we know that $\partial_u x^g$ satisfies equation~\eqref{eq_derivative}. It follows that we can control the $L_p$-norms of $\partial_u x^g$ and of higher derivatives with respect to the initial condition $g$. 

\begin{lemme}
\label{lem_control_moments}
Let $\theta \in (0,1)$, $j\geq 1$, $\kappa >0$ and let $g$ be a $\mathcal G_0$-measurable random variable belonging $\mathbb P^0$-a.s. to $\mathbf G^{\kappa}$. Assume that $f$ is of order $\alpha>\kappa+\frac{1}{2}$.
Then there are constants\footnote{The constants $C_p$ and $C_{p,j}$ appearing in this lemma are independent of $\theta$ and of $g$.} $C_p$ and $C_{p,j}$ such that $\mathbb P^0$-a.s.
\begin{align}
\label{Lp norm 1}
\text{if } \kappa \geq 1+\theta, \quad &\Ewb{\sup_{t\leq T}    \left\| \partial_u x_t^g \right\|_{L_p[0,1]}^p }
\leq C_p \;\|g'\|_{L_p[0,1]}^p,\\
\label{Lp norm 2}
\text{if } \kappa \geq 1+\theta, \quad&\Ewb{\sup_{t\leq T}    \left| \partial_u x_t^g(0) \right|^p }
\leq C_p \; g'(0)^p, \\
\label{Lp norm 5}
\text{if } \kappa \geq j+\theta, \quad&\Ewb{\sup_{t\leq T}    \left\| \partial_u^{(j)} x_t^g \right\|_{L_p[0,1]}^p}
\leq C_{p,j} \left\{ 1+ \|\partial_u^{(j)} g\|_{L_p[0,1]}^p + \sum_{k=1}^{j-1} \|\partial_u^{(k)} g\|_{L_\infty[0,1]}^{jp} \right\},\\
\label{Lp norm 7}
\text{if } \kappa \geq j+1+\theta, \quad&\Ewb{\sup_{t\leq T}   \left\| \partial_u^{(j)} x_t^g\right\|_{L_\infty}^p }
\leq C_{p,j} \left\{ 1 + \|\partial_u^{(j+1)} g\|_{L_p}^p +  \sum_{k=1}^{j} \|\partial_u^{(k)} g\|_{L_\infty}^{(j+1)p} \right\},\\
\label{Lp norm 8}
\text{if }  \kappa \geq 2+\theta, \quad&\Ewb{\sup_{t\leq T}   \left\| \frac{1}{\partial_u x_t^g} \right\|_{L_\infty}^p }
\leq C_p \left\{ 1+ \frac{1}{g'(0)^p} + \left\|  \frac{1}{g'} \right\|_{L_{4p}}^{4p} + \| g'' \|_{L_{2p}}^{2p} + \|g'\|_{L_\infty}^{4p} \right\}.
\end{align}
\end{lemme}

\begin{proof}
The proofs of those five inequalities are pretty similar. Let us show inequality~\eqref{Lp norm 1} in details. 
Let $M>0$ be chosen large enough so that $\int_0^1 |g'(v)|^p \mathrm dv <M$. Define the stopping time $\sigma^M:= \inf \{ t \geq 0: \int_0^1| \partial_u x_t ^g(v) |^p \mathrm dv \geq M  \}$. Since $\partial_u x^g$  satisfies~\eqref{eq_derivative}, it follows from Burkholder-Davis-Gundy inequality that for every $t\in [0,T]$, 
\begin{multline}
\mathbb E^W \mathbb E^\beta \bigg[\sup_{s\leq t\wedge \sigma^M}  \int_0^1 \left| \partial_u x_s^g(v) \right|^p \mathrm dv  \bigg]\\
\begin{aligned}
&\leq C_p \|g'\|_{L_p[0,1]}^p 
+C_p \int_0^1\mathbb E^W \mathbb E^\beta \bigg[   \bigg|  \sum_{k \in \Z} f_k^2 \int_0^{t \wedge \sigma^M} |\partial_u x_s^g(v)|^2 \;|(-ik) e^{-ik x_s^g(v)}|^2 \mathrm ds \bigg|^{p/2}   \bigg]\mathrm dv\\
&\leq C_p \|g'\|_{L_p[0,1]}^p 
+C_p \Big(\sum_{k \in \Z} f_k^2 k^2  \Big)^{p/2} T^{p/2-1} \int_0^t\mathbb E^W \mathbb E^\beta \bigg[   \sup_{r\leq s \wedge \sigma^M} \int_0^1 |\partial_u x_r^g(v)|^p  \mathrm dv    \bigg]\mathrm ds.
\end{aligned}
\label{calcul norme Lp bis}
\end{multline}
Since $f$ is of order $\alpha> \frac{3}{2}$, $\sum_{k \in \Z} f_k^2 k^2$ converges. 
By Gronwall's Lemma, we deduce that there is a constant $C_p$ such that for every $M >\|g'\|_{L_p[0,1]}^p$, 
\begin{align}
\label{gronwall_arrete}
\Ewb{\sup_{t\leq \sigma^M}  \int_0^1  \left| \partial_u x_t^g(v) \right|^p \mathrm dv}
\leq C_p \|g'\|_{L_p[0,1]}^p.
\end{align}
Moreover, 
\begin{align*}
\mathbb P^W \otimes \mathbb P^\beta \left[ \sigma^M <T \right]
\leq \mathbb P^W \otimes \mathbb P^\beta \left[ \sup_{t\leq \sigma^M}  \int_0^1  \left| \partial_u x_t^g(v) \right|^p \mathrm dv \geq M \right] 
\leq \frac{C_p}{M} \|g'\|_{L_p[0,1]}^p, 
\end{align*}
hence we deduce $\mathbb P^W \otimes \mathbb P^\beta \left[ \bigcup_{M} \{ \sigma^M = T\} \right]=1$. Thus,  we let $M$ tend to $+\infty$ in~\eqref{gronwall_arrete} and inequality~\eqref{Lp norm 1} follows by Fatou's Lemma.

Inequality~\eqref{Lp norm 2} is obtained by writing equation~\eqref{eq_derivative} at $u=0$ and by mimicking the previous proof. 
We get inequality~\eqref{Lp norm 5} by induction over $j \geq 1$; we write for each $j\geq 1$ the equation satisfied by $\partial_u^{(j)} x^g$ and we apply successively Burkholder-Davis-Gundy inequality and Gronwall's Lemma. Similarly, we  also prove that $\mathbb P^0$-a.s.
\begin{align}
\label{Lp norm 6}
\Ewb{\sup_{t\leq T}   \left| \partial_u^{(j)} x_t^g(0) \right|^p }
\leq C_{p,j} \left\{ 1+ |\partial_u^{(j)} g(0)|^p + \sum_{k=1}^{j-1} \|\partial_u^{(k)} g\|_{L_\infty[0,1]}^{jp} \right\}. 
\end{align}
Moreover, since for every $1$-periodic function $f$,  $\|f\|_{L_\infty} \leq |f(0)|+ \int_0^1 |\partial_u f(v)| \mathrm dv $, we have for every $p \geq 2$, 
\begin{align*}
\Ewb{\sup_{t\leq T}   \left\| \partial_u^{(j)} x_t^g\right\|_{L_\infty}^p }
\leq C_p \Ewb{\sup_{t\leq T}  | \partial_u^{(j)} x_t^g (0)|^p }
+ C_p \Ewb{\sup_{t\leq T}  \int_0^1 \left| \partial_u^{(j+1)} x_t^g(v) \right|^p  \mathrm dv },
\end{align*}
which together with~\eqref{Lp norm 5} and~\eqref{Lp norm 6} leads to inequality~\eqref{Lp norm 7}. 

Furthermore, recall that $\mathbb P^W \otimes \mathbb P^\beta$-a.s., for every $t\in [0,T]$ and for every $v \in [0,1]$, $\partial_u x_t^g(v)>0$. 
By Itô's formula, it follows from equation~\eqref{eq_derivative} that
\begin{align*}
\frac{1}{\partial_u x_t^g(v) }
&= \frac{1}{g'(v)} - \sum_{k \in \Z} f_k \int_0^t \frac{1}{\partial_u x_s^g(v)} \Re \left(-ik e^{-ik x_s^g(v)} \mathrm dW^k_s \right) +\sum_{k \in \Z} f_k^2 k^2 \int_0^t \frac{1}{\partial_u x_s^g(v)} \mathrm ds.
\end{align*}
Again, applying successively Burkholder-Davis-Gundy inequality and Gronwall's Lemma, 
we get $\mathbb P^0$-a.s.
\begin{align*}
\Ewb{\sup_{t\leq T}  \int_0^1  \frac{1}{\left| \partial_u x_t^g(v) \right|^p} \mathrm dv}
&\leq C_p \left\|\frac{1}{g'}\right\|_{L_p[0,1]}^p, \\
\Ewb{\sup_{t\leq T}    \frac{1}{\left| \partial_u x_t^g(0) \right|^p} }
&\leq \frac{C_p}{g'(0)^p}.
\end{align*}
Then, using once more the $1$-periodicity of $\partial_u x^g$ and $\|f\|_{L_\infty} \leq |f(0)|+ \int_0^1 |\partial_u f(v)| \mathrm dv $, we obtain
\begin{align*}
\Ewb{\sup_{t\leq T}   \left\| \frac{1}{\partial_u x_t^g} \right\|_{L_\infty}^p }
\leq C_p \Ewb{\sup_{t\leq T}    \frac{1}{|\partial_u x_t^g(0)|^p} }
+C_p \Ewb{\sup_{t\leq T} \int_0^1   \left| \frac{\partial_u^{(2)}x_t^g(v)}{\partial_u x_t^g(v)^2} \right|^p \mathrm dv }.
\end{align*}
and we finally prove~\eqref{Lp norm 8} by using previous inequalities together with Cauchy-Schwarz inequality. 
\end{proof}

We obtain the same estimates for the solution to the parametric SDE~\eqref{eq_SDE_z}. 

\begin{prop}
\label{prop_para}
Let $f$ be of order $\alpha> \frac{3}{2}+\theta$ for some $\theta \in (0,1)$. Then there is a collection $(Z_t^x)_{t\in [0,T], x\in \R}$ such that for each $x\in \R$, $(Z_t^x)_{t\in [0,T]}$ is the unique solution to~\eqref{eq_SDE_z}. Moreover, 
$\mathbb P^W \otimes\mathbb P^\beta$-almost surely, the map $x\mapsto Z_t^x$ is differentiable on $\R$ for every $t\in [0,T]$ and there is a continuous version of the map continuous version of the map $(t,x)\in [0,T] \times \R \mapsto \partial_x Z_t^x $. 
Furthermore, for every $p \geq 2$, there is  $C_p>0$ such that for every $x,y \in \R$, 
\begin{align}
\label{différence Ztx Zty}
\Ewb{\sup_{t\leq T}\left|Z_t^x - Z_t^y \right|^p} &\leq C_p |x-y|^p \\
\label{borne sur pZtx}
\Ewb{\sup_{t\leq T}| \partial_x Z_t^x|^p} &\leq C_p, \\
\label{différence pZtx pZty}
\Ewb{\sup_{t\leq T}| \partial_x Z_t^x - \partial_x Z_t^y|^p} &\leq C_p |x-y|^{p \theta}. 
\end{align}
If $\alpha> \frac{5}{2}$, then~\eqref{différence pZtx pZty} holds with $\theta$ equal to $1$.  
\end{prop}

\begin{proof}
Let us focus on the proof of inequality~\eqref{différence pZtx pZty}, which is the only one which is new  with respect to what has been proved in Lemma~\ref{lem_control_moments} for $(x_t^g(u))_{t\in [0,T]}$. 
For every $x,y \in \R$, 
\begin{multline*}
\Ewb{\sup_{s\leq t}| \partial_x Z_s^x - \partial_x Z_s^y|^p}
\leq C_p \Ewb{\sup_{s\leq t}\left| \sum_{k \in \Z} f_k \!\int_0^s (\partial_x Z_r^x-\partial_x Z_r^y) \Re \left(-ik e^{-ik Z_r^x} \mathrm dW^k_r \right) \right|^p} \\
 + C_p \Ewb{\sup_{s\leq t}\left| \sum_{k \in \Z} f_k \int_0^s \partial_x Z_r^y \Re \left(-ik (e^{-ik Z_r^x}-e^{-ik Z_r^y}) \mathrm dW^k_r \right) \right|^p}.
\end{multline*}
By Burkholder-Davis-Gundy inequality,  we deduce that
\begin{align*}
\Ewb{\sup_{s\leq t}| \partial_x Z_s^x - \partial_x Z_s^y|^p}
&\leq C_p  \Big( \sum_{k \in \Z} f_k^2k^2 \Big)^{p/2}  \int_0^t\Ewb{\sup_{r\leq s}|\partial_x Z_r^x-\partial_x Z_r^y  |^p} \mathrm dr \\
&\quad + C_p \Big( \sum_{k \in \Z} f_k^2 |k|^{2+2\theta} \Big)^{p/2} \int_0^t \Ewb{|\partial_x Z_s^y|^p  \left| Z_s^x - Z_s^y\right|^{p \theta}}\mathrm ds.
\end{align*}
Furthermore, by~\eqref{différence Ztx Zty} and~\eqref{borne sur pZtx}, we obtain for every $s\in [0,T]$ and for every $x$ and $y$,  
\begin{align*}
\Ewb{|\partial_x Z_s^y|^p  \left| Z_s^x - Z_s^y\right|^{p \theta}}
\leq \Ewb{ |\partial_x Z_s^y|^{\frac{p}{1-\theta}}  }^{1-\theta} \Ewb{\left| Z_s^x - Z_s^y\right|^{p} }^{\theta}
\leq C_{p,\theta} |x-y|^{p\theta}. 
\end{align*}
Since $f$ is of order $\alpha>\frac{3}{2}+\theta$,  the series $ \sum_{k \in \Z} f_k^2|k|^{2+2\theta}$ is finite. By Gronwall's Lemma, we deduce~\eqref{différence pZtx pZty}. 
By Kolmogorov's Lemma, it follows from~\eqref{différence pZtx pZty} (with $p$ larger than $\frac{1}{\theta}$) that there is a continuous version of $(t,x) \mapsto \partial_x Z_t^x$. 
\end{proof}

\begin{proof}[Proof (Proposition~\ref{prop:kunita_identities})]
Fix $g \in \mathbf G^1$ and $u\in \R$. The processes $(Z_t^{g(u)})_{t\in [0,T]}$ and $(x_t^g(u))_{t\in [0,T]}$ both satisfy the same SDE~\eqref{eq_SDE_diff_torus} with same initial condition. Since $f$ is of order $\alpha> \frac{3}{2}$, pathwise uniqueness holds for this equation. Hence for every $u\in \R$, $Z_t^{g(u)}=x_t^g(u)$ holds almost surely. Moreover, since $u\in \R \mapsto x_t^g(u)$ and $x \in \R\mapsto Z_t^x$ are $\mathbb P^W \otimes \mathbb P^\beta$-almost surely continuous, and $g$ is continuous, we deduce that $Z_t^{g(u)}=x_t^g(u)$ holds almost surely for every $u\in \R$. 

Moreover, differentiating the equations~\eqref{eq_SDE_diff_torus} and~\eqref{eq_SDE_z}, we have $\mathbb P^W \otimes \mathbb P^\beta$-almost surely, for every $u\in [0,1]$ and for every $t\in [0,T]$, 
\begin{align*}
\partial_u x_t^g(u) &=g'(u) \exp\Big( \sum_{k \in \Z} f_k \int_0^t  \Re \left(-ik e^{-ik x_s^g(u)} \mathrm dW^k_s \right)-\frac{t}{2} \sum_{k \in \Z} f_k^2 k^2 \Big); \\
\partial_x Z_t^{g(u)} &=\exp\Big( \sum_{k \in \Z} f_k \int_0^t  \Re \Big(-ik e^{-ik Z_s^{g(u)}} \mathrm dW^k_s \Big)-\frac{t}{2} \sum_{k \in \Z} f_k^2 k^2 \Big).
\end{align*}
Using equality~\eqref{egalite Zx}, we get: $\partial_u x_t^g(u) =g'(u)\partial_x Z_t^{g(u)}$. 
\end{proof}

\subsection{Differential calculus on the Wasserstein space}
\label{app_diff_calculus}

Let us recall a few results about differentiation of real-valued functions $\phi$ defined on $\mathcal P_2(\R)$. We refer to~\cite{lions_college}, \cite[Chap.\! 5]{carmonadelarue18} or \cite{cardaliaguet13} for a complete introduction to those differential calculus. 

\textbf{Lions-derivative or L-derivative. }
Let $(\Omega, \mathcal F, \mathbb P)$ be a probability space rich enough 
so that for any probability measure $\mu$ on any Polish space, we can construct on $(\Omega, \mathcal F, \mathbb P)$ a random variable with distribution~$\mu$; a sufficient condition is that $(\Omega, \mathcal F, \mathbb P)$ is Polish and atomless.  Let $L_2(\Omega)$ be the set of square integrable random variables on $(\Omega, \mathcal F, \mathbb P)$, modulo the equivalence relation of almost sure equality. 
For any $\phi: \mathcal P_2(\R) \to \R$, we define $\widehat{\phi}:L_2(\Omega) \to \R$ by $\widehat{\phi}(X)= \phi( \mathcal L(X))$, where $\mathcal L(X)$ denotes the law of $X$. If $\widehat{\phi}:L_2(\Omega) \to \R$ is  Fréchet-differentiable,  there is a measurable function $\partial_\mu \phi (\mu): \R \to \R$, called L-derivative of $\phi$,  such that for every $X$ with distribution $\mu$,  $D \widehat{\phi}(X) = \partial_\mu \phi (\mu) (X)$.

\textbf{Linear functional derivative. }
Basically, it is nothing but the notion of differentiability we would use for $\phi: \mathcal M(\R) \to \R$ if it were defined on the whole ${\mathcal M}(\R)$, where $\mathcal M(\R)$ is the linear space of signed 
measures on $\R$. Note that a subset $K$ of $\mathcal P_2(\R)$ is said to be bounded if there is $M$ such that for every $\mu \in K$, $\int_\R |x|^2 \mathrm d\mu(x) \leq M$. 
A function $\phi: \mathcal P_2(\R) \to \R$ is said to have a linear functional derivative if there exists a function 
\begin{align*}
\frac{\delta \phi}{\delta m}: \mathcal P_2(\R) \times \R &\to \R \\
(m,v) &\mapsto \frac{\delta \phi}{\delta m} (m) (v),
\end{align*}
jointly continuous in $(m,v)$, such that for any bounded subset $K$ of $\mathcal P_2(\R)$, the function $v \mapsto\frac{\delta \phi}{\delta m}(m) (v)$ is at most of quadratic growth in $v$ uniformly in $m$ for $m \in K$, and such that for all $m,m' \in \mathcal P_2(\R)$, $\phi(m') - \phi (m)
= \int_0^1 \!\!\int_\R  \frac{\delta \phi}{\delta m} (\lambda m'+(1-\lambda) m) (v)  \; \mathrm d (m'-m) (v) \mathrm d\lambda$.
Note that $\frac{\delta \phi}{\delta m}$ is uniquely defined up to an additive constant only. 

\textbf{Link between both derivatives. }
\begin{prop}
\label{prop:link_derivatives}
Let $\phi: \mathcal P_2(\R) \to \R$ be L-differentiable on $\mathcal P_2(\R)$, such that the  Fréchet derivative of its lifted function $D \widehat{\phi}:L_2(\Omega) \to L_2(\Omega)$ is uniformly Lipschitz-continuous. Assume also that for each $\mu \in \mathcal P_2(\R)$, there is a version of $v \in \R \mapsto\partial_\mu \phi (\mu) (v)$ such that the map $(v,\mu) \in \R \times \mathcal P_2(\R) \mapsto\partial_\mu \phi (\mu) (v)$ is continuous. 

Then $\phi$ has a linear functional derivative and for every $\mu \in \mathcal P_2(\R)$, 
\begin{align*}
\partial_\mu \phi (\mu) (\cdot) = \partial_v \left\{ \frac{\delta\phi}{\delta m}(\mu) \right\}(\cdot). 
\end{align*}
\end{prop}

\subsection{Functions of probability measures on the torus}
\label{subsec:functions_proba_measures}

We use in this paper some well-known properties of the L-derivative, which are recalled in this paragraph. Moreover, since we work  in the particular case of probability measures on the torus, it leads naturally to the periodicity of the L-derivative, see Proposition~\ref{prop:periodicite de la derivee} and foll. Finally, we prove Proposition~\ref{prop:phi assumptions Pt phi}.

\begin{lemme}
\label{lemme:derivee de Lions}
Let $\phi:\mathcal P_2(\R)\to \R$ be a function satisfying the $\phi$-assumptions. 
Then there is a constant $C>0$ such that for every $\mu \in \mathcal P_2(\R)$, we can redefine $\partial_\mu \phi(\mu)(\cdot)$ on a $\mu$-negligible set in such a way that for every $v,v' \in \R$, 
\begin{align}
\label{Lipschitz derivee de Lions 1}
\left| \partial_\mu \phi (\mu) (v)- \partial_\mu \phi(\mu) (v') \right| \leq C|v-v'|,
\end{align}
and $(\mu,v) \mapsto \partial_\mu \phi (\mu) (v)$ is continuous at any point $(\mu,v)$ such that $v$ belongs to the support of~$\mu$. 
Furthermore, there is $C>0$ such that for every $\mu \in \mathcal P_2(\R)$, for every $v \in \R$, 
\begin{align}
\label{Lipschitz derivee de Lions 2}
\left| \partial_\mu \phi (\mu) (v) \right| \leq C(1+|v|) + C \int_\R |x| \mathrm d\mu (x). 
\end{align}
\end{lemme}

\begin{proof}
By~\cite[Proposition 5.36]{carmonadelarue18}, inequality~\eqref{Lipschitz derivee de Lions 1} is a consequence of assumption $(\phi 3)$. The proof of the  continuity of $(\mu,v) \mapsto \partial_\mu \phi (\mu) (v)$ at any point where $v$ belongs to the support of $\mu$ is given in~\cite[Corollary 5.38]{carmonadelarue18}. 
Moreover, it follows by~\eqref{Lipschitz derivee de Lions 1} that
\begin{align*}
\left| \partial_\mu \phi (\mu) (v)- \int_\R \partial_\mu \phi(\mu) (v') \mathrm d\mu(v') \right|
&\leq \int_\R \left| \partial_\mu \phi (\mu) (v)- \partial_\mu \phi(\mu) (v')  \right|\mathrm d\mu(v')  \\
&\leq  C \int_\R \left| v-v' \right|\mathrm d\mu(v') 
\leq C |v| + C \int_\R \left| v' \right|\mathrm d\mu(v') .
\end{align*}
By assumption $(\phi2)$, $|  \int_\R \partial_\mu \phi(\mu) (v') \mathrm d\mu(v') | \leq (\int_\R | \partial_\mu \phi(\mu) (v')|^2 \mathrm d\mu (v'))^{\frac{1}{2}} \leq C$. Thus we deduce inequality~\eqref{Lipschitz derivee de Lions 2}. 
\end{proof}

\begin{prop}
\label{prop:periodicite de la derivee}
Let $\phi:\mathcal P_2(\R)\to \R$ be a function satisfying the $\phi$-assumptions. 
Let $\mu \in \mathcal P_2(\R)$.
Then, up to redefining $\partial_\mu \phi (\mu)(\cdot)$ on a $\mu$-negligible set, the map $v \mapsto \partial_\mu \phi(\mu) (v)$ is $2\pi$-periodic. 
\end{prop}

\begin{proof}
Let $X \in L_2(\Omega)$ be a random variable with distribution $\mu$. 
For any $Y \in L_2(\Omega)$ and for any $\Z$-valued random variable $K$, we have
\begin{align*}
\frac{\mathrm d}{\mathrm d\eps}_{\vert \eps=0} \widehat{\phi} (X+2K\pi + \eps Y)
=\E{D \widehat{\phi}(X+2 K \pi)\; Y}=\E{\partial_\mu \phi(\mathcal L(X+2K \pi)) (X+2 K \pi) \;Y}.
\end{align*}
Moreover, for any $\eps$, $\widehat{\phi} (X+2K\pi + \eps Y)=\widehat{\phi} (X + \eps Y)$ because $\mathcal L(X+2K\pi + \eps Y)\sim \mathcal L (X+\eps Y)$. Therefore, 
\begin{align*}
\frac{\mathrm d}{\mathrm d\eps}_{\vert \eps=0} \widehat{\phi} (X+2K\pi + \eps Y)
=\frac{\mathrm d}{\mathrm d\eps}_{\vert \eps=0} \widehat{\phi} (X + \eps Y)
=\E{\partial_\mu \phi(\mathcal L(X)) (X)\; Y}.
\end{align*}
Thus for any $Y \in L_2(\Omega)$, $\E{\partial_\mu \phi(\mathcal L(X+2K \pi)) (X+2 K \pi) Y}=\E{\partial_\mu \phi(\mathcal L(X)) (X) Y}$. We deduce that for any $\Z$-valued random variable $K$, almost surely, 
\begin{align}
\label{egalite:+2Kpi}
\partial_\mu \phi(\mathcal L(X+2K \pi)) (X+2 K \pi)=\partial_\mu \phi(\mathcal L(X))(X).
\end{align}

\textbf{First, assume that the support of $\mu=\mathcal L(X)$ is equal to $\R$.}
For every $\delta\in (0,1)$, let $K^\delta$ be a random variable on $(\Omega,\mathcal F, \mathbb P)$  independent of $X$ with a Bernoulli distribution of parameter~$\delta$. Thus it follows from~\eqref{egalite:+2Kpi} that
\begin{align*}
1&= (1-\delta) \;\P{\partial_\mu \phi(\mathcal L(X+2K^\delta \pi)) (X)=\partial_\mu \phi(\mathcal L(X))(X)} \\
&\quad + \delta \;\P{\partial_\mu \phi(\mathcal L(X+2K^\delta \pi)) (X+2  \pi)=\partial_\mu \phi(\mathcal L(X))(X)}.
\end{align*}
We deduce that  $\P{\partial_\mu \phi(\mathcal L(X+2K^\delta \pi)) (X+2  \pi)=\partial_\mu \phi(\mathcal L(X))(X)}=1$ for any $\delta \in (0,1)$. 
Since the support of $\mathcal L(X)$ is equal to $\R$, it follows from Lemma~\ref{lemme:derivee de Lions} that $(\nu,x) \mapsto \partial_\mu \phi(\nu)(x)$ is continuous at $(\mu,x)$ for every $x\in \R$.  
Moreover $\mathcal L(X+2K^\delta \pi)$ tends in $L_2$-Wasserstein distance to $\mathcal L(X)=\mu$ when $\delta \to 0$. 
So, there exists an event $\widetilde{\Omega}$ of probability one such that for every $\omega \in \widetilde{\Omega}$ and every $\delta \in (0,1) \cap \Q$, $\partial_\mu \phi(\mathcal L(X+2K^\delta \pi)) (X(\omega)+2  \pi)=\partial_\mu \phi(\mu)(X(\omega))$. Thus for every $\omega \in \widetilde{\Omega}$, $\partial_\mu \phi(\mu) (X(\omega)+2  \pi)=\partial_\mu \phi(\mu)(X(\omega))$.
Since $\P{\widetilde{\Omega}}=1$ and the support of $\mu$ is~$\R$, we deduce that  $\partial_\mu \phi(\mu) (x+2  \pi)=\partial_\mu \phi(\mu)(x)$ holds with every $x$ in  a dense subset of $\R$. By continuity, of $\partial_\mu \phi(\mu)(\cdot)$, the last equality holds for every $x\in \R$. We deduce that $\partial_\mu \phi(\mu)(\cdot)$ is $2\pi$-periodic. 

\textbf{Then,  consider a general $\mu \in \mathcal P_2(\R)$.}
Let $Z$ be a random variable on  $(\Omega,\mathcal F, \mathbb P)$  independent of $X$ with normal distribution $\mathcal N(0,1)$ and  $(a_n)_{n \in \N}$ be a sequence such that for all $n \in \N$, $a_n \in (0,1)$ and $a_n \to_{n \to +\infty} 0$.   For every $n \in \N$,  the support of the distribution $\mathcal L(X+a_n Z)$ is equal to $\R$. 
Thus for every $n \in \N$, $v \mapsto \partial_\mu \phi(\mathcal L(X+a_n Z))(v)$ is $2\pi$-periodic. 

By~\eqref{Lipschitz derivee de Lions 1}, the sequence of continuous functions $( \partial_\mu \phi(\mathcal L(X+a_n Z)))_{n \in \N}$ is equicontinuous. Furthermore, \eqref{Lipschitz derivee de Lions 2} implies that for every $v\in [0,2\pi]$,
\begin{align*}
\left| \partial_\mu \phi (\mathcal L(X+a_n Z)) (v) \right| \leq C(1+|v|) + C\E{|X+a_n Z|}.
\end{align*}
Since $(a_n)_{n \in \N}$ is bounded by $1$ and $X \in L_2(\Omega)$, there exists $C>0$ such that for every $n\in \N$ and for every $v\in [0,2\pi]$
\begin{align}
\label{uniformement borne}
| \partial_\mu \phi (\mathcal L(X+a_n Z)) (v)| \leq C(1+|v|) \leq C(1+2\pi).
\end{align}
Recall that $v \mapsto \partial_\mu \phi(\mathcal L(X+a_n Z))(v)$ is $2\pi$-periodic for every $n \in \N$. 
Thus the sequence $( \partial_\mu \phi(\mathcal L(X+a_n Z)))_{n \in \N}$ is uniformly bounded on $\R$.
By Arzela-Ascoli's Theorem, up to extracting a subsequence, $( \partial_\mu \phi(\mathcal L(X+a_n Z)))_{n \in \N}$ converges uniformly to a limit $u:\R \to \R$. In particular, $u$ is a $2\pi$-periodic function.

Moreover, we prove that the following quantity tends to zero. Let $Y \in L_2(\Omega)$. 
\begin{multline*}
\Big| \E{\partial_\mu \phi(\mathcal L(X+a_n Z)) (X+a_n Z) Y } - \E{u(X) Y} \Big|  \\ \begin{aligned}
&\leq \Big| \E{\partial_\mu \phi(\mathcal L(X+a_n Z)) (X+a_n Z) Y } -  \E{\partial_\mu \phi(\mathcal L(X+a_n Z)) (X) Y } \Big|\\
&\quad + \Big| \E{\partial_\mu \phi(\mathcal L(X+a_n Z)) (X) Y } - \E{u(X) Y} \Big| \\
&\leq C a_n \E{|Z Y|} + \Big| \E{\partial_\mu \phi(\mathcal L(X+a_n Z)) (X) Y } - \E{u(X) Y} \Big|,
\end{aligned}
\end{multline*}
by inequality~\eqref{Lipschitz derivee de Lions 1}. 
Since $ZY$ is integrable, $C a_n \E{|Z Y|} \to_{n \to +\infty} 0$. 
Moreover, let us show that $\left| \E{\partial_\mu \phi(\mathcal L(X+a_n Z)) (X) Y } - \E{u(X) Y} \right|  \to_{n \to +\infty} 0$. Remark that $( \partial_\mu \phi(\mathcal L(X+a_n Z)))_{n \in \N}$ converges uniformly to $u$, hence it converges pointwise to $u$. Moreover, we have a uniform integrability property. Indeed, by $2\pi$-periodicity of $\partial_\mu \phi(\mathcal L(X+a_n Z))$ and by~\eqref{uniformement borne}
\begin{align*}
\E{ |  \partial_\mu \phi(\mathcal L(X+a_n Z)) (X) Y  |^{3/2}}
&\leq \E{ |  \partial_\mu \phi(\mathcal L(X+a_n Z)) (X)  |^6}^{1/4} \E{|Y|^2}^{\frac{3}{4}} \\
&=\E{ |  \partial_\mu \phi(\mathcal L(X+a_n Z)) (\accolade{X})  |^6}^{1/4} \E{|Y|^2}^{\frac{3}{4}} \leq C \E{|Y|^2}^{\frac{3}{4}}. 
\end{align*}
Since $Y$ is square integrable, $\sup_{n \in \N} \E{ |  \partial_\mu \phi(\mathcal L(X+a_n Z)) (X) Y  |^{3/2}} <+\infty$. Thus the sequence $(\partial_\mu \phi(\mathcal L(X+a_n Z)) (X) Y)_{n\in \N}$ is uniformly integrable. By Fatou's Lemma, if follows that $\E{|u(X) Y|^{3/2}} <+\infty$, thus we conclude that $ \E{\partial_\mu \phi(\mathcal L(X+a_n Z)) (X) Y } \to_{n \to + \infty} \E{u(X) Y}$. Therefore, 
\begin{align*}
\E{\partial_\mu \phi(\mathcal L(X+a_n Z)) (X+a_n Z) Y } \to_{n \to + \infty} \E{u(X) Y}.
\end{align*}
On the other hand, 
\begin{align*}
\E{\partial_\mu \phi(\mathcal L(X+a_n Z)) (X+a_n Z) Y }=\E{D \widehat{\phi}(X+a_nZ) Y} \underset{n \to +\infty}{\longrightarrow} \E{D \widehat{\phi}(X) Y}
\end{align*}
because $D\widehat{\phi}$ is Lipschitz by assumption $(\phi3)$. We deduce that $\E{u(X) Y}=\E{D \widehat{\phi}(X) Y}$ for every $Y \in L_2(\Omega)$, hence almost surely, $u(X)=\partial_\mu \phi (\mu)(X)$. Recall that $u$ is continuous and $2\pi$-periodic. Therefore, up to redefining $ \partial_\mu \phi (\mu)(\cdot)$ on a $\mu$-negligible set, $v \mapsto \partial_\mu \phi (\mu)(v)$ is continuous and $2\pi$-periodic. 
\end{proof}

 For every $\mu \in \mathcal P_2(\R)$, we define $\widetilde{\mu}$ the measure satisfying $\widetilde{\mu}(A)=\mu(A+2\pi \Z)$ for every $A \in \mathcal B[0,2\pi]$. Clearly, $\widetilde{\mu}$ belongs to $\mathcal P(\tor)$ and $\widetilde{\mu}=\widetilde{\nu}$ for every $\mu \sim \nu$ in the sense of Definition~\ref{defin:tor_stable}.

\begin{coro}
Let $\phi:\mathcal P_2(\R)\to \R$ be a function satisfying the $\phi$-assumptions. 
Let $\mu \in \mathcal P_2(\R)$ and assume that $\widetilde{\mu}$  has a density belonging to $\mathcal P^+$ in the sense of Definition~\ref{defin_density}. 
Then there is a unique $2 \pi$-periodic  and continuous version of $\partial_\mu \phi (\mu) (\cdot)$. Furthermore, for every $v\in [0,2\pi]$, $\partial_\mu \phi (\mu) (v)=\partial_\mu \phi (\widetilde{\mu}) (v)$.
\label{coro:unique version periodique}
\end{coro}

\begin{proof}
Let $X\in L_2(\Omega)$ with distribution $\mu$. 
Then the law of $\accolade{X}$ is $\widetilde{\mu}$, seen as an element of $\mathcal P_2 (\R)$ with support included in $[0,2\pi]$. Since the density of $\widetilde{\mu}$ belongs to $\mathcal P^+$, the support of $\widetilde{\mu}$ is equal to $[0,2\pi]$.

Furthermore, by equality~\eqref{egalite:+2Kpi} applied to $K=\frac{1}{2\pi} (\accolade{X}-X)$, \textit{i.e} $X+2K \pi =\accolade{X}$,  the following equality holds almost surely: $\partial_\mu \phi (\mathcal L(\accolade{X})) (\accolade{X})=\partial_\mu \phi (\mathcal L(X)) (X)$.
By Proposition~\ref{prop:periodicite de la derivee}, $\partial_\mu \phi (\mathcal L(X)) (\cdot)$ is $2\pi$-periodic, so $\partial_\mu \phi (\mathcal L(X)) (X)=\partial_\mu \phi (\mathcal L(X)) (\accolade{X})$.
Since the support of $\mathcal L(\accolade{X})$ is equal to $[0,2\pi]$, we deduce that for every $v\in [0,2\pi]$, $\partial_\mu \phi (\mathcal L(X)) (v)=\partial_\mu \phi (\mathcal L(\accolade{X})) (v)$.
This shows that  there is a unique $2\pi$-periodic and continuous version $\partial_\mu \phi (\mathcal L(X)) (\cdot)$. 
\end{proof}

In the light of Proposition~\ref{prop:link_derivatives}, we prove that the linear functional derivative is also $2\pi$-periodic:

\begin{prop}
\label{prop:integrale nulle}
Let $\phi:\mathcal P_2(\R)\to \R$ be a function satisfying the $\phi$-assumptions. Let $\mu \in \mathcal P_2(\R)$ be such that $\widetilde{\mu}$ has a density belonging to~$\mathcal P^+$ in the sense of Definition~\ref{defin_density}. 
Then $\int_\tor \partial_\mu \phi (\mu) (v) \mathrm dv=0$.  In other words, $v \mapsto \frac{\delta \phi}{\delta m}(\mu) (v)$ is $2\pi$-periodic.
\end{prop}

\begin{proof}
By Corollary~\ref{coro:unique version periodique}, it is sufficient to prove that $\int_\tor \partial_\mu \phi (\widetilde{\mu}) (v) \mathrm dv=0$. 
Let $Y_0$ be a random variable on $(\Omega, \mathcal F, \mathbb P)$ with distribution equal to $\widetilde{\mu}$. Let $p:\R \to \R$ denote its density, extended by $2\pi$-continuity. By assumption, $p(v)>0$ for every $v \in [0,2\pi]$, hence it holds for every $v \in \R$.  

Define the following ordinary differential equation:
\begin{align*}
\dot{Y}_t = \frac{1}{p(Y_t)},
\end{align*}
with initial condition $Y_0$. 
Denoting by $F:=x \mapsto \int_0^x p(v) \mathrm dv$ and by $g=F^{-1}$ respectively the c.d.f.\! and the quantile function associated to $p$, we have $\frac{\mathrm d}{\mathrm dt} F(Y_t)=1$. Thus for every $t\geq 0$, $F(Y_t)=F(Y_0)+t$ and $Y_t=g(F(Y_0)+t)=g_t(F(Y_0))$, where $g_t(\cdot)=g(\cdot+t)$. 

Fix $t\geq 0$. Since $F(Y_0)$ has a uniform distribution on $[0,1]$, $Y_t=g_t(F(Y_0))$ implies that $g_t$ is the quantile function of the random variable $Y_t$. 
According to Definition~\ref{defin_g}, $g_t \sim g$, thus we deduce that the law of $\accolade{Y_t}$ is  $\widetilde{\mu}$. Since $\phi$ is $\tor$-stable, $\widehat{\phi}(Y_t)=\widehat{\phi}(\accolade{Y_t})=\widehat{\phi}(Y_0)$ for every $t \geq 0$. Thus $\frac{\mathrm d}{\mathrm dt} _{\vert t=0}\widehat{\phi}(Y_t) = 0$. Thus
\begin{align*}
0&=\frac{\mathrm d}{\mathrm dt} _{\vert t=0}\widehat{\phi}(Y_t) = \E{D \widehat{\phi}(Y_0) \dot{Y}_0}=\E{D \widehat{\phi}(Y_0) \frac{1}{p(Y_0)}} \\
&=\int_\R \partial_\mu \phi (\widetilde{\mu}) (v) \frac{1}{p(v)}\mathrm d\widetilde{\mu} (v) 
=\int_0^{2\pi} \partial_\mu \phi (\widetilde{\mu}) (v) \frac{p(v) }{p(v)}\mathrm dv
=\int_0^{2\pi} \partial_\mu \phi (\widetilde{\mu}) (v) \mathrm dv,
\end{align*}
since $p$ is the density of the measure $\widetilde{\mu}$. The statement of the proposition follows. 
\end{proof}

Let us finally prove Proposition~\ref{prop:phi assumptions Pt phi}. 

\begin{proof}[Proof (Proposition~\ref{prop:phi assumptions Pt phi})]
Fix $t \in [0,T]$. Let us check the three assumptions of Definition~\ref{def:phi assumptions}.

\textbf{Assumption $(\phi1)$:}
We start by proving that $P_t\phi$ is $\tor$-stable.
Let $\mu \sim \nu$ in the sense of Definition~\ref{defin:tor_stable} and $X, Y \in L_2[0,1]$ satisfy $\mathcal L_{[0,1]} (X)= \mu$ and $\mathcal L_{[0,1]} (Y)= \nu$. Recall that $\accolade{x}$ denotes the unique number in~$[0,2\pi)$ such that $x- \accolade{x} \in 2\pi \Z$. Since $\mathcal L_{[0,1]} (X) \sim \mathcal L_{[0,1]} (Y)$, we have $\mathcal L_{[0,1]} (\accolade{X}) = \mathcal L_{[0,1]} (\accolade{Y})$. 
By Proposition~\ref{prop:stabilite}, it follows that  $\mathbb P^W $-almost surely,
the laws of $\accolade{Z_t^X}$ and of $\accolade{Z_t^Y}$ under $[0,1] \times \Omega^\beta$ are equal. 
Therefore, $\mathbb P^W $-almost surely, $\widehat{\phi}(\accolade{Z_t^X})=\widehat{\phi}(\accolade{Z_t^Y})$.
Since $\phi$ is $\tor$-stable, it implies that $\mathbb P^W $-almost surely, $\widehat{\phi}(Z_t^X)=\widehat{\phi}(Z_t^Y)$.
By Definition~\ref{defin:semi-group P_t}, $P_t\phi(\mu)=\widehat{P_t\phi}(X)=\mathbb E^W [ \widehat{\phi}(Z_t^X)] =\mathbb E^W [ \widehat{\phi}(Z_t^Y)]=\widehat{P_t\phi}(Y)=P_t\phi(\nu)$. Thus $P_t \phi$ is $\tor$-stable. 

By Definition~\ref{defin:semi-group P_t}, it is clear that $P_t\phi$ is bounded on $\mathcal P_2(\R)$, because $\phi$ is bounded. 
Furthermore, $P_t\phi$ is continuous on $\mathcal P_2(\R)$, and  even Lipschitz-continuous. Indeed, let $\mu, \nu \in \mathcal P_2(\R)$ and $X, Y \in L_2[0,1]$ be the quantile functions respectively associated with $\mu$ and $\nu$: $\mu =\mathcal L_{[0,1]} (X)$ and $\nu=\mathcal L_{[0,1]} (Y)$; in other words, $X$ (resp. $Y$) is the increasing rearrangement of $\mu$ (resp. $\nu)$. A classical result in optimal transportation (see e.g. \cite[Theorem 2.18]{villani03}) states that $(X,Y)$ realises the optimal coupling in the definition of the $L_2$-Wasserstein distance: $W_2(\mu,\nu)^2= \int_0^1 |X(u)-Y(u)|^2 \mathrm du$. 

By Remark~\ref{rem:phi2_implique_lipschitz}, $\widehat{\phi}$ is Lipschitz-continuous, thus:
\begin{align*}
|P_t\phi(\mu) - P_t \phi (\nu) |
=|\widehat{P_t\phi}(X) - \widehat{P_t\phi}(Y) |
&\leq \mathbb E^W \left[ |\widehat{\phi} (Z_t^X) -  \widehat{\phi} (Z_t^Y)|\right] \\
&\leq   \|\widehat{\phi}\|_{\operatorname{Lip}} \; \mathbb E^W \left[ \|Z_t^X- Z_t^Y\|_{L_2([0,1] \times \Omega^\beta)}\right]\\
&\leq   \|\widehat{\phi}\|_{\operatorname{Lip}} \; \Ewb{\int_0^1 \big|Z_t^{X(u)}-Z_t^{Y(u)} \big|^2 \mathrm du}^{\frac{1}{2}}.
\end{align*}
By~\eqref{différence Ztx Zty}, we have
\begin{align*}
\Ewb{\int_0^1 \big|Z_t^{X(u)}-Z_t^{Y(u)} \big|^2 \mathrm du}
&\leq C_2  \int_0^1 |X(u)-Y(u) |^2 \mathrm du
=W_2(\mu,\nu)^2.
\end{align*}
Thus $|P_t\phi(\mu) - P_t \phi (\nu) | \leq C W_2(\mu,\nu)$; in particular, $P_t\phi$ is continuous. 

\textbf{Assumption $(\phi2)$ and proof of equality~\eqref{frechet_derivative}:}
Let us prove that $P_t \phi$ is $L$-differentiable.  
Let $\mu, \nu \in \mathcal P_2(\R)$ and  $X, Y \in L_2[0,1]$ such that $\mathcal L_{[0,1]} (X)= \mu$ and $\mathcal L_{[0,1]} (Y)= \nu$. We prove that the Fréchet derivative of $\widehat{P_t\phi}$ at point $X$ is given by
\begin{align*}
\partial_\mu P_t\phi (\mu) (X) \cdot Y= D \widehat{P_t \phi}(X) \cdot Y
&= \int_0^1 \Ewb{D \widehat{\phi}(Z_t^{X})_u \; \partial_x Z_t^{X(u)} } Y(u)  \mathrm du,
\end{align*}
which implies~\eqref{frechet_derivative} using identities~\eqref{egalite Zx} and~\eqref{liens entre les dérivées}.

Assume that $\|Y\|_{L_2} \leq 1$.  We compute
\begin{align*}
\widehat{P_t \phi}(X+ Y)-\widehat{P_t \phi}(X)
&= \mathbb E^W \left[\widehat{\phi}(Z_t^{X+ Y})-\widehat{\phi}(Z_t^X)   \right] 
= \mathbb E^W \left[ \int_0^1  \frac{\mathrm d}{\mathrm d\lambda} \widehat{\phi}(Z_t^{X+\lambda  Y} )  \mathrm d\lambda  \right] \\
&= \mathbb E^W \left[ \int_0^1  D \widehat{\phi} (Z_t^{X+\lambda  Y} ) \cdot \partial_x   Z_t^{X+\lambda  Y} Y     \mathrm d\lambda  \right] \\
&= \mathbb E^W \left[ \int_0^1 \mathbb E^\beta \left[\int_0^1 D \widehat{\phi} (Z_t^{X+\lambda  Y})_u \;\partial_x Z_t^{(X+\lambda  Y)(u)} Y(u)   \mathrm du \right]   \mathrm d\lambda \right],
\end{align*}
since $\widehat{\phi}$ is Fréchet-differentiable. 
Therefore
\begin{align*}
\left|\widehat{P_t \phi}(X + Y)-\widehat{P_t \phi}(X)
-\int_0^1 \Ewb{D \widehat{\phi} (Z_t^X)_u \; \partial_x Z_t^{X(u)} } Y(u)  \mathrm du
  \right| \leq |D_1| + |D_2|, 
\end{align*}
where
\begin{align*}
D_1&:= \Ewb{ \int_0^1 \!\! \int_0^1 \left( D \widehat{\phi} (Z_t^{X+\lambda  Y})_u  -D \widehat{\phi} (Z_t^X)_u \right) \partial_x Z_t^{X(u)} Y(u)  \;\mathrm d\lambda \mathrm du} ; \\
D_2&:=  \Ewb{\int_0^1 \!\! \int_0^1  D \widehat{\phi} (Z_t^{X+\lambda  Y})_u  \left(\partial_x Z_t^{(X+\lambda  Y)(u)} -\partial_x Z_t^{X(u)} \right)Y(u) \;\mathrm d\lambda \mathrm du} .
\end{align*}

Let us start by estimating $D_1$. By Cauchy-Schwarz inequality, 
\begin{align*}
|D_1| 
&\leq \Ewb{ \int_0^1 \!\! \int_0^1 \left| D \widehat{\phi} (Z_t^{X+\lambda  Y})_u  -D \widehat{\phi} (Z_t^X)_u \right|^2 \mathrm d\lambda \mathrm du}^\frac12
\Ewb{ \int_0^1   |\partial_x Z_t^{X(u)} Y(u)|^2   \mathrm du}^\frac12 \!\!.
\end{align*}
On the one hand, by assumption $(\phi3)$ and by~\eqref{différence Ztx Zty}, we have (with constants modified from a line to the next):
\begin{multline*}
\Ewb{ \int_0^1 \!\! \int_0^1 \left| D \widehat{\phi} (Z_t^{X+\lambda  Y})_u  -D \widehat{\phi} (Z_t^X)_u\right|^2 \mathrm d\lambda \mathrm du} \\
\begin{aligned}
&\leq C \; \Ewb{ \int_0^1 \!\! \int_0^1 \left|Z_t^{(X+\lambda  Y)(u)} - Z_t^{X(u)} \right|^2 \mathrm d\lambda \mathrm du} \\
&\leq C \int_0^1 \!\! \int_0^1  \lambda^2  |Y(u)|^2 \mathrm d\lambda \mathrm du 
\leq C \|Y \|_{L_2}^2. 
\end{aligned}
\end{multline*}
On the other hand
\begin{align*}
\Ewb{ \int_0^1  |\partial_x Z_t^{X(u)} Y(u)|^2  \mathrm du}
=\int_0^1 \Ewb{  |\partial_x Z_t^{X(u)}|^2 } |Y(u)|^2   \mathrm du.
\end{align*}
By~\eqref{borne sur pZtx}, there is $C>0$ such that for every $u \in [0,1]$, $\mathbb E^W \mathbb E^\beta \left[|\partial_x Z_t^{X(u)}|^2\right]\leq C$. Thus $\Ewb{ \int_0^1  |\partial_x Z_t^{X(u)} Y(u)|^2  \mathrm du}
\leq C  \|Y \|_{L_2}^2$.
Finally, we get
\begin{align}
\label{D1}
|D_1| \leq C  \|Y \|_{L_2}^2. 
\end{align}

Moreover,  compute $D_2$. 
By Cauchy-Schwarz inequality, $|D_2| \leq |D_{2,1}|^{\frac{1}{2}}\cdot |D_{2,2}|^{\frac{1}{2}}$
where
\begin{align*}
D_{2,1}&:= \Ewb{\int_0^1 \!\! \int_0^1 \Big| D \widehat{\phi} (Z_t^{X+\lambda  Y})_u \; Y(u)\Big|^2  \mathrm d\lambda \mathrm du}; \\
D_{2,2}&:= \Ewb{\int_0^1 \!\! \int_0^1  \left|\partial_x Z_t^{(X+\lambda  Y)(u)} -\partial_x Z_t^{X(u)} \right|^2\mathrm d\lambda \mathrm du}.
\end{align*}
On the one hand, 
\begin{align*}
D_{2,1}
= \int_0^1 \!\! \int_0^1 \Ewb{\Big| D \widehat{\phi} (Z_t^{X+\lambda  Y})_u \Big|^2 } |Y(u)|^2  \mathrm d\lambda \mathrm du.
\end{align*}
Recall that $\phi$ is $\tor$-stable. It follows that  for any random variables $U,V \in L_2([0,1] \times \Omega^\beta)$, $D\widehat{\phi}(U) \cdot V = \lim_{\eps \to 0} \frac{\widehat{\phi}(U+\eps V) -\widehat{\phi}(U)}{\eps}=\lim_{\eps \to 0} \frac{\widehat{\phi}(\accolade{U} +\eps V) -\widehat{\phi}(\accolade{U})}{\eps}=D\widehat{\phi}(\accolade{U}) \cdot V$. Hence for every $U \in L_2([0,1] \times \Omega^\beta)$, $D\widehat{\phi}(U)=D\widehat{\phi}(\accolade{U})$.

Let us denote by $\xi:=\mathcal L_{[0,1] \times \Omega^\beta} (\accolade{Z_t^{X+\lambda Y}})$. Then
\begin{align*}
\Ewb{\Big| D \widehat{\phi} (\accolade{ Z_t^{X+\lambda Y}})_u \Big|^2}
=\Ewb{\Big| \partial_\mu \phi (\xi) (\accolade{ Z_t^{(X+\lambda Y)(u)}}) \Big|^2 } 
\leq \mathbb E^W \left[ \sup_{v \in \R}  | \partial_\mu \phi (\xi) (v)|^2 \right].
\end{align*}
By Proposition~\ref{prop:periodicite de la derivee}, $v \mapsto \partial_\mu \phi (\xi)(v)$ is $2\pi$-periodic. By inequality~\eqref{Lipschitz derivee de Lions 2}, it follows that
\begin{align*}
 \sup_{v \in \R}  | \partial_\mu \phi (\xi) (v)| 
= \sup_{v \in [0,2\pi]}  | \partial_\mu \phi (\xi) (v)| 
&\leq  C(1+2\pi) + C \int_\R |x| \mathrm d\xi (x) \\
&= C(1+2\pi) + C \mathbb E^\beta \left[ \int_0^1 |\accolade{Z_t^{(X+\lambda Y)(u)}}| \mathrm du \right] \leq  C, 
\end{align*}
since $\accolade{Z_t^{(X+\lambda Y)(u)}}$ takes values in $[0,2\pi)$. 
Thus $\mathbb E^W \left[ \sup_{v \in \R}  | \partial_\mu \phi (\xi) (v)|^2 \right] \leq C$, where $C$ is independent of $X$, $\lambda$ and $Y$. 
We deduce that $D_{2,1} \leq  \int_0^1 \!\! \int_0^1 C |Y(u)|^2  \mathrm d\lambda \mathrm du \leq  C \|Y\|_{L_2}^2$.

On the other hand, since $f$ is of order $\alpha> \frac{5}{2}$, inequality~\eqref{différence pZtx pZty} holds with $\theta=1$. Thus $D_{2,2} \leq  C \int_0^1 \!\! \int_0^1 \lambda^2  |Y(u)|^2 \mathrm d\lambda \mathrm du \leq C  \|Y\|_{L_2}^2$. 
We finally obtain that $|D_2| \leq C  \|Y \|_{L_2}^2$. 
It follows from the latter inequality and from~\eqref{D1} that for every $\|Y\|_{L_2} \leq 1$, 
\begin{align*}
\left|\widehat{P_t \phi}(X+ Y)-\widehat{P_t \phi}(X)
-\int_0^1 \Ewb{D \widehat{\phi} (Z_t^X)_u \; \partial_x Z_t^{X(u)} } Y(u)  \mathrm du
  \right|  \leq C   \|Y \|_{L_2}^2. 
\end{align*}
Thus $\widehat{P_t \phi}$ is Fréchet-differentiable at point $X$ and the derivative is given by~\eqref{frechet_derivative}. 

Moreover,  prove that 
\begin{align*}
\sup_{\mu \in \mathcal P_2(\R)} \int_\R |\partial_\mu P_t\phi (\mu) (x) |^2 \mathrm d\mu(x)
<+\infty.
\end{align*}
Observe that $\sup_{\mu \in \mathcal P_2(\R)} \int_\R |\partial_\mu P_t\phi (\mu) (x) |^2 \mathrm d\mu(x)
=\sup_{X \in L_2[0,1]} \int_0^1 |D \widehat{ P_t\phi} (X)_u |^2 \mathrm du$. 
Let us apply~\eqref{frechet_derivative} with $Y=D \widehat{ P_t\phi} (X)$. We obtain
\begin{align}
\label{eq:23d}
\int_0^1 |D \widehat{ P_t\phi} (X)_u |^2 \mathrm du 
&=\Ewb{\int_0^1 D \widehat{\phi} (Z_t^X)_u \; \partial_x Z_t^{X(u)} D \widehat{ P_t\phi} (X)_u \mathrm du} \notag \\
&\leq  \Ewb{\int_0^1 |D\widehat{\phi}(Z_t^X)_u|^2  \mathrm du}^\frac12
 \Ewb{\int_0^1 |\partial_x Z_t^{X(u)} D \widehat{ P_t\phi} (X)_u|^2  \mathrm du}^\frac12.
\end{align}
By~\eqref{borne sur pZtx}, 
\begin{align}
\label{eq:23e}
 \Ewb{\int_0^1 |\partial_x Z_t^{X(u)} D \widehat{ P_t\phi} (X)_u|^2  \mathrm du}
& =\int_0^1 \Ewb{|\partial_x Z_t^{X(u)}|^2}  |D \widehat{ P_t\phi} (X)_u |^2  \mathrm du \notag\\
&\leq C \int_0^1  |D \widehat{ P_t\phi} (X)_u|^2  \mathrm du.
\end{align}
It follows from~\eqref{eq:23d} and~\eqref{eq:23e} that
\begin{align*}
 \int_0^1 |D \widehat{ P_t\phi} (X)_u |^2 \mathrm du 
&\leq C\; \Ewb{\int_0^1 |D\widehat{\phi}(Z_t^X)_u|^2  \mathrm du} =C \;\mathbb E^W \left[ \int_\R |\partial_\mu \phi (\xi) (x)|^2  \mathrm d\xi(x) \right],
\end{align*}
where $\xi=\mathcal L_{[0,1] \times \Omega^\beta} (Z_t^{X(u)})$. The last term is bounded by a constant independent of $\xi$ because by assumption $(\phi 2)$, $\sup_{\mu \in \mathcal P_2(\R)} \int_\R |\partial_\mu \phi (\mu) (x) |^2 \mathrm d\mu(x)<+\infty$. 

\textbf{Assumption $(\phi3)$:}
Let us prove that for every $X_1, X_2, Y \in L_2[0,1]$, 
\begin{align}
\label{eq:23f}
|D\widehat{P_t \phi}(X_1) \cdot Y - D\widehat{P_t \phi}(X_2) \cdot Y|
\leq C \|X_1-X_2\|_{L_2[0,1]} \|Y\|_{L_2[0,1]}.
\end{align}
By  formula~\eqref{frechet_derivative}, 
 $|D\widehat{P_t \phi}(X_1) \cdot Y - D\widehat{P_t \phi}(X_2) \cdot Y| \leq |D_3| +|D_4|$, where
\begin{align*}
D_3&:= \Ewb{ \int_0^1 ( D\widehat{\phi}(Z_t^{X_1})_u-D\widehat{\phi}(Z_t^{X_2})_u) \partial_x Z_t^{X_1(u)} Y(u) \mathrm du }   ; \\
D_4&:= \Ewb{ \int_0^1  D\widehat{\phi}(Z_t^{X_2})_u (\partial_x Z_t^{X_1(u)}-\partial_x Z_t^{X_2(u)}) Y(u) \mathrm du }  .
\end{align*}
Up to replacing $X$ and $X+\lambda Y$  by $X_1$ and $X_2$, $D_3$ and $D_4$ are equivalent to $D_1$ and $D_2$. Thus we get by the same computations as for $D_1$ and $D_2$:
\begin{align*}
|D_3| &\leq C \| X_1-X_2 \|_{L_2[0,1]} \| Y\|_{L_2[0,1]} \\
|D_4| &\leq C \| X_1-X_2 \|_{L_2[0,1]} \| Y\|_{L_2[0,1]}. 
\end{align*}
This completes the proofs of~\eqref{eq:23f} and of the proposition. 
\end{proof}

\nocite{KSI}
\nocite{KSII}
\nocite{KSIII}
\nocite{norris86}
\nocite{nualart06}

\bibliographystyle{alpha}

\bibliography{ref_victor}

\end{document}